\documentclass[11pt, oneside, reqno]{amsart}
\usepackage[shortlabels, inline]{enumitem}
\usepackage{mathtools, mathrsfs, xargs, amssymb}
\usepackage[T1]{fontenc}
\usepackage{libertine}
\usepackage[top=1.2in, bottom=1.2in, left=1.2in, right=1.2in]{geometry}
\usepackage{stmaryrd}

\usepackage[dvipsnames]{xcolor}
\usepackage[
  colorlinks=true,
  linkcolor=Maroon,
  citecolor=Maroon,
  urlcolor=Maroon
]{hyperref}

\newcommand{\fin}{
  \bibliographystyle{amsalpha} 
  \bibliography{Zit}
  \end{document}
}

\makeatletter
\renewcommand\subsubsection{\@startsection{subsubsection}{3}{\z@}%
  {.2\linespacing\@plus.2\linespacing}{-.5em}{\normalfont\itshape}}
\makeatother

\numberwithin{equation}{section}
\theoremstyle{plain}
\newtheorem{theorem}{Theorem}[section]
\newtheorem{proposition}[theorem]{Proposition}
\newtheorem{lemma}[theorem]{Lemma}
\newtheorem{corollary}[theorem]{Corollary}
\newtheorem{assumption}[theorem]{Assumption}
\theoremstyle{definition}
\newtheorem{definition}[theorem]{Definition}
\newtheorem{example}[theorem]{Example}
\theoremstyle{remark}
\newtheorem{remark}[theorem]{Remark}

\setlist[enumerate,1]{label=(\arabic*), ref=\arabic*}
\setlist[enumerate]{leftmargin=*}

\usepackage{xargs} 
\usepackage{mathtools}

\usepackage{dashrule}

\providecommand{\alias}{}
\renewcommand{\alias}[1]{\providecommand{#1}{}\renewcommand{#1}}

\DeclarePairedDelimiter\ab{\langle}{\rangle} 
\DeclarePairedDelimiter\abs{\lvert}{\rvert}   
\DeclarePairedDelimiter\norm{\lVert}{\rVert}  
\DeclarePairedDelimiter\bkt{[}{]}             
\DeclarePairedDelimiter\brc{\{}{\}}           
\DeclarePairedDelimiter\prn{(}{)}             

\providecommand\giv{}

\alias\given\giv 

\DeclarePairedDelimiterXPP\pp[1]{\mathbb{P}}[]{}{
   \renewcommand\giv{\nonscript\:\delimsize\vert\nonscript\:\mathopen{}}
   #1}
\DeclarePairedDelimiterXPP\ppup[2]{\mathbb{P}^{#1}}[]{}{
   \renewcommand\giv{\nonscript\:\delimsize\vert\nonscript\:\mathopen{}}
   #2}
\DeclarePairedDelimiterXPP\ppdown[2]{\mathbb{P}_{#1}}[]{}{
   \renewcommand\giv{\nonscript\:\delimsize\vert\nonscript\:\mathopen{}}
   #2}
   \DeclarePairedDelimiterXPP\ppupdown[3]{\mathbb{P}^{#1}_{#2}}[]{}{
   \renewcommand\giv{\nonscript\:\delimsize\vert\nonscript\:\mathopen{}}
   #3}
   
\DeclarePairedDelimiterXPP\ee[1]{\mathbb{E}}[]{}{
   \renewcommand\giv{\nonscript\:\delimsize\vert\nonscript\:\mathopen{}}
   #1}

\DeclarePairedDelimiterXPP\var[1]{\mathrm{Var}}[]{}{
   \renewcommand\giv{\nonscript\:\delimsize\vert\nonscript\:\mathopen{}}
   #1}

 \DeclarePairedDelimiterXPP\cov[1]{\mathrm{Cov}}[]{}{
   \renewcommand\giv{\nonscript\:\delimsize\vert\nonscript\:\mathopen{}}
   #1}

\DeclarePairedDelimiterXPP\eeup[2]{\mathbb{E}^{#1}}[]{}{
   \renewcommand\giv{\nonscript\:\delimsize\vert\nonscript\:\mathopen{}}
   #2}

\DeclarePairedDelimiterXPP\eedown[2]{\mathbb{E}_{#1}}[]{}{
   \renewcommand\giv{\nonscript\:\delimsize\vert\nonscript\:\mathopen{}}
   #2}

   \DeclarePairedDelimiterXPP\eeupdown[3]{\mathbb{E}^{#1}_{#2}}[]{}{
   \renewcommand\giv{\nonscript\:\delimsize\vert\nonscript\:\mathopen{}}
   #3}
\DeclarePairedDelimiterXPP\eeud[3]{\mathbb{E}^{#1}_{#2}}[]{}{
   \renewcommand\giv{\nonscript\:\delimsize\vert\nonscript\:\mathopen{}}
   #3}

\let\preexp\exp
\let\exp\relax

\DeclarePairedDelimiterXPP\exp[1]{\preexp}(){}{#1}

\newcommand{\Exp}[1]{e^{#1}}

\newcommandx{\prf}[2][2=t\in\bkt{0,T}]{ \set{#1}_{#2}}
\newcommandx{\prfi}[1]{\set{#1}_{t \in \sp}}

\newcommandx{\seq}[2][2=n\in\N]{\brc*{#1}_{#2}}
\newcommandx{\sseq}[2][2=n\in\N]{\{#1\}_{#2}}

\newcommand{\seqz}[1]{\seq{#1}[n\in\N_0]}

\newcommandx{\seqk}[1]{\seq{#1}[k\in\N]}
\newcommandx{\seqkz}[1]{\seq{#1}[k\in\N_0]}

\newcommandx{\seqj}[1]{\seq{#1}[j\in\N]}
\newcommandx{\seqjz}[1]{\seq{#1}[j\in\N_0]}

\newcommandx{\seqi}[1]{\seq{#1}[i\in\N]}
\newcommandx{\seqiz}[1]{\seq{#1}[i\in\N_0]}

\newcommandx{\seqm}[1]{\seq{#1}[m\in\N]}
\newcommandx{\seqmz}[1]{\seq{#1}[m\in\N_0]}

\newcommandx{\sseqk}[1]{\sseq{#1}[k\in\N]}
\newcommandx{\sseqkz}[1]{\sseq{#1}[k\in\N_0]}

\newcommandx{\fml}[2]{\set{#1}_{#2}}
\newcommandx{\tol}[1]{\stackrel{#1}{\longrightarrow}}
\newcommandx{\eqd}[1][1=d]{\stackrel{(#1)}{=}}
\newcommandx{\neqd}[1][1=d]{\stackrel{(#1)}{\neq}}

\makeatletter
\let\oldabs\abs \def\abs{\@ifstar{\oldabs}{\oldabs*}}
\let\oldab\ab \def\ab{\@ifstar{\oldab}{\oldab*}}
\let\oldnorm\norm \def\norm{\@ifstar{\oldnorm}{\oldnorm*}}
\let\oldbkt\bkt \def\bkt{\@ifstar{\oldbkt}{\oldbkt*}}
\let\oldbrc\brc \def\brc{\@ifstar{\oldbrc}{\oldbrc*}}
\let\oldprn\prn \def\prn{\@ifstar{\oldprn}{\oldprn*}}
\let\oldpp\pp \def\pp{\@ifstar{\oldpp}{\oldpp*}}
\let\oldppup\ppup \def\ppup{\@ifstar{\oldppup}{\oldppup*}}
\let\oldppdown\ppdown \def\ppdown{\@ifstar{\oldppdown}{\oldppdown*}}
\let\oldee\ee \def\ee{\@ifstar{\oldee}{\oldee*}}
\let\oldvar\var \def\var{\@ifstar{\oldvar}{\oldvar*}}
\let\oldeeup\eeup \def\eeup{\@ifstar{\oldeeup}{\oldeeup*}}
\let\oldeedown\eedown \def\eedown{\@ifstar{\oldeedown}{\oldeedown*}}
\let\oldeeupdown\eeupdown \def\eeupdown{\@ifstar{\oldeeupdown}{\oldeeupdown*}}
\let\oldeeud\eeud \def\eeud{\@ifstar{\oldeeud}{\oldeeud*}}

\let\oldexp\exp \def\exp{\@ifstar{\oldexp}{\oldexp*}} \makeatother

\newcommand{\Bprn}[1]{ \prn*[\Big]{#1} }
\newcommand{\bprn}[1]{ \prn*[\big]{#1} }

\newcommand{\babs}[1]{\abs*[\big]{#1}}

\newcommand{\set}[1]{ \brc{#1} }
\newcommand{\sets}[2]{ \brc{#1\,:\,#2} }

\alias{\R}{{\mathbb R}}
\alias{\C}{{\mathbb C}}
\renewcommand{\Re}{\operatorname{Re}}
\renewcommand{\Im}{\operatorname{Im}}
\alias{\Z}{{\mathbb Z}}
\alias{\N}{{\mathbb N}}
\alias{\Nz}{{\mathbb N}_0}

\DeclareMathOperator\supp{supp}

\DeclareMathOperator\Leb{Leb}

\DeclareMathOperator\Var{Var}

\DeclareMathOperator\Cov{Cov}

\DeclareMathOperator*\esssup{esssup}

\newcommand{\oo}[1]{\frac{1}{#1}}

\newcommand{\too}[1]{\tfrac{1}{#1}}

\newcommand{\tot}{\too{2}} 

\newcommand{\ind}[1]{ 1_{{#1}}} 
\newcommand{\inds}[1]{ 1_{\set{#1}}} 

\newcommand{\downto}{\searrow}
\newcommand{\upto}{\nearrow}

\newcommand{\tsum}{\textstyle\sum}

\newcommand{\tint}{\textstyle\int}

\newcommand{\al}{\alpha}
\newcommand{\be}{\beta}

\newcommand{\eps}{\varepsilon}
\newcommand{\ld}{\lambda}
\newcommand{\Ld}{\Lambda}
\newcommand{\gm}{\gamma}
\newcommand{\vp}{\varphi}
\newcommand{\sg}{\sigma}

\newcommand{\sA}{\mathcal{A}} 
\newcommand{\sB}{\mathcal{B}} 
 \newcommand{\bC}{\mathbb{C}}
\newcommand{\sD}{\mathcal{D}} 
\newcommand{\sE}{\mathcal{E}} \newcommand{\bE}{\mathbb{E}}
\newcommand{\sF}{\mathcal{F}} 
\newcommand{\sG}{\mathcal{G}} 
\newcommand{\sH}{\mathcal{H}} 
\newcommand{\sI}{\mathcal{I}} 
 
\newcommand{\sK}{\mathcal{K}} 
\newcommand{\sL}{\mathcal{L}} \newcommand{\bL}{\mathbb{L}}
\newcommand{\sM}{\mathcal{M}} 
 \newcommand{\bN}{\mathbb{N}}
 
 \newcommand{\bP}{\mathbb{P}}
 
 \newcommand{\bR}{\mathbb{R}}

\newcommand{\sV}{\mathcal{V}}

\newcommand{\sZ}{\mathcal{Z}} 

\newcommand{\tA}{\tilde{A}}

\newcommand{\tN}{\tilde{N}}

 \newcommand{\hQ}{\hat{Q}}
 
\newcommand{\tS}{\tilde{S}} 
\newcommand{\tT}{\tilde{T}}

\newcommand{\tZ}{\tilde{Z}} 

\newcommandx{\pds}[2][1= ]{\frac{\partial #1}{\partial #2}}
\newcommandx{\tpds}[2][1= ]{\tfrac{\partial #1}{\partial #2}}
\newcommandx{\dd}[2][1= ]{\frac{d #1}{ d#2}}
\newcommandx{\tdd}[2][1= ]{\tfrac{d #1}{ d#2}}
\newcommandx{\pdt}[2][1= ]{\frac{\partial^2 #1}{\partial #2}}
\newcommandx{\tpdt}[2][1= ]{\tfrac{\partial^2 #1}{\partial #2}}

\newcommandx{\pdm}[3][2= ]{\frac{\partial^{#1} #2}{\partial #3}}
\newcommandx{\tpdm}[4][2= ]{\tfrac{\partial^{#1} #2}{\partial #3}}

\newcommandx{\up}[2]{{#1}^{#2}}
\newcommandx{\upp}[2]{{#1}^{(#2)}}

\newcommand{\eas}{\text{ a.s.}}

\newcommand{\eiff}{\text{ if and only if }}
\newcommand{\efor}{\text{ for }}
\newcommand{\eforall}{\text{ for all }}
\newcommand{\eforeach}{\text{ for each }}
\newcommand{\eand}{\text{ and }}

\newcommand{\eon}{\text{ on }}

\newcommand{\eie}{\text{ i.e. }}
\newcommand{\ewhere}{\text{ where }}
\newcommand{\ewith}{\text{ with }}

\newcommand{\cd}{c\` adl\` ag}

\newcommand{\eemu}[1]{\eeup{\mu}{#1}}
\newcommand{\seemu}[1]{\bE^{\mu}[#1]}
\newcommandx{\prfx}[1]{\set{#1}_{x\geq 0}}
\newcommandx{\sprfi}[1]{\{#1\}_{t\geq 0}}

\newcommand{\sfamtau}[1]{\{#1\}_{\tau\in\tm}}

\newcommand{\sfammu}[1]{\{#1\}_{\mu\in\Mlf(\sp)}}

\newcommand{\tm}{[0,\infty)}              
\providecommand{\sp}{}
\renewcommand{\sp}{\bR_+}                 

\newcommand{\IG}{\mathrm{IG}}

\newcommand{\B}{B}
\newcommand{\Bp}{B_{+}}
\newcommand{\Bc}{B_{c}}
\newcommand{\Bcp}{B^{+}_{c}}

\newcommand{\Cz}{C_{0}}
\newcommand{\Cc}{C_{c}}

\newcommand{\Conec}{C^{1}_{c}}

\newcommand{\Ltwo}{\bL^{2}}
\newcommand{\Lone}{\bL^{1}}
\newcommand{\Ltn}{\Ltwo(\sp, \ee*{\xi})}
\newcommand{\Linf}{\bL^{\infty}}
\newcommand{\Linfl}{\bL_{\mathrm{loc}}^{\infty}}
\newcommand{\Ltwol}{\bL_{\mathrm{loc}}^{2}}
\newcommand{\Lpee}{\bL^{p}}
\newcommand{\Lque}{\bL^{q}}
\newcommand{\Lpeel}{\bL_{\mathrm{loc}}^{p}}
\newcommand{\Lquel}{\bL_{\mathrm{loc}}^{q}}
\newcommand{\Lonel}{\bL_{\mathrm{loc}}^{1}}

\newcommand{\Mf}{\sM_{f}}

\newcommand{\Mlf}{\sM_{lf}}

\newcommand{\hrho}{\hat{\rho}}            
\newcommand{\define}[1]{\textbf{#1}}
\newcommand{\tv}{\tilde v}
\newcommand{\tw}{\tilde w}

\newcommand{\diffd}{\mathrm{d}}
\newcommand{\dt}{\diffd t}
\newcommand{\ds}{\diffd s}

\newcommand{\du}{\diffd u}
\newcommand{\dw}{\diffd w}

\newcommand{\da}{\diffd a}
\newcommand{\dtau}{\diffd \tau}
\newcommand{\dsg}{\diffd \sg}

\newcommand{\dxi}{\diffd \xi}
\newcommand{\dmu}{\diffd \mu}
\newcommand{\dbP}{\diffd \bP}
\newcommand{\dB}{\diffd B}
\newcommand{\dW}{\diffd W}
\newcommand{\dN}{\diffd N}
\newcommand{\dX}{\diffd X}
\newcommand{\dZ}{\diffd Z}
\newcommand{\dm}{\diffd m}

\newcommand{\oPi}{\overline{\Pi}}
\newcommandx{\prfis}[1]{\{#1\}_{t \in \sp}}
\newcommandx{\prfib}[1]{\big\{#1\big\}_{t \in \sp}}
\newcommandx{\prfiB}[1]{\Big\{#1\Big\}_{t \in \sp}}

\newcommand{\Cp}{\bC_{+}}
\newcommand{\Cpc}{\overline{\bC}_{+}}

\begin{document}

\title[Rough CIR Processes and Feller Random Measures]{A
Superprocess-based Approach to Rough CIR Processes and Feller Random Measures}

\author{Gordan {\v Z}itkovi{\' c}}
\address{Gordan {\v Z}itkovi{\' c}, Department of Mathematics, The
University of Texas at Austin}
\email{gordanz@math.utexas.edu}

\begin{abstract}
  Feller random measures generalize the Feller diffusion (the CIR process) by giving it memory. They arise as the scaling limits of
nearly unstable Hawkes processes, and include the rough CIR process
and its hyper-rough and discontinuous relatives. We show that every
Feller random measure is the occupation measure of a Dawson--Watanabe
superprocess whose spatial motion is a killed L\'evy subordinator,
integrated over the branching time. This is a continuum analogue of
the Hawkes--Oakes cluster representation. It yields existence, stability in the parameters, and stochastic
equations driven by an explicit noise. The structure of this noise
depends on whether the kernel has an atom at the origin.
Without an atom, the noise is a Brownian motion time-changed by the
distribution function of the measure. This leads to a martingale
characterization and to sharp results on densities and their H\"older
regularity. With an atom, the noise is a compensated inverse-Gaussian
process time-changed by the compensator of that distribution function,
and the measure is purely atomic.

\end{abstract}


\thanks{
  During the preparation of this work the author was
  supported by the National Science Foundation
  under Grant DMS-2307729. Any opinions, findings and
  conclusions or recommendations expressed in this material are
  those of the author and do not necessarily reflect the views of the
  National Science Foundation (NSF).
}


\maketitle

\section{Introduction}
\label{S:intro}

\subsection{Background and motivation}
\label{s:background}
Feller random
measures generalize the Feller diffusion (the CIR process of
mathematical finance) by giving it memory. Each one is indexed by
a pair $(\rho,\mu)$, where a probability measure $\rho$ on $\sp$ in
the positive geometrically infinitely divisible (PGID) class
(Definition \ref{d:pgid}) plays the role of a memory kernel and $\mu
\in \Mlf(\sp)$ describes ``immigration''. Its law is given by
$\ee{e^{-\ab{\xi,f}}} = e^{-\ab{\mu,v}}$ for bounded $f \geq 0$
with bounded support, where $v$ is the unique bounded nonnegative
solution of the Riccati-Volterra equation $v(t) = \int (f -
v^{2}/2)(t+s)\,\rho(\ds)$ (Definition \ref{d:frm} and
Proposition \ref{p:unique}). The family was introduced twice,
independently. In \cite{PacZit25} it arises as the class of
\emph{all possible} limits of nearly unstable, possibly
nonhomogeneous, Hawkes processes.
In \cite{Xu24a} these objects appear, within a larger class, under the name
\emph{$(K,\Phi)$-Hawkes Volterra processes}, as the solutions of
a family of stochastic Volterra equations whose kernel is a
measure rather than a function, and again as scaling limits of
asymptotically critical Hawkes processes (see also
\cite{SzyXu26}). 

The central observation of this paper is that Feller random measures
admit a continuum analogue of the Hawkes-Oakes cluster representation
(Theorem \ref{t:main}), which makes them a natural continuous
generalization of Hawkes processes. The family is also affine: its log-Laplace functional is linear in
$\mu$, with slope given by the solution of the
Riccati-Volterra equation above. Consequently, it retains the branching
property of the Feller diffusion. As we show in this paper, the
law of a Feller random measure depends continuously on $(\rho,\mu)$.
The family accommodates a wide range of self-exciting behavior, with
roughness and memory decoupled. Its members range from measures with
H\"older-continuous densities, through diffuse measures without a
square-integrable density, to purely atomic ones.

Special cases were studied long before the family was precisely
identified. 
An
exponential $\rho$ corresponds to the classical Feller (CIR) diffusion,
obtained in \cite{JaiRos15} as the nearly unstable limit of
light-tailed Hawkes processes. The Mittag-Leffler kernels of index $\al
\in (1/2,1)$ produce the \emph{rough CIR process} of
\cite{JaiRos16}, i.e., the solution of
\begin{align*}
  Z_{t} = Z_{0} + \frac{1}{\Gamma(\al)}\int_{0}^{t}(t-s)^{\al-1}\,
  \lambda\prn{\bar z - Z_{s}}\,\ds
  + \frac{\nu}{\Gamma(\al)}\int_{0}^{t}(t-s)^{\al-1}\sqrt{Z_{s}}\,\dB_{s},
\end{align*}
with constants $\lambda, \bar z, \nu > 0$ and a Brownian motion $B$.
It became the variance process of the rough Heston model
\cite{ElERos19} and a representative member of the affine Volterra
class \cite{AbiLarPul19}.
The fractional kernels of index $\al \in (0,1/2]$ constitute the
\emph{hyper-rough} regime of \cite{JusRos20, Abi21}, in which there is
no square-integrable density $Z$. The
critical case $\al = 1/2$ 
also arises from the local occupation time of a super-Brownian motion at
a point catalyst \cite{DawFle94, FleLeG95, MytSal15}.
In the parametrization $\al = H + 1/2$ by the Hurst index $H$,
empirical estimates of the roughness of volatility, initially around
$H \approx 0.1$ \cite{GatJaiRos18}, have recently moved toward $H = 0$
\cite{BayHarPig21, FukTakWes22, ForFukGerSmi22}, the boundary between
the rough and the hyper-rough regimes.

\subsection{The superprocess construction}
\label{s:constr}
Our main tool is the following construction (Theorem \ref{t:main}).
Let $K$ 
be a L\'evy subordinator 
killed at an independent
exponential time of rate one, chosen so that its potential measure is
$\rho$, and let $\sfamtau{X_{\tau}}$ be the Dawson-Watanabe
superprocess with spatial motion $K$ and the critical branching
mechanism $\ld \mapsto \ld^{2}/2$, started from $X_{0}=\mu$. Then the
occupation measure
\begin{align*}
  \xi(A) = \int_{\tm}X_{\tau}(A)\,\dtau \efor A \in \sB(\sp)
\end{align*}
is the $(\rho,\mu)$-Feller random measure, and every Feller random
measure arises in this way. Two time-like parameters
occur here. The variable $\tau$
is the time of the superprocess, which we call \emph{horizontal}, or
\emph{branching}, time. The variable ranging over $\sp$, the state
space of the superprocess, is the time of the Feller random measure,
and we call it \emph{vertical} time, or height. In this language,
Theorem \ref{t:main} says that a Feller random measure is what remains
of a superprocess once the branching time has been integrated away.

This construction is a continuum version of the cluster construction
of \cite{HawOak74}, which builds a Hawkes process generation by
generation rather than forward in time. Immigrants arrive according to
$\mu$, and each individual has a Poisson number of offspring, 
displaced from it by independent draws from a fixed law. Recording the $n$-th
generation as a measure on $\sp$ turns this into a branching random
walk, and superposing the generations recovers the Hawkes process. In
the nearly unstable regime, the generations become a continuum indexed
by $\tau$, the branching random walk becomes the superprocess $X$, and
the superposition becomes the integral above. The same picture
identifies the admissible $\rho$. The displacement of a randomly chosen
individual from its ancestor is a geometric sum of independent steps,
and its possible limits, as the branching ratio tends to one and
after rescaling, are the PGID laws of \cite{KleManMel84}. These are exactly the
potential measures of L\'evy subordinators killed at rate one. Equivalently,
$1/\hrho$, where $\hrho$ is the Laplace transform of $\rho$, is
the Laplace exponent
$\Psi(\ld) = 1 + b\ld + \int_{(0,\infty)}\prn*{1-e^{-\ld t}}\,\Pi(\dt)$
of $K$, where $b \geq 0$ is the drift and $\Pi$ the L\'evy measure of
$L$. 
The identity $\hrho = 1/\Psi$ relates the perspectives
and is used
throughout the paper. Remark \ref{r:cluster} describes the
correspondence in more detail.

Superprocesses have been used in this context before, albeit in a
significantly different manner: there the Volterra variable is the
time of the superprocess, here it is its space. The link with stochastic Volterra
equations goes back to \cite{MytSal15}, motivated by catalytic
superprocesses. In \cite{Abi21}, the rough square-root Volterra
equation for $H=0$ is obtained from a super-Brownian motion that
branches only at a single point catalyst \cite{DawFle94,
FleLeG95}.

Like the Markovian lifts of \cite{CarCou98}, \cite{AbiElE19},
\cite{CucTei20} or \cite{Ham25}, our construction uses a Markov
process to represent $\xi$, but in a different way. Instead of
enlarging the state space, it introduces an additional time
parameter, in which the superprocess is Markov, and integrates
this parameter away. The whole path of $\xi$ is thus obtained at
once, from one trajectory of a measure-valued Markov process. Unlike
the lifts, which require local square integrability of the kernel, our approach applies also to the
hyper-rough range.

Our construction also explains why the inverse-Gaussian law appears
repeatedly in this literature, e.g., in \cite{BonFor25} or
\cite{AbiAtt25}. If $\mu$ is finite, then $\xi(\sp) =
\int_{0}^{\infty}\ab{X_{\tau},1}\,\dtau$ is the integral of the total
mass $m_{\tau}$ of the superprocess, which follows the 
classical Feller diffusion, $\diffd m_{\tau} = -m_{\tau}\,\dtau +
\sqrt{m_{\tau}}\,\diffd W_{\tau}$ (Proposition
\ref{p:dw-props}(\ref{i:mass-feller})). Its law depends only on the
branching mechanism and the rate-one killing, and not on the drift or
the jumps of $K$. The integral of $m$ up to extinction is
inverse-Gaussian, so $\xi(\sp)$ has the same inverse-Gaussian law for
every $\rho$ (Proposition \ref{p:basic}(\ref{i:mass-ig})). A
conditional version holds for the jumps: when $\rho(\set{0})>0$, the
conditional law of a jump of $\xi[0,\cdot]$ at a predictable time,
given the past, is inverse-Gaussian (part \ref{i:pred-jump} of Theorem
\ref{t:ig-rep}).
 
\subsection{Main results}
\label{s:results}
Besides the superprocess representation itself, the main new results are
(i) existence for every PGID $\rho$, including
$\rho(\set{0})>0$, together with stability in $(\rho,\mu)$, which recovers,
for instance, the limit of the integrated variance in
\cite{AbiAttRos26} as $H \downto -1/2$; (ii) the structure of
the noise, which is determined by $\rho$ and is inverse-Gaussian when
$\rho(\set{0})>0$; (iii) the ``only if'' half of the density
criterion (see subsubsection \ref{s:densities} below);
and (iv) the exact H\"older exponent and the failure of the
semimartingale property of the density. The construction also
gives new proofs of several known results, often under weaker
assumptions.

What follows is a more
detailed list with attributions and comparisons with
existing results.

\subsubsection{Existence, uniqueness and stability} We show
that every PGID $\rho$ and every $\mu \in \Mlf(\sp)$ define a
Feller random measure, constructed directly from the superprocess
(Corollary \ref{c:exists}). The solution $v$ in
Definition \ref{d:frm} is unique for every PGID $\rho$, including the
discontinuous case (Proposition \ref{p:unique}), so the law of $\xi$ is
determined by $(\rho,\mu)$.
In \cite{PacZit25}, existence was
obtained only for $\rho(\set{0})=0$ and only as a limit of Hawkes
processes. \cite[Theorem 2.9]{Xu24a} also obtains these objects
as limits of Hawkes processes, including kernels with an atom at
the origin, with convergence in the $S$ topology. In the
univariate case, the kernels of \cite{Xu24a} are the measures with
Laplace transform $1/\Phi$ for an extended Bernstein function $\Phi$,
i.e., a Bernstein function plus a real constant. When $\Phi(0)>0$ they
are, up to a constant factor, exactly the PGID laws. The kernels with
$\Phi(0) \leq 0$ have infinite mass and lie outside the present
setting. Moreover, we
show that the law of $\xi$ depends continuously on $(\rho,\mu)$
(Proposition \ref{p:stability}), and also provide a vague version
which accounts for kernel mass escaping to infinity (Corollary
\ref{c:vague-stability}).

\subsubsection{Stochastic equations and their noise}
Several stochastic equations for $\xi$ are derived from the construction.
The martingale measure of the underlying superprocess is projected onto
a single \emph{vertical noise} $N$ on $\sp$, in the same way that $\xi$
is obtained from $X$. Integrating the associated SPDE over the branching
time then gives equations for $\xi$ driven by $N$, among them the
``divergence-form'' and Volterra equations
\begin{align}
  \label{grace}
  \prn{ k * \xi }(\dt) = \prn{ N_{t} + \mu[0,t] }\,\dt
  \eand
  \xi[0,t] = (\rho*\mu)[0,t] + \int_{0}^{t}\rho[0,t-s]\,\dN_{s},
\end{align}
where the measure $k$ satisfies the Sonin relation $k*\rho =
\Leb$ \cite{Son84}, extended here from functions to measures (see also
\cite{Koc11}), and
is computed directly from $\Psi$ (Section \ref{S:repr}). These
equations  hold for every PGID
$\rho$. In the continuous case, Volterra equations of this kind
have been studied by \cite{Abi21} and, in the context of Hawkes
processes, by \cite{Xu24a}.
In the
theory of affine stochastic Volterra equations \cite{AbiLarPul19,
AbiCucLarPul21, Abi21}, by contrast, the kernel is given and the
counterpart of $k$ (the resolvent of the first kind of the kernel)
enters mainly through hypotheses on its sign and monotonicity.
Moreover, \cite{Alf25} singles out admissible kernels through
the preservation of nonnegativity, and the density is the object of interest. Here the
primitive object is $\xi$, and a density is an optional feature of it.

In our approach, the noise is not given in advance but determined by
$\rho$. We show, nevertheless, that, on a
possibly enlarged probability space, we have 
\begin{align*}
  N_{t} = \tS_{\ab{N}_{t}},
\end{align*}
where $\ab{N}$ (the predictable quadratic variation of $N$) is the
compensator of $\xi[0,\cdot]$, and $\tS$ is a
centered L\'evy process with $\Var\prn*{\tS_{y}} = y$.
The process $\tS$ can only be of one of the following two types: it is a Brownian
motion when $\rho(\set{0})=0$ (Corollary \ref{c:dds}), as in
\cite[Theorem 2.14]{Xu24a}, or a compensated inverse-Gaussian process
when $\rho(\set{0})>0$ (Theorem \ref{t:ig-rep}). 

\subsubsection{A martingale characterization}
When $\rho(\set{0})=0$, we show that a random measure $\xi$ is a
$(\rho,\mu)$-Feller random measure if and only if there is a continuous
local martingale $\eta$ with $\bkt{\eta}_{t} = \xi[0,t]$ such that
the divergence-form equation \eqref{grace} above holds with $N$ replaced by
$\eta$ (Theorem \ref{t:mart-char}). This is a martingale characterization of the
distributional Definition \ref{d:frm}, in the spirit of L\'evy's
characterization of Brownian motion. It needs no density. Remark \ref{r:sing-noise}
compares it with the characterization in Volterra form of
\cite[Theorem 2.14]{Xu24a}. 
The convergence of the inverse-Gaussian simulation scheme (iVi) of
\cite[Corollary 2.3]{AbiAtt25} relies on uniqueness in law, which was
previously known, for example, for completely monotone kernels
\cite{AbiAtt25, Abi21} (see also \cite[Theorem 2.14]{Xu24a}).
Theorem \ref{t:mart-char} provides this uniqueness whenever the kernel
of \cite{AbiAtt25} corresponds to a PGID $\rho$.

\subsubsection{Densities}
\label{s:densities}
Next, we prove that the measure $\xi$ has a square-integrable
density $Z$, i.e., one that is square integrable on $[0,T] \times
\Omega$ for each $T$, if and only if $\rho$ has a locally
square-integrable density $r$ (Proposition \ref{p:density-crit}). The
``if'' half is close to \cite[Corollary 2.16]{Xu24a} (see Remark
\ref{r:density-conj}). For fractional kernels with $H \in
(-1/2,0]$, \cite[Theorem 3.1]{JusRos20} show that $\xi[0,\cdot]$ is not
continuously differentiable. The ``only if'' half shows, for general
$\rho$, that there is not even a square-integrable density. The density
$Z$ satisfies the stochastic equations of Theorem \ref{t:Z-eqs}, which
reduce to those of \cite{JaiRos16, ElERos19} in the rough CIR case. At
the critical index $H=0$, Example \ref{e:critical-log} shows that a
slowly varying correction to the jump measure restores a
square-integrable density, in line with the log-modulated models of
\cite{BayHarPig21}, so that the Blumenthal-Getoor index of $\Pi$
alone does not settle the question in general.

\subsubsection{Regularity}
The growth of $\Psi$ at infinity controls the H\"older regularity of
$Z$ (Corollary \ref{c:holder}), which recovers the known H\"older
continuity of the rough CIR process \cite{AbiLarPul19}. 
For $b=0$, jump densities
comparable to $t^{-1-\al}$ near the origin and bounded on $[1,\infty)$,
with $\al \in (1/2,1)$, and nonzero $\mu$ with a locally square-integrable
density, we prove that the H\"older exponent $\al-1/2$ is exact, and
that no version of $Z$ is a semimartingale (Propositions
\ref{p:max-holder} and \ref{p:not-semimart}, Corollary
\ref{c:exact-holder}). The failure of the semimartingale property is
often asserted, e.g., in \cite{AbiAttRos26, FriGerWie25}, and the
optimality of the H\"older exponent is implicit in the literature, but
we have been unable to locate proofs of either. The companion
statement, the failure of the Markov property, is proved in
\cite{FriGerWie25}.

\subsubsection{The discontinuous case}
When $\rho(\set{0})>0$, we prove that the noise $N$ is a compensated
inverse-Gaussian process time-changed by the compensator of
$\xi[0,\cdot]$ (Theorem \ref{t:ig-rep}), and that $\xi$ is purely atomic
(Corollary \ref{c:atomic}). For a given inverse-Gaussian subordinator,
$\xi[0,\cdot]$ is the unique solution of a Volterra equation based on
this description of its noise (Theorem \ref{t:disc-eqs}). In
\cite{Xu24a}, these objects appear as limits, together with a
Volterra equation, but without a detailed description of its
martingale part (see Remark \ref{r:xu}).

We note that kernels with an atom at the origin arise naturally in numerical work:
the discretizations used to simulate Volterra equations are of this
type, and the discretized processes have jumps \cite[Equation
(2.4)]{AbiAtt25}. They lie outside the solution concept of
\cite{Abi21}, whose solutions are continuous by definition. The
simplest case $\rho = \delta_{0}$ is the limit obtained in
\cite{AbiAttRos26} as $H \downto -1/2$ (Example
\ref{e:delta}).

\subsubsection{Examples}
We provide several new examples, and establish novel properties of
known ones. Every L\'evy subordinator gives an admissible kernel,
with the behavior of $\Psi$ at infinity setting the roughness and
around $0$ the memory. In the continuous case, the examples of subsection
\ref{s:examples} are the classical CIR process (Example \ref{e:cir}),
the rough CIR family (Example \ref{e:rough-cir}), the critical index
$\al=1/2$, where the spatial motion is, up to a constant time change, the
first-passage process of a Brownian motion, and which connects with
the catalytic construction (Example \ref{e:critical}),
log-modulated kernels at the critical index (Example
\ref{e:critical-log}), and tempered kernels, which decouple roughness
from memory and correspond to the gamma kernel of \cite{AbiLarPul19}
(Example \ref{e:temper}). In the discontinuous case, we treat
$\rho=\delta_{0}$, for which $\xi$ is independently scattered with
inverse-Gaussian marginals (Example \ref{e:delta}). For Lebesgue $\mu$,
its distribution function is the inverse-Gaussian L\'evy process
obtained in \cite{AbiAttRos26} as the limit of hyper-rough square-root
processes as $H \downto -1/2$. We recover this convergence of the
integrated variance, in the
Skorokhod $M_{1}$ topology, from Corollary
\ref{c:vague-stability}. The joint convergence with the martingale
part, in the $S$ topology, is not covered by our argument. We
also treat lattice kernels, which give exact inverse-Gaussian
autoregressions (Examples \ref{e:lattice} and \ref{e:lattice-gen}).
In the Poisson case, the superprocess reduces to an infinite system of
Feller-type stochastic differential equations, and $\xi$ is a
Ji\v{r}ina process \cite{Jir58} (Example \ref{e:lattice}). We also treat
exponential excitation, which leads to a continuous-state branching
process with immigration (Example \ref{e:exp-disc}).

\subsection{Organization of the paper}
Section \ref{S:defs} fixes the notation, defines Feller random
measures and PGID kernels, and proves uniqueness in law. Section
\ref{S:repr} proves the superprocess representation, derives from it
existence, basic properties and continuity in the parameters, and
obtains the stochastic equations driven by the vertical noise. Section
\ref{S:cont} treats the continuous case: the noise, the martingale
characterization, the existence of a density, its equations and
regularity, and the sharpness of the regularity results. Section
\ref{S:disc} treats the discontinuous case. Appendix \ref{S:spier} is a
short primer on Dawson-Watanabe superprocesses, and the remaining
appendices contain all but the shortest proofs, section by section.

\subsection{Use of generative AI}
During the preparation of this manuscript, the author used
Anthropic's Claude (Opus 4.8, 5 and 5.5) to assist with literature
searches, to draft and polish exposition, and to
check the manuscript for errors. 

Two mathematical ideas in the paper, neither of which is central, were suggested by the model. 
The first is to pass to complex
arguments and to Fourier (rather than Laplace) transforms
in the proof of Corollary
\ref{c:holder}. The second is a simplification of 
the proof of Proposition \ref{p:not-semimart}. Apart from these,
every definition, statement, proof and example in the paper,
while possibly edited and rewritten by the model, was conceived
by the author.

All AI-generated output, including the references identified
through literature sear\-ches, was independently reviewed and
verified by the author. The author takes full responsibility for
the correctness, originality, and content of the manuscript,
including any errors that may remain.

\section{Setup} 
\label{S:defs} 

This section is mostly preparatory. It consists of definitions
and of known facts recast in the form we shall need, but it also
contains one new result, Proposition \ref{p:unique}, which extends a
known uniqueness statement to the discontinuous case. 
The main results of this paper begin in Section \ref{S:repr}. All proofs, except
for very short ones, are postponed to the Appendix.

\subsection{Notation}
\label{s:notation}

$\sB(U)$ denotes the Borel $\sigma$-algebra of a topological
space $U$. For $I \subseteq \bR$, $\B(I)$ denotes the family of
all bounded Borel functions on $I$, $\Bp(I)$ its nonnegative
members, $\Bc(I)$ those supported by a bounded set, and $\Bcp(I)
= \Bp(I) \cap \Bc(I)$.

$\Cz(I)$ and $\Cc(I)$ denote the families of continuous
functions on $I$ which are, respectively, vanishing at infinity and
compactly supported, while
$\Conec(I)$ denotes the continuously differentiable members of
$\Cc(I)$. The Lebesgue spaces are denoted by $\bL^{p}$, with the
underlying measure space displayed when there is a possibility of confusion.
Subscript ``loc'' localizes properties to all compact subsets.
For a function $f$ on $\sp$ and $T>0$ we write
$\norm{f}_{\infty,T} := \sup_{[0,T]}\abs{f}$.
$f \sim g$ stands for $f/g \to 1$,  with the precise limit to be
specified for each use. 

\smallskip

$\Leb$ denotes the Lebesgue measure, which we also write as $\dt$
(or $\ds$, $\du$, \dots).
$\nu(\dt) = h(t)\,\dt$ means that $\nu$ is
absolutely continuous with density $h$.
For a measure
$\nu$ and a measurable set $I$, $\nu|_{I}$ denotes the measure on the
same domain, defined by $A \mapsto \nu(A \cap I)$. We write $\Mf(\sp)$
and $\Mlf(\sp)$ for the finite and the locally finite measures on
$\sB(\sp)$, respectively, and use the weak topology on the former and
the vague one on the latter. Both carry the compatible $\sg$-algebra
generated by the evaluations $\mu \mapsto \mu(A)$, $A \in \sB(\sp)$.
We also set $\Mlf^{*}(\sp) := \Mlf(\sp) \setminus \set{0}$. A
random measure on $\sp$ is always assumed to be locally finite,
i.e., a random element of $\Mlf(\sp)$. 

\smallskip

The pairing $\int f\,\dmu$ of a function and a measure is written
as $\ab{\mu,f}$. Convolution of functions or measures is denoted
by $\ast$. All functions and measures defined on $\sp$ are
implicitly extended trivially to $(-\infty, \infty)$: functions
take the value $0$ on $(-\infty,0)$ and measures do not charge it.
In particular, this applies to the paths of stochastic
processes: for a \cd{} process $F$ on $\sp$ we have $F_{0-}=0$
and $\Delta F_{0}=F_{0}$. If $F$ is, in addition, of finite
variation, $\diffd F$ denotes the random signed locally finite
measure on $\sp$ it induces, so that $\diffd F(\set{0})=F_{0}$.
Accordingly, $\int_{0}^{t}\vp\,\diffd F$ includes the term
$\vp(0)F_{0}$, while $\int_{0+}^{t}\vp\,\diffd F$ does not.

\smallskip

$\Cp = \set{z \in \bC: \Re z > 0}$ 
and $\Cpc = \set{z \in \bC: \Re z \geq 0}$.
The Laplace-Fourier transform of $\mu \in \Mlf(\sp)$ is
$\hat{\mu}(z) = \ab{ \mu, e^{-z \cdot}} = \int_{\sp}e^{-z
t}\,\mu(\dt)$, defined for those $z \in \Cpc$
for which the integral converges absolutely. If $z$ is
real, we use $\ld$ instead of $z$ and we allow the value
$+\infty$, so that the map $\ld \mapsto \hat\mu(\ld)$ is defined
for every $\ld \geq 0$, with values in $[0,\infty]$. 

\smallskip

$\{\theta_{t}\}_{t \in \sp}$ denotes the family of shifts,
acting on functions by $\theta_{s}f = f(s + \cdot)$ and, dually,
on measures by $\ab{\theta_{s}\mu, f} := \ab{\mu, \theta_{s}f}$,
so that $\theta_{s}\mu$ is the image (push-forward) of $\mu$
under $x \mapsto s + x$.

\smallskip

Square brackets are used for the quadratic variation, as in
$\bkt{M}_{\tau}$ or $\bkt{M,N}_{\tau}$, 
since $\ab{\cdot,\cdot}$ is already reserved 
for measure-function pairings. The exception is the predictable
quadratic variation $\ab{N}$ of a martingale $N$, for which we keep the
standard angle brackets. 

A point superscript, as
in $\eeup{t}{\cdot}$, refers to the spatial motion started at
$t$. A measure superscript, as in $\eemu{\cdot}$, refers to the
superprocess started from $\mu$ or the Feller random measure with
parameter $\mu$, and not to the spatial motion started from
the distribution $\mu$. A superscript pair, as in
$\eeup{\tau,t}{\cdot}$ or $\eeup{\tau,\mu}{\cdot}$, prepends the
branching time parameter at which the process is started, so that
$\eeup{t}{\cdot} = \eeup{0,t}{\cdot}$. For a random measure
$\xi$, we write $\ee{\xi}$ for the intensity measure of $\xi$,
i.e., $\ee{\xi}[A] = \ee{\xi[A]}$ for all $A$.

\smallskip

We keep the horizontal (branching) and the vertical time axes of
subsection \ref{s:constr} typographically separated. The horizontal
axis is denoted by $\tm$, with Greek variables such as $\tau$ and
$\sg$. The vertical axis is $\sp$, with Latin variables such as $t$
and $s$, whose values we also call heights. Since the
vertical axis is the space of the Dawson-Watanabe superprocess, the
usual space-time coordinates $(t,x)$ become $(\tau,t)$ in this paper.

\subsection{Feller random measures}
\label{s:Feller}
The central object of this paper is the \emph{Feller random
measure}, introduced in \cite{PacZit25}. Feller random measures 
form a family of
generalizations of the classical Feller diffusion, indexed by a probability
measure $\rho$ on $\sp$, which plays the role of a memory kernel, and
a measure $\mu$, which describes immigration. For an
exponential $\rho$ one recovers the classical picture: the
$(\rho,\mu)$-Feller random measure is $\xi(\dt) = Z_{t}\,\dt$, where
$Z$ is a Feller diffusion with time-dependent immigration given by $\mu$. 
For a general $\rho$ the
memory is no longer exponential, and $\xi$ may fail to have a density
altogether. Examples of concrete choices of $\rho$, and the random measures they
give rise to (the rough CIR process among them), are collected in
subsection \ref{s:examples} and in Section \ref{S:disc}, once enough theory is available.

We start with the class of kernels required by the
construction.
\begin{definition}[\cite{KleManMel84}]
\label{d:pgid}
A probability measure $\rho$ on $\sp$ is called \define{positive
geometrically infinitely divisible (PGID)} if, for every
$p \in (0,1)$, it is the distribution of a sum of $G_{p}$ nonnegative
i.i.d.\ random variables whose common law may depend on $p$,
where $G_{p}$ is independent of the summands and geometric with
parameter $p$, i.e., $\pp{G_{p}=k}=p(1-p)^{k-1}$ for $k \geq 1$.
\end{definition}

Feller random measures are defined by prescribing their Laplace
functionals, with $\rho$ entering as the kernel of a Riccati-Volterra
equation. Why $\rho$ is required to be PGID, rather than an
arbitrary probability measure on $\sp$, is explained in Remark
\ref{r:cont-disc}(\ref{i:pgid-needed}) below.

\begin{definition}\label{d:frm}
Let $\rho$ be PGID, and let
$\mu \in \Mlf(\sp)$. A random measure $\xi$ on
$\sp$ is a $(\rho,\mu)$-\define{Feller random measure} if for every
$f \in \Bcp(\sp)$ there is a bounded measurable solution $v \geq 0$ of
\begin{align} \label{riccati}
  v(t) = \int_{\sp} \Bprn{ f(t + s) - \tot v^{2}(t + s)
  }\,\rho(\ds)
\end{align}
such that
\begin{align}\label{laplace}
  \ee{ e^{-\ab{\xi,f}} } = e^{-\ab{\mu,v}} .
\end{align}
We say that a Feller random measure $\xi$ is \define{continuous}
if $\rho(\set{0})=0$, and \define{discontinuous} otherwise. 
\end{definition}

\begin{remark}\leavevmode
\label{r:cont-disc}
\begin{enumerate}
\item \label{i:naming}
The names refer to the distribution function $t \mapsto \xi[0,t]$.
It is continuous when $\rho(\set{0})=0$, by Proposition
\ref{p:basic}(\ref{i:diffuse}) below, and, for $\mu \in \Mlf^{*}(\sp)$, it
is a pure-jump function when $\rho(\set{0})>0$, by Corollary
\ref{c:atomic} in Section \ref{S:disc}. Proposition \ref{p:no-atom}
collects several conditions equivalent to $\rho(\set{0})=0$.
\item \label{i:pgid-needed}
Not
every probability measure $\rho$ on $\sp$ defines a random
measure via \eqref{riccati} and \eqref{laplace}. For example, \eqref{riccati} has
no bounded nonnegative solution if
$\rho$ is uniform on $[0,1]$, which is not PGID by Remark
\ref{r:pgid-ex}, and $f = \ind{[1,2)}$. On the other hand, by
Corollary \ref{c:exists}, every PGID $\rho$ defines a Feller random
measure. 
\item \label{i:scaling}
We normalize the coefficient of $v^{2}$ in \eqref{riccati} to
$1/2$ and take $\rho$ to be a probability measure only for notational
simplicity. A general coefficient $\sigma^{2}/2$ and a finite kernel
$c\rho$ reduce to this case by rescaling $\xi$ and $\mu$: the resulting
random measure has the law of $c^{2}\sigma^{2}\xi_{0}$, where $\xi_{0}$
is the $(\rho,\mu/(c\sigma^{2}))$-Feller random measure.
\end{enumerate}
\end{remark}

Uniqueness in law of the Feller random measure corresponding to a
given pair $(\rho,\mu)$ has been treated in
\cite[Theorem 2.2 and Proposition A.1]{PacZit25} in the
continuous case. Proposition \ref{p:unique} below extends this to
the remaining discontinuous case. Existence will be established
  later, in Corollary \ref{c:exists}.

\begin{proposition}
\label{p:unique}
For $f \in \Bcp(\sp)$
the equation
\eqref{riccati} has at most one bounded nonnegative solution.
Moreover, any such solution vanishes on $(\sup \supp f, \infty)$
and satisfies $0 \leq v \leq \norm{f}_{\infty}$.
\end{proposition}
\begin{proof} In Appendix \ref{s:prf-unique}.
\end{proof}

Thanks to Proposition \ref{p:unique}, the uniqueness of the law of
$\xi$ follows from the fact that
$\Bcp(\sp)$ is a
measure-determining class \cite[Corollary 2.3]{Kal17}.
We write
$v[f]$ for the unique solution of \eqref{riccati} corresponding to
$f \in \Bcp(\sp)$ and observe that uniqueness also implies that
\begin{align}\label{shift}
v[f](s+t) = v[\theta_s f](t) = v[\theta_{s+t} f](0) \efor s,t\in
\sp,
\end{align}
where $\theta$ is the shift operator of subsection \ref{s:notation}. 

Throughout, $U$ denotes the adjoint of convolution by $\rho$,
\begin{align}\label{conv}
  U \vp (t)
  = \int_{\sp} \vp(t+s)\, \rho(\ds) =
  \ab{ \theta_{t}\rho, \vp } \efor \vp \in \B(\sp) ,
\end{align}
so that \eqref{riccati} reads $v = U\prn{ f - v^{2}/2 }$. The operator $U$ is
positive, with
$\norm{U\vp}_{\infty} \leq \norm{\vp}_{\infty}$
and
$U\vp(0) = \ab{\rho,\vp}$. 

\subsection{A L\'evy subordinator}
\label{s:subord}
To construct the Dawson-Watanabe superprocess needed for
Theorem \ref{t:main}, we first describe its spatial motion. 

\smallskip

Let $L$ be a conservative (not killed)
subordinator, i.e., a nondecreasing \cd{} L\'evy
process.
Let $K$ be $L$ killed at an independent exponential
time of rate $1$. We use the standard convention that functions on $\sp$ are
extended to the \emph{cemetery state}
$\dagger$ by setting
$f(\dagger)=0$ (in particular, $e^{-z\dagger}:=0$ for $z \in \Cpc$). The Laplace exponent $\Psi$ of $K$ is
defined by
\begin{align*}
  e^{- \tau\Psi\prn{ \ld  }} = \ee{ \exp{ - \ld K_{\tau}} }
  \efor \ld \geq 0 \eand \tau \in \tm.
\end{align*}
The function $\Psi$ admits the L\'evy-Khinchine
representation
\begin{align}\label{exponent}
  \Psi\prn{ \ld  } = 1 + b\ld  + \int_{(0,\infty)}\prn{ 1 - e^{- \ld t}
  }\,\Pi\prn{ \dt } \eforall \ld \geq 0
\end{align}
with the \emph{killing rate} $1$, 
a \emph{drift coefficient} $b \geq 0$ and a
measure $\Pi$ on $(0,\infty)$ which integrates $\min(1,t)$,
called the \emph{L\'evy measure}.
Equivalently, $b$ and $\Pi$ are the drift and the L\'evy measure
of $L$, and $\Psi-1$ is its Laplace exponent.
Since $\Psi-1$ is a Bernstein function, $\Psi$ extends
holomorphically to $\Cp$ and continuously to $\Cpc$,
where \eqref{exponent} continues to hold and $\Re \Psi \geq 1$ (see
\cite[Proposition 3.6 and its proof]{SchSonVon12}). Both sides of the
relation defining $\Psi$ are therefore holomorphic on $\Cp$ (the
right-hand one because $\abs*{e^{-zK_{\tau}}} \leq 1$ on $\Cpc$)
and 
continuous on $\Cpc$, and equal on $[0,\infty)$. Hence
$\ee{ e^{-z K_{\tau}} } = e^{-\tau\Psi(z)}$ for every $\tau \in \tm$
and $z \in \Cpc$. At $z = -iu$ with $u \in \bR$ this exhibits
$\Psi(-iu)$ as the characteristic exponent of $K$:
\begin{align*}
  \ee{ e^{iuK_{\tau}} } = e^{-\tau\Psi(-iu)}
  \eforall u \in \bR \eand \tau \in \tm.
\end{align*}

\smallskip

Let $\sprfi{\bP^{t}}$ be the Markov family associated to $K$ on the
canonical space $D_{\sp \cup \{ \dagger \}}$ of \cd{} paths absorbed at
the cemetery state $\dagger$, i.e.,
\begin{align*}
  \ppup{t}{B} = \pp{ t + K_{\cdot} \in B }\efor B \in
  \sB(D_{\sp \cup \{ \dagger \}}) \eand t \in \sp,
\end{align*}
with the convention $t + \dagger = \dagger$.
The associated Feller ($\Cz(\sp) \to \Cz(\sp)$) semigroup
$\sfamtau{P_{\tau}}$ is given by
\begin{align*}
  P_{\tau}f(t) = \eeup{t}{ f(K_{\tau}) } = e^{-\tau}\,\ee{ f(t +
  L_{\tau}) },
\end{align*}
and its generator $\sA$ by
\begin{align*}
  \sA f = - f + b\, f' + \int_{(0,\infty)}\Bprn{ f( \cdot + s) - f(
  \cdot ) }\,\Pi(\ds)
  \efor f \in \sD(\sA) \supseteq \Conec(\sp).
\end{align*}

\subsection{PGIDs are potential measures} \label{s:potential}
The connection between L\'evy subordinators and Feller random
measures starts to emerge in the following lemma, in which $E$ denotes a
rate-one exponentially distributed random variable, independent
of $L$.
\begin{lemma}
\label{l:pgid}
For a probability measure $\rho$ on $\sp$, the following are
equivalent:
\begin{enumerate*}[label=(\arabic*)]
\item $\rho$ is PGID,
\item $\rho$ is the law of $L_{E}$ for some L\'evy subordinator $L$,
\item $1/\hrho - 1$ is a Bernstein function.
\end{enumerate*}
In that case, $1/\hrho - 1 = \Psi - 1$ is the Laplace exponent of
$L$, and
\begin{align}\label{rho-psi}
  \hrho(z) = \oo{\Psi(z)} \eforall z \in \Cpc.
\end{align}
\end{lemma}
\begin{proof}
The equivalence of (1) and (2) is implicit in \cite[Theorem
2]{KleManMel84}. If $\rho$ is the law of $L_E$, then, since $\Psi-1$ is the Laplace
exponent of $L$, we have
\begin{align}
  \label{rho-lap}
  \hrho(\ld) = \ee{ \exp{-\ld L_E}} = \int_{\tm} e^{-\tau}\,
  \ee{ \exp{-\ld L_{\tau}} }\, \dtau = \oo{\Psi(\ld)}
\end{align}
for all $\ld \geq 0$. 
Bernstein functions vanishing at $0$ 
are exactly the Laplace exponents of
conservative subordinators (see \cite[Theorem 5.2 and the
discussion following Proposition 5.5]{SchSonVon12}), which gives
the equivalence of (2) and (3).
To extend \eqref{rho-lap} to $\Cpc$, note that $\hrho$ is
holomorphic on $\Cp$ and continuous on $\Cpc$, since $\rho$ is a
probability measure. The function $\Psi$ has the same properties by
subsection \ref{s:subord}, and it does not vanish on $\Cpc$ since
$\Re\Psi \geq 1$, so $1/\Psi$ has them as well. Two such functions
that agree on $[0,\infty)$ agree on $\Cp$, and hence on $\Cpc$ by
continuity.
\end{proof}
Equation \eqref{rho-lap} then gives a
bijective correspondence between the laws of L\'evy subordinators and
PGID probability measures.
From now on, whenever $L$ (or $K$) and $\rho$ appear in the same
context, $\rho$ is the law of $L_{E}$, so that \eqref{rho-psi} holds.
The following characterizations of continuity of Feller random measures
are used in both
Section \ref{S:cont} and Section \ref{S:disc}.
\begin{proposition}
\label{p:no-atom}
The following statements are equivalent:
  \begin{enumerate*}[label=(\arabic*), ref=\arabic*]
  \item \label{i:atom0} $\rho(\set{0})=0$,
  \item \label{i:no-atoms} $\rho$ has no atoms at all,
  \item \label{i:not-cp} $L$ is not a compound Poisson process,
  \item \label{i:strict-incr} $L$ is strictly increasing, a.s.,
  \item \label{i:psi-infty} $\lim_{\ld \to \infty} \Psi(\ld) = \infty$,
  \item \label{i:inf-activity} $b > 0$ or $\Pi(0,\infty) = \infty$.
  \end{enumerate*}
\end{proposition}
\begin{proof}
  For the implication $(\ref{i:not-cp}) \Rightarrow (\ref{i:no-atoms})$, see \cite[Proposition 15]{Ber96}. The other implications follow readily
  from basic properties of subordinators and their Laplace
  transforms.
\end{proof}
\begin{remark}
\label{r:pgid-ex}
We give an overview of
well-known examples of PGID distributions, both continuous and
discontinuous.

The most prominent nontrivial examples of PGID probabilities form the
\define{Mittag-Leffler family}, parameterized up to scaling by
$\al \in (0,1]$ through 
\begin{align*}
  \hrho_{\al}(\ld) = (1+\ld^{\al})^{-1},
\end{align*}
 with
the value $\al = 1$ corresponding to the exponential distribution.
For this family,
$L$ is the $\al$-stable subordinator for $\al<1$ and the
deterministic unit-drift subordinator for $\al=1$.

A large supply of further examples comes from complete
monotonicity (see \cite{SchSonVon12} for the terminology used
here). If $\rho$ admits a completely monotone density $r$, say
$r(t) = \int_{\sp}e^{-ts}\,\sigma(\ds)$, then
$\hrho(\ld) = \int_{\sp}(\ld+s)^{-1}\,\sigma(\ds)$ is a Stieltjes
function, so that $1/\hrho$ is a complete Bernstein function
\cite[Theorem 7.3]{SchSonVon12}. Since $\hrho(0)=1$, the difference
$1/\hrho - 1$ is then a Bernstein function vanishing at the origin,
and $\rho$ is PGID. This covers all mixtures of exponential
distributions, such as the Gamma distributions with shape parameter
$\al \leq 1$, including the case $\theta = 1$ of Example
\ref{e:temper},
and the Pareto densities $r(t) = \beta(1+t)^{-\beta-1}$,
$\beta>0$.

The discontinuous PGIDs admit an equally explicit description. By
Proposition \ref{p:no-atom}, $\rho(\set{0})>0$ precisely when $L$ is a
compound Poisson process, say with rate $c$ and jump distribution
$\pi$ on $(0,\infty)$, in which case $\Psi = 1 + c\prn{ 1-\hat{\pi} }$. Writing
$p := (1+c)^{-1}$, we get
$\hrho = p/\prn{ 1-(1-p)\hat{\pi} }$, i.e.,
$\rho = \sum_{k \geq 0}p(1-p)^{k}\,\pi^{*k}$, with $\rho(\set{0}) = p$.
Thus the discontinuous PGIDs are exactly the laws of
$\sum_{i=1}^{N}Y_{i}$, where $Y_{1},Y_{2},\dots$ are i.i.d.\ and strictly
positive, and $N$ is an independent geometric random variable on
$\set{0,1,2,\dots}$ (the empty sum is interpreted as $0$).
For $c = 0$, i.e., $L \equiv 0$, we get $\rho = \delta_{0}$. For
$\pi = \delta_{1}$ this is the geometric distribution on
$\set{0,1,2,\dots}$, and for $\pi$ exponential with rate $\theta$, it
is $\rho = p\,\delta_{0} + (1-p)\,\rho_{1}$, where $\rho_{1}$ is
exponential with rate $p\theta$.

On the other hand, not all probability measures on $\sp$ are
PGID. For $a>0$, the measure $\rho=\delta_a$ is not PGID, since $1/\hrho(\ld) - 1 = e^{a\ld}-1$ is convex
rather than concave. Less obviously, the uniform distribution
on $[0,1]$, used in Remark \ref{r:cont-disc}(\ref{i:pgid-needed}), is not PGID,
either. If it were, the function $1/\hrho(\ld) - 1 =
\ld/(1-e^{-\ld}) - 1$ would be a Bernstein function, hence concave. It
is not, because its second derivative at $\ld = 0$ equals $1/6 > 0$.
\end{remark}

The correspondence also has a potential-theoretic form,
which explains the title of this subsection.
Let the operator $U$ be as defined in \eqref{conv}, and $L$, $E$ and
$K$ as above. Then
\begin{align}
  \label{pot-op}
  U \vp(t) = \ee{
  \vp(t+ L_E)} =\bE^t\bkt{  \tint_{\tm} \vp(K_{\tau})\, \dtau}.
\end{align}
Therefore, $U$ is the \emph{potential
operator}, i.e., the $0$-resolvent operator, associated to $K$.
Furthermore, $\rho$ is the 
potential measure of $K$, namely $\rho(A) =
\bE^0[  \int_{\tm}\ind{A}(K_{\tau})\,\dtau] $ for $A \in
\sB(\sp)$. Equivalently, $\rho$ is the $1$-resolvent measure of $L$.

We record the following direct consequence of the monotonicity of $L$
for later use:
\begin{lemma}
  \label{l:one-sided}
  Let $I \subseteq \sp$ be a half-line, i.e., $I = (T,\infty)$ or
  $I = [T,\infty)$ for some $T \geq 0$, and let $\vp \in \B(\sp)$
  vanish on
  $I$. Then so do $U \vp$, $\sA \vp$ for $\vp \in \sD(\sA)$, and
  $P_{\tau}\vp$
  for every $\tau \in \tm$.
\end{lemma}

\section{The superprocess representation and its stochastic equations}
\label{S:repr}
This section represents $\xi$ as the occupation measure of a
Dawson-Watanabe superprocess, derives from this representation
existence, basic properties and continuity in the parameters, and
extracts from it the
\emph{vertical noise} $N$ and the stochastic equations it drives.

\subsection{Feller random measures as occupation measures}
The central object of this section is the Dawson-Watanabe superprocess $\sfamtau{X_{\tau}}$ with spatial
motion $K$ and the critical branching mechanism $\ld \mapsto
\ld^{2}/2$. Informally, $X_{\tau} \in \Mlf(\sp)$ is a cloud of
infinitesimal particles at (horizontal) time $\tau \in \tm$. Each
particle moves through $\sp$ (vertical time) as an independent
copy of $K$ and branches critically. Equivalently, $X$ has the
law of the superprocess with the conservative spatial motion $L$ and
the subcritical branching mechanism $\ld \mapsto \ld + \ld^{2}/2$.
Appendix \ref{S:spier} contains a short
primer on superprocesses, with the formal definition, the
properties used in the paper, and the accompanying notation. 

\begin{theorem}
\label{t:main}
Let $\rho$ be PGID, with $L$ and $K$ related to $\rho$ as in Lemma \ref{l:pgid}, and let
$\sfamtau{X_{\tau}}$ be a
Dawson-Watanabe superprocess with spatial motion $K$ and
branching mechanism $\ld \mapsto \ld^{2}/2$.
If $X$ is started from $X_0=\mu \in \Mlf(\sp)$, then the
occupation measure
\begin{align}\label{occup}
  \xi(A) = \int_{\tm}X_{\tau}(A)\,\dtau \efor A \in \sB(\sp)
\end{align}
is a
$(\rho,\mu)$-Feller random measure.
\end{theorem}
\begin{proof} In Appendix \ref{s:prf-main}.
\end{proof}
\begin{remark}
\label{r:cluster}
Theorem \ref{t:main} is a continuum version of the cluster, or
Hawkes-Oakes, construction, in which a Hawkes process is built
generation by generation. The superprocess $X_{\tau}$ is ``generation
$\tau$'', recorded as a measure on $\sp$, and \eqref{occup}
superposes the generations. In this picture, $L_{\tau}$ is the
displacement of a descendant $\tau$ generations down from its
ancestor, the killing that turns $L$ into $K$ is the subcriticality,
and $\rho$ is the normalized mean measure of a single cluster.
Conversely, discretizing $\tau$ in steps of length $\eps$ gives the
Hawkes process $H_{\eps}$ with immigration intensity $\eps^{-1}\mu$
and excitation measure the law of $K_{\eps}$, of total mass
$e^{-\eps}<1$. Conditioning on the offspring of an ancestor gives
$\ee*{ e^{-\ab{\eps^{2}H_{\eps},f}} } = e^{-\ab{\mu,v_{\eps}}}$, where
$v_{\eps}$ solves an equation that is, formally, a one-step
discretization of \eqref{riccati}. We do
not pursue the convergence as $\eps \to 0$ here.
\end{remark}

\begin{definition}
\label{d:occ-form}
  A $(\rho,\mu)$-Feller random measure is \define{in occupation
  form} if it is given by \eqref{occup} under $\bP^{\mu}$, where $X$ is
  the canonical process on $D_{\Mlf(\sp)}$ and $\sfammu{\bP^{\mu}}$ is the
  Markov family of subsection \ref{s:construct} and Remark \ref{r:lf-ext}.
\end{definition}

Whenever $\xi$ is in occupation form, we use without further mention
the objects built from $X$: the martingale measure $M$ of subsection
\ref{s:mart}, the horizontal filtration $\sfamtau{\sF_{\tau}}$, and
the vertical noise $N$ and the vertical filtration $\prfis{\sV_{t}}$,
defined in \eqref{noise} and \eqref{vfilt} below. Here $M$ and $N$ are
understood as versions chosen under each $\bP^{\mu}$. When
additional randomness is needed, we work
on a product extension of the canonical
space.

\smallskip

\begin{corollary}[Existence]
\label{c:exists}
A $(\rho,\mu)$-Feller random measure exists, and can be realized
in occupation form, for each PGID probability
measure $\rho$ and each $\mu \in \Mlf(\sp)$.
\end{corollary}
\begin{proof}
  Immediate from Theorem \ref{t:main}, since a superprocess
  $\sfamtau{X_{\tau}}$ with the spatial motion $K$ can be started from
  any $\mu \in \Mlf(\sp)$ by Remark \ref{r:lf-ext}.
\end{proof}

\subsection{Basic properties and continuity in the parameters}
\label{s:basic}
With existence at hand, we list some basic properties used in the
sequel. Most of them were established in \cite{PacZit25} in the
continuous case.
\begin{proposition}
\label{p:basic}
Let $\rho$ be PGID, let $\mu \in \Mlf(\sp)$, let $\xi$ be the
$(\rho,\mu)$-Feller random measure, and let
$\vp,\vp' \in \Bc(\sp)$. Then
\begin{enumerate}
  \item \label{i:moments}
    $\ee{ \ab{\xi,\vp} } = \ab{\mu, U\vp}$ and $
    \Cov\prn{ \ab{\xi,\vp},\ab{\xi,\vp'} } =
    \ab{ \mu,\, U\prn{ U\vp \cdot U\vp' } }$.
    In particular, $\ee{\xi} = \rho * \mu$.
  \item \label{i:sum} If $\xi_{1}$ and $\xi_{2}$ are independent Feller random
    measures with parameters $(\rho,\mu_{1})$ and $(\rho,\mu_{2})$, then
    $\xi_{1} + \xi_{2}$ is a $(\rho,\mu_{1}+\mu_{2})$-Feller random
    measure.
  \item \label{i:mass-ig} If $\mu$ is finite, then the total mass $\xi(\sp)$ of $\xi$
    has the
    inverse-Gaussian distribution with mean and
    variance $\mu(\sp)$, independently of the choice of $\rho$, i.e.,
    \begin{align}\label{mass}
      \ee{ e^{-\ld\, \xi(\sp)} } =
      e^{-\mu(\sp)\prn{\sqrt{1+2\ld}-1}} \eforall \ld \geq -\tot .
    \end{align}
  \item \label{i:locality}
    For $T \in \sp$, the
    distribution of the restriction of $\xi$ to $[0,T]$ depends on
    $\mu$ and $\rho$
    only through their restrictions to $[0,T]$.  In particular,
    \begin{align}\label{expmom}
      \ee{ e^{\ld \, \xi[0,T]} } \leq e^{\mu[0,T]}
      \eforall \ld \leq \tot .
    \end{align}
  \item \label{i:diffuse} If $\rho(\set{0})=0$ then 
    $\xi$ is diffuse, i.e., almost surely
    $\xi(\set{t}) = 0$ for all $t \geq 0$.
\end{enumerate}
\end{proposition}
\begin{proof} In Appendix \ref{s:prf-basic}.
\end{proof}

\begin{proposition}[Continuity in the parameters]
\label{p:stability}
Let $\rho_{n}$ and $\rho$ be PGID with $\rho_{n} \to \rho$ weakly,
let $\mu_{n}, \mu \in \Mlf(\sp)$ with $\mu_{n} \to \mu$ vaguely,
and let 
$v_{n}[f]$ and $v[f]$ denote the solutions of \eqref{riccati} for
$\rho_{n}$ and $\rho$. 
Then $v_{n}[f] \to v[f]$ uniformly on $\sp$ for
every nonnegative $f \in \Cc(\sp)$, and the $(\rho_{n},\mu_{n})$-Feller
random measures converge in distribution, with respect to the vague
topology, to the $(\rho,\mu)$-Feller random measure.
\end{proposition}
\begin{proof}
In Appendix \ref{s:prf-stability}.
\end{proof}

If $\rho_{n}$ converges only vaguely, part of its mass may escape to
infinity, which,  by Remark
\ref{r:cont-disc}(\ref{i:scaling}), only rescales the limit.
Proposition \ref{p:stability} and the following corollary are used
in Examples \ref{e:delta}, \ref{e:lattice} and \ref{e:exp-disc}.
\begin{corollary}
\label{c:vague-stability}
Let $\rho_{n}$ be PGID with $\rho_{n} \to \rho_{\infty}$ vaguely,
and let $\mu_{n}, \mu \in \Mlf(\sp)$ with $\mu_{n} \to \mu$ vaguely.
If $\rho_{\infty} = 0$, the $(\rho_{n},\mu_{n})$-Feller random measures
converge in distribution to $0$. Otherwise, $\rho_{\infty} = c\,\rho_{0}$
for some $c \in (0,1]$ and a PGID $\rho_{0}$, and the
$(\rho_{n},\mu_{n})$-Feller random measures converge in distribution,
with respect to the vague topology, to $c^{2}\xi_{0}$, where $\xi_{0}$
is the $(\rho_{0},\mu/c)$-Feller random measure.
\end{corollary}
\begin{proof}
In Appendix \ref{s:prf-vague-stability}.
\end{proof}

\begin{remark}
\label{r:stability}
The convergence in Proposition \ref{p:stability} and Corollary
\ref{c:vague-stability} can also be read off the distribution
functions $t \mapsto \nu[0,t]$, extended by $0$ to $(-\infty,0)$ and
viewed as elements of $D_{\bR}$. For nondecreasing functions,
convergence in the Skorokhod $M_{1}$ topology is pointwise convergence
at the continuity points of the limit \cite[Corollary 12.5.1]{Whi02}.
Hence convergence in distribution of Feller random measures is
equivalent to that of their distribution functions in $M_{1}$, and,
when the limit is a.s.\ diffuse, in the Skorokhod $J_{1}$ topology.
\end{remark}

\subsection{Vertical noise as an isometry}
\label{s:noise}
Theorem \ref{t:main} exhibits the Feller random measure $\xi$ as a
``marginal'' of the superprocess $X$: the parameter $t$,
which plays the role of space for $X$, becomes time for
$\xi$, while the branching time $\tau$ gets integrated away. The
\define{vertical noise} $N$ is obtained from the martingale measure $M$
associated to $X$ in the same way, namely
\begin{align}
\label{noise}
N(\vp) = \int_{\sp \times \tm}\vp(t)\,
M\prn{ \dt,\dtau }
\efor \vp \in \Ltn. 
\end{align}
By Tonelli's theorem,
$\vp \in \Ltn$ if and only if
  $\ee*{ \int_{\sp}\vp^{2}(t)\,\xi\prn{ \dt } } < \infty$. By
\eqref{mart-bracket} and \eqref{occup}, this bounds $\ee*{ M_{\tau}(\vp)^{2} }$ uniformly in
$\tau$, so \eqref{noise} is the terminal value of an
$\Ltwo$-bounded martingale. We abbreviate
the norm of $\Ltn$ by
$\norm{\vp}_{N} := \prn*{ \int_{\sp}\vp^{2}\,\diffd\ee*{\xi} }^{1/2}$,
so that the Walsh isometry reads
$\ee*{ N(\vp)^{2} } = \norm*{\vp}_{N}^{2}$.

\subsection{Stochastic equations in weak form}
\label{s:weak}
Continuing in the spirit of Theorem \ref{t:main}, we ``marginalize''
the martingale problem \eqref{mart-prob} to obtain the weak form of
the equation satisfied by $\xi$, in two equivalent versions.

\begin{proposition}
\label{p:weak}
Let $\rho$ be PGID, let $\mu \in \Mlf(\sp)$, and let $\xi$ be the
$(\rho,\mu)$-Feller random measure in occupation form. Then:
\begin{align}\label{weak-gen}
   \ab{\mu, \vp} + \ab{\xi, \sA \vp}
  + N(\vp)  &= 0 \eas \eforeach
  \vp \in \sD(\sA) \cap \Cc(\sp), \eand \\
  \label{weak-res}
   \ab{\mu,U\psi} 
 - \ab{\xi,\psi}
  + N( U\psi)  &= 0 \eas
  \eforeach \psi \in \Cc(\sp).
\end{align}
\end{proposition}

\begin{proof} In Appendix \ref{s:prf-weak}.
\end{proof}

\subsection{Vertical noise as a martingale}
\label{s:noise-mg}
To use $N$ as the driving noise in Proposition \ref{p:weak}, we
identify it as a martingale with respect to a suitable filtration.

Equation \eqref{dw-filt} in Appendix \ref{S:spier} defines the
filtration $\sfamtau{\sF_\tau}$, which we call the
\define{horizontal filtration}, as the universally
completed natural filtration of $X$. To define a
filtration compatible with $\xi$, we need a ``vertical''
construction: for $t \geq 0$, let
\begin{align*}
\sV^0_t
:= \sigma\Bprn{X_\tau(A):\tau\in\tm,\,A\in\sB[0,t]}.
\end{align*}
With
$\sZ^{\mu}$ denoting the family of $\bP^{\mu}$-null sets,
the \define{vertical filtration} $\prfis{\sV_t}$ is defined
as the right-continuous augmentation of the universal completion of
$\prfis{\sV^0_t}$, i.e.,
\begin{align}\label{vfilt}
\sV_t
:= \bigcap_{s>t} \overline{\sV}_s, \ewhere \overline{\sV}_s := \quad
\smashoperator{\bigcap_{\mu\in \Mlf(\sp)}} \quad
(\sV^0_s\vee\sZ^\mu) \efor
t\geq0.
\end{align}
We write $N_{t} := N\prn{ \ind{[0,t]} }$
for the vertical noise accumulated up to the height $t \in \sp$, and
$\Gamma_{t} := \xi[0,t]$ for the distribution function of $\xi$.
The initial value $N_{0}$ of $N$ need not vanish; recall the
conventions of subsection \ref{s:notation} concerning $t=0$.
\begin{proposition}
\label{p:noise-mg}
Let $\rho$ be PGID, let $\mu \in \Mlf(\sp)$, and let $\xi$ be the
$(\rho,\mu)$-Feller random measure in occupation form. Then
$\prfis{N_{t}}$ is a square-integrable mean-zero
    $(\bP^{\mu},\prfis{\sV_{t}})$-martingale, which we always take
    in its \cd{} version,  
    and we have
    \begin{align}\label{stoch-int}
      N(\vp) = \int_0^{\infty} \vp(t) \, \dN_t \efor \vp \in \Ltn.
    \end{align}
    Moreover, the processes
    \begin{align}\label{expmg}
      \prfis{ N_{t}^{2}-\Gamma_t } \eand
      \prfiB{ \Exp{ \ld N_{t} - \tot\ld^{2}\,\Gamma_t} } \efor
      \abs{\ld} \leq 1/4
    \end{align}
    are $(\bP^{\mu}, \prfis{\sV_t})$-martingales.
\end{proposition}
\begin{proof} In Appendix \ref{s:prf-noise-mg}.
\end{proof}

\begin{remark}
\label{r:not-qv}
Even though the processes in \eqref{expmg} are martingales, it
can happen that 
$\xi[0,t]$
is neither the predictable quadratic variation $\ab{N}_{t}$ nor
the optional quadratic variation $\bkt{N}_{t}$ of $N$. By 
Corollary \ref{c:dichotomy} below, these three processes are
a.s.\ different as soon as $\rho(\set{0})>0$ and $\mu \in \Mlf^{*}(\sp)$.
\end{remark}
\subsection{Stochastic equations in strong form}
\label{s:strong}
With $N$ identified as the noise induced by a \cd{}
martingale, we can transform the weak stochastic equation
\eqref{weak-gen} into a stronger, pointwise form.   
Recall that measures on $\sp$ are
extended trivially to $\bR$. Throughout the rest of this subsection,
$\partial$ denotes
the distributional derivative on $\bR$. A $\vp \in \Conec(\sp)$
need not vanish at the origin, so its trivial extension to $\bR$ need
not be $C^{1}$ there. The pairing of $\partial\nu$ with such a $\vp$
is therefore always understood as
$\ab{ \partial\nu,\vp } := -\ab{ \nu,\vp' }$, with $\vp'$ the
derivative on $\sp$.

\smallskip

For $\nu \in \Mlf(\sp)$, let $\sA^{\ast} \nu$ be the linear
functional on the space $\Conec(\sp)$ defined by the adjoint relation
\begin{align}
  \label{adjoint}
  \ab{ \sA^{\ast}\nu, \vp } := \ab{ \nu, \sA\vp }
\end{align}
 for all
$\vp \in \Conec(\sp)$. 
The right-hand side is finite because
$\sA$ maps $\sD(\sA) \supseteq \Conec(\sp)$ into $\Cz(\sp)$,
and because $\sA\vp$ vanishes on a half-line by Lemma
\ref{l:one-sided}.
For  $\vp \in \Conec(\sp)$,  we have
\begin{align}\label{kgen}
  \sA\vp(t) = b\,\vp'(t) +
  \int_{0}^{\infty}\vp'(t+u)\prn{ 1 + \oPi(u) }\,\du
  \ewhere
  \oPi(t) := \Pi(t,\infty),
\end{align}
with the killing at rate $1$ accounting for the constant in
$1 + \oPi$.
Therefore, if we set
\begin{align}\label{sonin} k(\dt) := b\,\delta_{0}(\dt) +
  \prn{ 1 + \oPi(t) }\,\dt,
\end{align}
we immediately get the following representation, which puts $\sA^*$ in
``divergence form'':
\begin{lemma}
\label{l:adjoint-k}
For every $\nu \in \Mlf(\sp)$, $k * \nu \in \Mlf(\sp)$,
and, interpreted as
linear functionals on $\Conec(\sp)$, we have
\begin{align}\label{divform}
  -\sA^{\ast}\nu = \partial\prn{ k * \nu }.
\end{align}
\end{lemma}
\begin{remark}
\label{r:symbol}
Taking Laplace transforms in \eqref{sonin} gives
  $\hat k(\ld) = \Psi(\ld)/\ld$,
so that \eqref{divform} reads
  $\widehat{ \sA^{\ast}\nu } = -\Psi\,\hat\nu$ for every $\nu$
whose Laplace transform converges. Thus $\sA^*$ is a constant-coefficient
pseudodifferential operator with symbol $-\Psi$.
Combined with $\hrho = 1/\Psi$ from \eqref{rho-psi}, this shows that
the pair $(k,\rho)$
satisfies the \emph{Sonin relation} (see \cite{Son84})
\begin{align}\label{sonin-rel}
\hat k(\ld)\,\hat\rho(\ld) = 1/\ld, \eie
  (k * \rho)(\dt) = \dt.
\end{align}
Equivalently, $k$ is the \emph{resolvent of the first kind} of
$\rho$, in the sense of \cite[Chapter 5]{GriLonSta90}.
\end{remark}
Integrating \eqref{weak-gen} by parts gives the following pointwise
form. Its second equation, \eqref{volterra}, generalizes the
Volterra equations used to define classical and rough Feller (CIR)
processes \cite{AbiLarPul19, Abi21, ElERos19}. It follows from
\eqref{strong-k} by convolution with $\rho$, using \eqref{sonin-rel} and
the stochastic Fubini theorem.
\begin{proposition}[Stochastic equations: the strong form]
\label{p:strong}
Let $\rho$ be PGID, let $\mu \in \Mlf(\sp)$, let $\xi$ be the
$(\rho,\mu)$-Feller random measure in occupation form, and let $k$ be given by
\eqref{sonin}. Then, almost surely,
the measure $k*\xi$ is absolutely continuous and
\begin{align}
  \label{strong-k}
  \prn{ k * \xi }(\dt) &= (N_t + \mu[0,t])\, \dt, \\
  \label{volterra}
  \xi[0,t] &= \int_0^t \rho[0,t-s]\, \dN_s + (\rho*\mu)[0,t]
  \eforall t \in \sp .
\end{align}
\end{proposition}

\section{The continuous case}
\label{S:cont}
The noise behaves quite differently in the continuous and the discontinuous cases.
In this section we treat the continuous case, where $\rho(\set{0})=0$,
equivalently $b>0$ or $\Pi(0,\infty)=\infty$.

We first describe the noise (subsection \ref{s:cont-noise}) and show
that, together with its quadratic variation, it characterizes
continuous Feller random measures (subsection \ref{s:mart-char}). The
rest of the section concerns the density $Z$: when it exists
(subsection \ref{s:density}), which equations it satisfies (subsection
\ref{s:st-eq-den}), how regular it is (subsection \ref{s:reg}), and
in what sense this regularity is sharp (subsection \ref{s:sharp}).
Subsection \ref{s:examples} collects examples.

\subsection{The structure of the noise}
\label{s:cont-noise}
We first show that a continuous Feller random measure comes with
a continuous noise process:
\begin{proposition}
  \label{p:noise-cont}
  Let $\rho$ be PGID with $\rho(\set{0})=0$, let
  $\mu \in \Mlf(\sp)$, and let $\xi$ be the $(\rho,\mu)$-Feller random
  measure in occupation form. Then the
  $(\bP^{\mu},\prfis{\sV_{t}})$-martingale $\prfis{N_{t}}$ of
  Proposition \ref{p:noise-mg} is continuous, with $N_{0}=0$ and
  quadratic variation $\bkt{N}_{t} = \Gamma_t$.
\end{proposition}
\begin{proof}
  In Appendix \ref{s:prf-noise-cont}.
\end{proof}
The theorem of 
Dambis, Dubins and Schwarz \cite[Chapter V, Theorem 1.7]{RevYor99} now yields
the following:
\begin{corollary}
\label{c:dds}
Under the assumptions of Proposition \ref{p:noise-cont}, the
process $\Gamma$ is continuous 
and $\ab{N} = \bkt{N} = \Gamma$. Moreover, there exists, on a possibly
extended probability space, a Brownian motion $W$ such that $N_{t} =
W_{\Gamma_{t}}$ for all $t \in \sp$, a.s.
\end{corollary}

\begin{remark} \label{r:intrinsic-cont} The occupation form is
  only a convenient device here. By \eqref{strong-k}, $N_{t} +
  \mu[0,t]$ is a density of $k*\xi$. Since $N$ is continuous,
  $N_{t} = \lim_{h \downto 0} h^{-1}(k*\xi)(t-h,t] - \mu[0,t)$, which
  depends only on the restriction of $\xi$ to $[0,t]$, so that $N$ is a
  functional of $\xi$. Proposition \ref{p:noise-cont} and
  Corollary \ref{c:dds} therefore hold with $\prfis{\sV_{t}}$
  replaced by the natural filtration of $\Gamma$, by the tower
  property. By the transfer theorem \cite[Theorem 8.17]{Kal21},
they also hold for every $(\rho,\mu)$-Feller random measure.
\end{remark}

\subsection{A martingale characterization}
\label{s:mart-char}
A first consequence of Proposition \ref{p:noise-cont} is that
the stochastic equation \eqref{strong-k}, with $N$ replaced by a
continuous local martingale with quadratic variation $\Gamma$, characterizes
continuous Feller random measures:

\begin{theorem}[The martingale characterization]
\label{t:mart-char}
Let $\rho$ be a PGID probability on $\sp$,
and let $\mu \in \Mlf^{*}(\sp)$. For a random measure $\xi$ on $\sp$
the following are equivalent:
\begin{enumerate}
  \item \label{i:frm} $\rho(\set{0})=0$ and $\xi$ is a $(\rho,\mu)$-Feller
    random measure.
\item \label{i:strong-form} There exist a 
filtration $\prfis{\sH_{t}}$ on the underlying probability space to
which the increasing process $\prfis{\Gamma_t}$ is adapted, and a
continuous $\prfis{\sH_{t}}$-local martingale $\eta$ with $\eta_{0}=0$,
such that, almost surely, 
\begin{align}\label{bracket}
  \bkt{\eta}_{t} = \Gamma_t \eforall t \in \sp,
\end{align}
and
\begin{align}\label{char-strong}
  \prn{ k * \xi }(\dt) = \prn{ \eta_{t} + \mu[0,t] }\,\dt,
\end{align}
where $k$ is defined by \eqref{sonin}.
\item \label{i:weak-form} The same as (\ref{i:strong-form}), but with \eqref{char-strong} replaced by the weak form
\begin{align}\label{char-weak}
  \ab{\mu,\vp} + \ab{\xi,\sA\vp} + \int_{\sp} \vp\, \diffd \eta = 0 \eas \eforeach
  \vp \in \sD(\sA) \cap \Cc(\sp).
\end{align}
\end{enumerate}
\end{theorem}
\begin{proof} In Appendix \ref{s:prf-mart-char}.
\end{proof}

\begin{remark}
\label{r:sing-noise}
Nothing in Theorem \ref{t:mart-char} requires that $\xi$ has a density.
As in Corollary \ref{c:dds}, there is, on a possibly extended
probability space, a Brownian motion $W$ with $\eta_{t} = W_{\Gamma_{t}}$,
so that \eqref{char-strong} and its convolution with $\rho$ (compare
\eqref{volterra}) read
\begin{align}
  \label{clock-k}
  \prn{ k*\xi }(\dt) &= \bprn{ W_{\Gamma_t} + \mu[0,t] }\,\dt, \\
  \label{clock-volterra}
  \Gamma_t &= \int_{0}^{t}\rho[0,t-s]\,\diffd W_{\Gamma_s}
             + (\rho*\mu)[0,t] .
\end{align}
When $\xi$ has a density $Z$, we have $\Gamma_{t} = \int_{0}^{t}Z_{s}\,\ds$
and $W_{\Gamma_{t}} = \int_{0}^{t}\sqrt{Z_{s}}\,\dB_{s}$ for a Brownian
motion $B$, which returns Theorem \ref{t:Z-eqs}. Equations of the form
\eqref{clock-volterra} appear in the fractional and hyper-rough
literature \cite{JusRos20, Abi21, AbiAtt25, AbiAttSoj26}, also without
a density \cite{AbiAttSoj26}. Theorem \ref{t:mart-char} shows that,
here, they characterize Feller random measures. 
A characterization in Volterra form, with
$\eta = W_{\Gamma}$ for a given Brownian motion $W$, is the
content of \cite[Theorem
2.14]{Xu24a}. Theorem \ref{t:mart-char} allows an arbitrary continuous
local martingale and is stated in divergence and weak form. Since it
needs no density, it covers the hyper-rough range without separate
treatment, whereas \cite{Abi21} formulates the theory there in terms of
the integral of the missing density.
\end{remark}

\subsection{Existence of a density}
\label{s:density}
A measurable process
$\prfis{Z_t}$ such that $\xi(\dt) = Z_t\, \dt$ a.s.\ is called a
\define{density} of $\xi$. We call it
\define{square integrable} if $\ee*{\int_{0}^{T}Z_{t}^{2}\,\dt}<\infty$
for every $T>0$.
We first investigate when a density exists,
and then derive further properties of $\xi$ when it does.

By Proposition \ref{p:basic}(\ref{i:moments}) we have $\ee{\xi} = \rho*\mu$,
which must therefore
be absolutely continuous whenever $\xi$ has a density. In
  particular, $\rho$ itself is absolutely continuous when
$\mu$ is an atom (or purely atomic). The case of a more general
$\rho$ is more
complex and not particularly important for the applications we
have in mind, so the results below assume that $\rho$ admits a density,
except for Corollary \ref{c:holder}, which derives it.

The main result in this direction is Proposition \ref{p:density-crit}
below. It gives a necessary and sufficient condition for the
existence of a square-integrable density, and its proof rests on
\eqref{weak-res}:
\begin{proposition}
\label{p:density-crit}
Let $\rho$ be PGID with $\rho(\dt) = r(t)\,\dt$, let
$\mu \in \Mlf^{*}(\sp)$, and let $\xi$ be
a $(\rho,\mu)$-Feller random measure.
Then
    $\xi$ admits a square-integrable density $\prfis{Z_t}$ if and only if
\begin{align}
  \label{r-sq-int}
  \int_0^T r(t)^2\, \dt < \infty \eforall T>0.
\end{align}
In that case, if $\xi$ is in occupation form, we have
\begin{align}
  \label{Z-rep}
  Z_t = \int_{[0,t]} r(t-s)\, \mu(\ds)
  + \int_0^t r(t-s)\, \dN_{s},
\ \dt \otimes \dbP\text{-a.e.}
\end{align}
\end{proposition}
\begin{proof} In Appendix \ref{s:prf-density-crit}.
\end{proof}
\begin{remark}\leavevmode
\label{r:density-conj}
\begin{enumerate}
  \item 
The ``if'' half of Proposition \ref{p:density-crit} is close to
\cite[Corollary 2.16]{Xu24a}. There, a locally square-integrable
density of $\rho$ and a differentiable $(\rho*\mu)[0,\cdot]$ yield a
density solving \eqref{Z-resolvent}. Proposition \ref{p:density-crit}
adds the converse, drops the differentiability assumption, and gives
the explicit form \eqref{Z-rep}. 
\item
Proposition \ref{p:density-crit} concerns only square-integrable
densities. We conjecture that
there are no others: for $\mu \in \Mlf^{*}(\sp)$, $\xi$ admits a
  density \emph{only if} $\rho(\dt) = r(t)\,\dt$ with $r$ satisfying
\eqref{r-sq-int}, so that every density is automatically square
integrable. A precedent is the occupation measure of
super-Brownian motion in $\bR^{d}$, which has a density exactly when $d
\leq 3$, that is, exactly when the Green kernel is locally square
integrable (see \cite{Sug89, Per02}).

The condition \eqref{r-sq-int} is also the dividing line
between the $\Ltwo$-theory of stochastic Volterra equations
\cite{AbiLarPul19} and the $\Lone$-theory \cite{Abi21}, in which
only the integrated solution exists.
\end{enumerate}
\end{remark}

\subsection{Stochastic equations for the density}
\label{s:st-eq-den}
The general stochastic equations of Propositions \ref{p:weak} and 
\ref{p:strong} start to resemble the stochastic differential
equations for the classical Feller diffusion and its rough
counterpart when the Feller random measure admits a density. 

Before we state our main theorem, we introduce the 
occupation measure $U^L$ of the subordinator $L$:
\begin{align}\label{U-def}
  U^{L}(A) := \ee{ \textstyle\int_{\tm} \ind{A}(L_{\tau})\,\dtau }
  \efor A \in \sB(\sp).
\end{align}
The Laplace exponent of $L$ is $\Psi-1$, so that
$\widehat{U^L} = 1/(\Psi -1 )$. It follows that $U^L$ is  locally
finite and  that 
\begin{align}
  \label{U-renewal}
  U^L = \textstyle\sum_{n\geq 1} \rho^{*n},
\end{align}
where $\rho^{*n}$ denotes the $n$-fold self-convolution. This
representation implies that $U^L$ has a density, denoted by $\ell$,
whenever $\rho$ has one. 
\begin{theorem}[Stochastic equations for $Z$]
\label{t:Z-eqs}
Let $\rho$ be PGID with density $r$, let $\mu \in \Mlf(\sp)$, and
let $\xi$ be a $(\rho,\mu)$-Feller random measure with density $Z$. Then
there exists a Brownian motion $B$, on a
possibly extended probability space, such that, almost surely,
for almost every $t \geq 0$, we have
\begin{align}
  \label{Z-sde}
  b Z_t  &=    \int_{[0,t]} \Bprn{  \mu(\ds) - (1+\oPi(t-s)) Z_s\, \ds }
  +  \int_0^t \sqrt{Z_s}\, \dB_{s} .
\end{align}
  If, moreover,
$r$ satisfies \eqref{r-sq-int}, then, in the same sense,
\begin{align}
  \label{Z-resolvent}
  Z_t & =  \int_{[0,t]}   r(t-s)\, \mu(\ds) + \int_0^t r(t-s)
  \sqrt{Z_s}\, \dB_s, \eand \\
  \label{Z-potential}
   Z_{t}  &=  \int_{[0,t]} \ell(t-s)\, \Bprn{ \mu(\ds)
   -  Z_{s}\, \ds }
   + \int_{0}^{t} \ell(t-s) \sqrt{Z_{s}}\, \dB_{s}.
\end{align}
\end{theorem}
\begin{proof} In Appendix \ref{s:prf-Z-eqs}.
\end{proof}

\begin{remark}[A generalized fractional SDE]
\label{r:gen-frac-sde}
The operator $\sA^{L} := \sA + I$ is the generator of the
subordinator $L$ (without killing). With $k^{L}(\dt) := b\,\delta_{0}(\dt)
+ \oPi(t)\,\dt$, so that $k = k^{L} + \Leb$, equation \eqref{Z-sde}
reads
\begin{align*}
  (k^{L}*Z)(t) = \int_{[0,t]}\Bprn{ \mu(\ds) - Z_{s}\,\ds }
  + \int_{0}^{t}\sqrt{Z_{s}}\,\dB_{s}.
\end{align*}
By Lemma \ref{l:adjoint-k}, $\partial(k^{L}*\nu) = \partial(k*\nu) -
\nu = -(\sA^{L})^{*}\nu$ for $\nu \in \Mlf(\sp)$, and, by Remark
\ref{r:symbol}, $-(\sA^{L})^{*}$ has symbol $\Psi-1$, the Laplace
exponent of $L$. Informally, therefore, $Z$ solves
\begin{align*}
  d^{L}Z_{t} = \mu(\dt) - Z_{t}\,\dt + \sqrt{Z_{t}}\,\dB_{t},
  \ewhere d^{L}Z_{t} := -((\sA^{L})^{*}Z)(\dt).
\end{align*}
When $\Pi = 0$, we have $d^{L}Z_{t} = b\,\dZ_{t}$, and this is the
classical Feller SDE \eqref{cir-sde}. In Example \ref{e:rough-cir}, we
have $k^{L}*Z = I^{1-\al}Z$, and $d^{L}$ is the fractional differential
$d^{\al}$ discussed there. Like the latter, this SDE is informal, since
\eqref{Z-sde} holds only in integrated form, for a.e.\ $t$. Operators
of the form $\partial(k^{L}*\cdot)$ are known as \emph{general fractional
derivatives} associated with the Bernstein function $\Psi-1$
(see, e.g.,  \cite{Koc11}, \cite{Toa15} and \cite{Luc21}).
\end{remark}

\begin{remark}[The fractional form in the literature]
\label{r:frac-lit}
Equations similar to \eqref{frac-int} in Example \ref{e:rough-cir}
below, though not always in this
form, have been considered in the literature on affine Volterra
processes. The passage from \eqref{rough-potential} to
\eqref{frac-int} rests on the identity $\sL*(X-X_{0}) = Y-Y_{0}$ for
$X-X_{0} = \sK*\diffd Y$, where $\sL$ is the resolvent of the first kind
of the kernel $\sK$ and $Y$ is a continuous semimartingale
\cite[Lemma 2.6]{AbiLarPul19}, \cite[Lemma B.1]{AbiElE19}. For
$\sK(t) = t^{\al-1}/\Gamma(\al)$, convolution with $\sL$ is $I^{1-\al}$.
A version of the identity for semimartingales with jumps is
\cite[Proposition 4]{BonLivPul24}, and in \cite[Lemma 2.5]{WanWuZhu26},
it is applied to the variance process of the rough Hawkes-Heston
model, yielding an equation of the form \eqref{frac-int} whose
right-hand side is a semimartingale with jumps.
We also note that 
fractional stochastic differential equations of Caputo type with
Lipschitz coefficients have been studied in \cite{DoaHuoKloTua18}.
\end{remark}

\begin{remark}[Setting the initial condition]
  \label{r:init-cond}
The second moment 
$\prn{r^{2}*(r*\mu)}(t)$
of the stochastic term in \eqref{Z-resolvent} is bounded from above by 
$\norm{r*\mu}_{\infty,t}\norm{r}^{2}_{\Ltwo[0,t]}$, which
vanishes as $t \downto 0$ whenever $r*\mu$ is locally bounded.
The initial value $Z_{0+}$, understood as the
$\Ltwo$-limit of $Z_{t}$ as $t \downto 0$, is therefore $(r*\mu)(0+)$,
whenever the latter limit exists. Under the conditions of
Proposition \ref{p:Z-modulus} below, this is of no use if one is trying to
set the initial value of $Z$ via $\mu$. Indeed, in that case
$r*m$, with $m$ the density of $\mu$, is
continuous and vanishes at the origin, so $Z_{0+}=0$ no matter what $\mu$ is.
To reach a prescribed
$z_{0}\geq 0$ one has to leave those conditions, and set
$\mu = \mu_{0}+\mu'$, where
\begin{align*}
  \mu_{0}(\ds) := z_{0}\bprn{ b\,\delta_{0}(\ds) + \oPi(s)\,\ds }
  = z_{0}\bprn{ k(\ds) - \ds }
\end{align*}
with $k$ as in \eqref{sonin}, and $\mu' \in \Mlf(\sp)$
such that
\begin{align}
  \label{mobed}
  r*\mu' \text{ is locally bounded, and } (r*\mu')(0+) = 0.
\end{align}
Indeed, the Sonin relation \eqref{sonin-rel}
  gives $(\rho*k)(\dt) = \dt$, while $\rho*\Leb$ has density
  $t \mapsto \rho[0,t]$, so that $\rho*\mu_{0}$ has density
  $g_{0}(t) = z_{0}\bprn{ 1 - \rho[0,t] }$.
This function is non-increasing, bounded by $z_{0}$, and satisfies
  $g_{0}(0+) = z_{0}$ because $\rho$ has no atom at the origin.
  Consequently $r*\mu = g_{0}+r*\mu'$ is locally bounded, with
  $(r*\mu)(0+) = z_{0}$, so that $Z_{0+} = z_{0}$, as required.
  A similar construction in the rough CIR case is described
  in \cite[Definition 2.3, Remark 2.4]{ElERos19}.
  The corresponding device on the side of stochastic Volterra
  equations is the class of admissible input curves of
  \cite[Equations (2.4)--(2.5)]{AbiElE19}. It is imposed there to secure the
  nonnegativity that is automatic here, and it is formulated in terms of the
  resolvent of the first kind of the kernel --- that is, since that
  kernel plays the role of $\ell$ in \eqref{Z-potential}, in terms of the
  measure $k^{L} = k - \Leb$ of Remark \ref{r:gen-frac-sde}.

  Both conditions in \eqref{mobed} hold, for
  example, if $r \in \Lpeel(\sp)$
  and $\mu'(\dt) = m'(t)\, \dt$ with
  $m' \in \Lquel(\sp)$, where $p,q$ are
  conjugate exponents. Indeed, H\"older's inequality then gives
  $\abs{ (r*m')(t) } \leq \norm{r}_{\Lpee[0,t]}\norm{m'}_{\Lque[0,t]}$.
\end{remark}

\subsection{Regularity of the density}
\label{s:reg}
When $b>0$, the density is well behaved (Corollary
\ref{c:drift-smooth} below). When $b=0$, it need not even be locally
bounded (Remark \ref{r:atom-blowup}). The results below rely on
\eqref{Z-resolvent}, which does not involve $b$ explicitly.

\smallskip

Let the
\emph{modulus of continuity} $\omega_T$ be defined, for $T, h \geq
0$, by
\begin{align*}
  \omega_{T}(h)^{2} := \norm{r}^2_{\Ltwo[0,h]}+ 
  \int_{0}^{T}\prn{ r(u+h)-r(u) }^{2}\,\du.
\end{align*}

\begin{proposition}
\label{p:Z-modulus}
Let $\rho$ be PGID with an $\Ltwol(\sp)$-density $r$, let
$\mu \in \Mlf^{*}(\sp)$ have an $\Ltwol(\sp)$-density $m$, and
let $\xi$ be a $(\rho,\mu)$-Feller random measure.
Then  $\xi$ admits a density $Z \in
\Ltwol(\sp)$, a.s., such that,
for every $p \geq 2$ and $T>0$, there is a constant $C = C(p,T)
< \infty$ 
such that for all $s,t \in [0,T]$, we have
\begin{align}\label{Z-incr}
  \norm{Z_t}_{\bL^p} \leq C \eand 
  \norm{ Z_{t}-Z_{s} }_{\bL^{p}} \leq
  C\,\omega_{T}\prn{ \abs{t-s} }. 
\end{align}
If $\omega_{T}(h) \leq C_{T}\,h^{\gm}$ for some
$\gm \in (0,1/2]$, every $T>0$ and every $h \in (0,1]$, then $Z$ admits a version whose
paths are locally H\"older continuous of every order $\theta < \gm$.

In the opposite direction, for all $t \geq 0$ and $h>0$,
\begin{align}\label{Z-incr-lb}
  \norm{ Z_{t+h}-Z_{t} }_{\Ltwo} \geq
  \norm{r}_{\Ltwo[0,h]} \,\bprn{ \textstyle\inf_{[t,t+h]} r*m
  }^{1/2}.
\end{align}
\end{proposition}
\begin{proof} In Appendix \ref{s:prf-Z-modulus}.
\end{proof}

Since the bound on $\omega_{T}$ required by Proposition
\ref{p:Z-modulus} is rarely available directly, we recast it, via
Plancherel's theorem and the relation $\hrho = 1/\Psi$, 
as a growth condition on the Laplace exponent $\Psi$. 
\begin{corollary}
\label{c:holder}
Suppose that 
there exist 
$\beta\in(1/2,1]$, $c>0$, and $\ld_{0} > 0$ such that 
\begin{align}\label{psi-power}
  \abs{ \Psi(-iu) } \geq c\,\abs{u}^{\beta}
  \eforall u \in \bR \ewith
  \abs{u}\geq\ld_{0}.
\end{align}
 Then $\rho$
admits a density $r \in \Ltwo(\sp)$,
and there exists a constant $C = C(c,\beta,\ld_{0})$, independent of
$T$ and $h$, such that
\begin{align*}
  \omega_{T}(h) \leq C\,h^{\beta-\tot}
\eforall T>0 \eand h>0.
\end{align*}
If, in addition,
$\mu \in \Mlf^{*}(\sp)$ has an $\Ltwol(\sp)$-density $m$, then the
$(\rho,\mu)$-Feller random measure has a density $Z$ with a
version whose paths are locally
H\"older continuous of every order $\theta<\beta-1/2$.
\end{corollary}
\begin{proof} In Appendix \ref{s:prf-holder}.
\end{proof}

\begin{corollary}[No roughness when there is a drift]
\label{c:drift-smooth}
Let $b>0$, so that $\rho$ has a density bounded by $1/b$ (see \cite[Chapter III]{Ber96}), and let $\xi$ be a $(\rho,\mu)$-Feller
random measure with density $Z$. Then $Z$ is locally
(essentially) bounded and has a \cd{} version, which is continuous if
and only if $\mu$ is diffuse. If, moreover, $\mu$ admits an
$\Ltwol(\sp)$-density, then $Z$ admits a version whose paths are
locally H\"older continuous of every order $\theta<1/2$, irrespective
of $\Pi$.
\end{corollary}
\begin{proof}
Since $b>0$,
the term $(1+\oPi)*Z$ in \eqref{Z-sde} is nonnegative, so 
$0 \leq b Z_{t} \leq  \mu[0,t] +  \int_{0}^{t}\sqrt{Z_{s}}\,\dB_{s}$
for a.e.\ $t$, a.s. It follows that $Z$ is locally
(essentially) bounded, which, in turn, implies that 
$\oPi*Z$  is continuous. Hence, by \eqref{Z-sde},
$Z$ has a \cd{} version, which is continuous if and only if $\mu$ is diffuse.

For the second statement, we have $\Im\Psi(-iu) = -bu -
\int_{(0,\infty)}\sin(ut)\,\Pi(\dt)$. 
The integral is $o(\abs{u})$ as $\abs{u} \to \infty$, by dominated
convergence, since $\abs{\sin(ut)} \leq \min(1,\abs{u}t)$ and
$\Pi$ is a L\'evy measure. 
Hence
$\abs{\Psi(-iu)} \geq b \abs{u}/2$ for $\abs{u}$ large. Thus
\eqref{psi-power} holds with $\beta=1$, and Corollary \ref{c:holder}
applies.
\end{proof}

\begin{remark}[Explosive behavior when $\mu$ is atomic]
  \label{r:atom-blowup}
Without the absolute continuity of $\mu$, the process $Z$ may
behave in an entirely different way from what Proposition \ref{p:Z-modulus} 
describes. Take, for example, $\mu = \delta_{0}$ and
$\Psi(\ld) = 1+\ld^{\al}$ with $\al \in (1/2,1)$, so that $L$ is a
stable subordinator, $\rho$ is a Mittag-Leffler distribution,
and we have
$r(t) \sim \frac{t^{\al-1}}{\Gamma(\al)}$ as $t \downto 0$.
It follows that $r \in \Ltwol(\sp)$ and, therefore, by Proposition
\ref{p:density-crit}, that the density $Z$ exists. On the other hand,
the density $r*\mu = r$ of $\ee{\xi}$ is unbounded around the origin, so
$\ee{Z_{t}} = r(t) \to \infty$ as $t \downto 0$. By
\eqref{Z-resolvent}, we have  
\begin{align*}
  \norm{ \frac{Z_{t}}{r(t)} - 1 }_{\Ltwo}^{2}
  = \oo{r(t)^{2}}\int_{0}^{t}r(t-u)^{2}\,r(u)\,\du
  \sim \frac{\Gamma(2\al-1)}{\Gamma(3\al-1)}\,t^{\al}, \text{ for
a.e. } t \in \sp,
\end{align*}
so that $Z_{t} / t^{\al-1} \to \Gamma(\al)^{-1} \in (0,\infty)$
in $\Ltwo$ along a set of $t$ of full measure. 

To get a pathwise estimate, we start with 
Proposition \ref{p:basic}(\ref{i:moments}), which  gives
$\Var\prn{ \Gamma_t } = \int_{[0,t]}\rho[0,t-u]^{2}\,\rho(\du)
\asymp t^{3\al}$ against $\rho[0,t]^{2} \asymp t^{2\al}$, so that
$\Gamma_t/\rho[0,t] \to 1$ in $\Ltwo$, with the squared $\Ltwo$-distance
of order $t^{\al}$, as above.
Chebyshev's inequality and the Borel-Cantelli lemma, applied along
$t_{n}=q^{n}$ with $q \in (0,1)$, whose $\al$-th powers are summable,
yield $\xi[0,t_{n}] \geq \rho[0,t_{n}]/2$ for all large $n$, a.s.
Since $\esssup_{(0,t_{n})}Z \geq \xi[0,t_{n}]/t_{n} \geq
\rho[0,t_{n}]/(2t_{n}) \asymp
t_{n}^{\al-1}$, we conclude that $\esssup_{(0,\delta)}Z = \infty$ for
every $\delta>0$, a.s.

The same argument
applies at any atom $t_{0}$ of $\mu$: by Proposition
\ref{p:basic}(\ref{i:sum}), we may realize $\xi$ as
$\xi_{1}+\xi_{2}$,
where $\xi_{1}$ is a $(\rho,\mu(\set{t_{0}})\delta_{t_{0}})$-Feller
random measure independent of $\xi_{2}$, so that $Z$ dominates a
density of $\xi_{1}$ almost everywhere, and
$\esssup_{(t_{0},t_{0}+\delta)}Z = \infty$ for every $\delta>0$,
a.s. Thus, if $\mu$ is purely atomic with dense support and $\rho$
is as above,
$Z$ is essentially unbounded on every nonempty interval, a.s.
\end{remark}

\subsection{Sharpness of the regularity}
\label{s:sharp}
This subsection contains two negative results, namely
Propositions \ref{p:max-holder} and \ref{p:not-semimart}. 
Both of them are stated under the following assumption:

\begin{assumption}
\label{a:pi-bound}
The measure $\rho$ is absolutely continuous, $b = 0$, $\xi$ has a
density $Z$, and the jump measure $\Pi$ has a density $\pi$ such that,
for some $\al \in (1/2,1)$ and $C<\infty$,
\begin{align}\label{pi-bound}
  \pi(t) \leq C\,t^{-1-\al} \efor t \in (0,1], \eand
  \sup_{t \geq 1}\pi(t) < \infty .
\end{align}
\end{assumption}

Assumption \ref{a:pi-bound} holds, e.g., for the rough CIR family
(Example \ref{e:rough-cir}), and for its tempered version (Example
\ref{e:temper}) when $\al>1/2$.
When $\mu$ has an $\Ltwol(\sp)$-density, Corollary
\ref{c:holder} gives H\"older continuity of every order below
$\al-1/2$ for both of these examples. 
This is well known in the rough case (see, e.g.,
\cite[Theorem 7.1(i)]{AbiLarPul19}).

Proposition \ref{p:max-holder} below shows that in these examples
the H\"older order $\al-1/2$ is not attained path by path, so that the
range above is the best possible. Proposition \ref{p:not-semimart}
shows that, when $\mu \in \Mlf^{*}(\sp)$, no version of $Z$ is a semimartingale. Both facts are often
asserted in the literature, but we have been unable to locate proofs
of either.

Under Assumption \ref{a:pi-bound}, 
we have the following consequence of \eqref{Z-sde}, valid
almost surely, for almost every $t \geq 0$:
\begin{align}\label{gen-frac-sde}
  (\oPi * Z)(t) = \int_{[0,t]}\Bprn{ \mu(\ds) - Z_{s}\,\ds }
  + \int_{0}^{t}\sqrt{Z_{s}}\,\dB_{s},
\end{align}
where $B$ is the Brownian motion of Theorem \ref{t:Z-eqs}.
Both proofs rest on the fact that the left-hand side of
\eqref{gen-frac-sde} is a fractional-type integral of $Z$, which
raises the regularity of $Z$, while the regularity of the right-hand
side is known.

\begin{proposition}
\label{p:max-holder}
Under Assumption \ref{a:pi-bound}, almost surely, $Z$ is not
Lebesgue-a.e.\ equal to a locally $(\al-1/2)$-H\"older function on any
interval on which it does not vanish a.e.
\end{proposition}
\begin{proof} In Appendix \ref{s:prf-max-holder}.
\end{proof}

\begin{proposition}
\label{p:not-semimart}
Suppose that Assumption \ref{a:pi-bound} holds and that $\mu \in
\Mlf^{*}(\sp)$. Then no version of $Z$ is a semimartingale, with
respect to any filtration.
\end{proposition}
\begin{proof} In Appendix \ref{s:prf-not-semimart}.
\end{proof}

Combining these results with Corollary \ref{c:holder}, we obtain the
exact H\"older exponent for jump measures with a power-law density at the
origin.

\begin{corollary}
\label{c:exact-holder}
Let $\rho$ be PGID with $b = 0$, and suppose that $\Pi$ has a density
$\pi$ such that, for some $\al \in (1/2,1)$ and $0<c\leq C<\infty$,
\begin{align}\label{pi-two-sided}
  c\,t^{-1-\al} \leq \pi(t) \leq C\,t^{-1-\al} \efor t \in (0,1],
  \eand \sup_{t\geq 1}\pi(t) < \infty.
\end{align}
Suppose that $\mu \in \Mlf^{*}(\sp)$ has an $\Ltwol(\sp)$-density, and let
$\xi$ be the $(\rho,\mu)$-Feller random measure. Then $\xi$ has a density
$Z$ with a version whose paths are locally H\"older continuous of every
order $\theta<\al-1/2$. Moreover, almost surely, $Z$ is not
Lebesgue-a.e.\ equal to a locally $(\al-1/2)$-H\"older function on any
interval on which it does not vanish a.e. Finally, no version of $Z$
is a semimartingale, with respect to any filtration.
\end{corollary}
\begin{proof}
For $\abs{u} \geq 1$, the lower bound in \eqref{pi-two-sided} and the
substitution $s = \abs{u}t$ give
\begin{align*}
  \abs{\Psi(-iu)} \geq \Re\Psi(-iu) \geq c\int_{0}^{1}\prn{1-\cos ut}\,t^{-1-\al}\,\dt
  \geq c\,\abs{u}^{\al}\int_{0}^{1}\prn{1-\cos s}\,s^{-1-\al}\,\ds.
\end{align*}
The last integral is finite and positive, so \eqref{psi-power} holds
with $\beta = \al$. By Corollary \ref{c:holder}, $\rho$ has a density
in $\Ltwo(\sp)$, so $\xi$ has a density $Z$, a version of which
is H\"older continuous by the same corollary. Together with $b=0$
and the upper bound in \eqref{pi-two-sided}, this verifies Assumption
\ref{a:pi-bound}, and the last two claims follow from Propositions
\ref{p:max-holder} and \ref{p:not-semimart}.
\end{proof}

\subsection{Examples}
\label{s:examples}
This subsection gives a number of examples, ranging from the
classical Feller diffusion to the rough CIR process and its tempered
variants, in which various quantities and equations 
of this section can be written out
explicitly. The critical index
$\al=1/2$ gets two examples of its own. It is both the
empirically relevant end of the range of the Hurst index and the
exact threshold for the existence of a square-integrable density.

\begin{example}[The classical Feller (CIR) process]
  \label{e:cir}
In the special case when
$\rho$ is the exponential distribution with mean $b>0$, i.e.,
$\Pi=0$ and $L$ is deterministic with slope $b$, the density
$r(t) = b^{-1}e^{-t/b}$ is bounded, so that Proposition
\ref{p:density-crit} provides a density $Z$ of $\xi$, and
\eqref{Z-sde} becomes the SDE for the classical Feller (CIR) process with immigration:
\begin{align}\label{cir-sde}
  \dZ_t = \oo{b} \mu(\dt)  - \oo{b} Z_t\, \dt + \oo{b} \sqrt{Z_t}\, \dB_t.
\end{align}
In particular, when $\mu(\dt) = b z_{0}\, \delta_0(\dt) + m\, \dt$,
as in Remark \ref{r:init-cond}, we
get the familiar CIR equation
\begin{align*}
  \dZ_t = \oo{b} ( m - Z_t)\,  \dt + \oo{b} \sqrt{Z_t}\,
  \dB_t \eand Z_0 = z_0.
\end{align*}
Here the spatial motion is linear, so vertical and horizontal time
coincide up to the factor $b$, and the construction of Theorem
\ref{t:main} reduces to Feller's original derivation of his
famous diffusion in \cite{Fel51}: for $\mu = b z_{0}\,\delta_{0}$,
$Z_{t} = b^{-1}\ab{X_{t/b},1}$ is his branching diffusion.
\end{example}

\begin{example}[The rough CIR process]
\label{e:rough-cir}
Let $\al \in (1/2,1)$,
$b=0$ and
$\Pi(\dt) = \frac{\al}{\Gamma(1-\al)}\,t^{-1-\al}\,\dt$, so that
$\Psi(\ld) = 1+\ld^{\al}$ and
$\oPi(t) = t^{-\al}/\Gamma(1-\al)$. The measure $\rho$ is then
the Mittag-Leffler distribution with density 
\begin{align*}
  r(t) = t^{\al-1} E_{\al, \al}(-t^{\al}) \efor t>0,  \ewhere
E_{\al,\be}(x) = \sum_{n=0}^{\infty}
\frac{x^{n}}{\Gamma(\al n + \be)}
\end{align*}
is the Mittag-Leffler function.
Throughout the example we assume, in addition, that $\mu \in
\Mlf^*(\sp)$
admits an $\Ltwol(\sp)$-density $m$. Since $\al>1/2$, the density $r$ is
locally square integrable, so that Proposition \ref{p:Z-modulus}
provides an a.s.~locally square-integrable density $Z$ of $\xi$.
We have $\abs{\Psi(-iu)} \sim \abs{u}^{\al}$ as $\abs{u} \to \infty$, so that
\eqref{psi-power} holds with $\beta=\al$ and Corollary \ref{c:holder}
applies. Moreover, $\pi(t) = \frac{\al}{\Gamma(1-\al)}\,t^{-1-\al}$, so
that \eqref{pi-two-sided} holds. By Corollary \ref{c:exact-holder}, the
H\"older range $(0,\al-1/2)$ is optimal, and no version of $Z$
is a semimartingale. Since $r(t) \sim t^{\al-1}/\Gamma(\al)$ as $t \downto 0$, we have
$\norm{r}_{\Ltwo[0,h]} \asymp h^{\al-1/2}$, so the upper bound
\eqref{Z-incr} and the lower bound \eqref{Z-incr-lb} (on intervals where
$r*m$ is bounded away from $0$) are both of order $h^{\al-1/2}$. The
$\Ltwo$-modulus of $Z$ is therefore exactly of order $h^{\al-1/2}$. The resolvent-form equation
\eqref{Z-resolvent} becomes 
\begin{multline}
  \label{rough-resolvent}
  Z_{t} = \int_{[0,t]}(t-s)^{\al-1}
  E_{\al,\al}\bprn{ -(t-s)^{\al} }\,\mu(\ds) 
   + \int_{0}^{t}(t-s)^{\al-1}
  E_{\al,\al}\bprn{ -(t-s)^{\al} }\sqrt{Z_{s}}\,\dB_{s},
\end{multline}
which is (a minimal generalization of)
\cite[Equation (9)]{JaiRos16} with $\mu^*=\ld=1$ in
their notation.

Next, since $\widehat{U^{L}}(\ld) = \ld^{-\al}$ and
  $\ell(t) = \frac{t^{\al-1}}{\Gamma(\al)}$, 
equation \eqref{Z-potential} reads
\begin{align}\label{rough-potential}
  Z_{t} = \oo{\Gamma(\al)}\int_{[0,t]}(t-s)^{\al-1}
  \Bprn{ \mu(\ds) - Z_{s}\,\ds }
  + \oo{\Gamma(\al)}\int_{0}^{t}(t-s)^{\al-1}
  \sqrt{Z_{s}}\,\dB_{s}. 
\end{align}
The choice
$\mu(\dt) = z_0 \oPi(t)\, \dt + m\,\dt$ falls outside our standing assumption, but it corresponds to
the equation in
\cite[Equation (1.3)]{ElERos19}, where
$Z_{0+} = z_{0}$.
\cite[Proposition 6.3]{ElERos19} gives the equivalence of
\eqref{rough-resolvent} and \eqref{rough-potential}.

In addition to \eqref{rough-resolvent} and \eqref{rough-potential},
Remark \ref{r:gen-frac-sde} gives a third equation for $Z$. Here $b=0$,
and $k^{L}*f = \oPi*f$ is the Riemann-Liouville fractional integral
$I^{1-\al}f(t) = \Gamma(1-\al)^{-1}\int_{0}^{t}(t-s)^{-\al}f(s)\,\ds$,
so that \eqref{Z-sde} becomes
\begin{align}\label{frac-int}
  I^{1-\al}Z(t) = \int_{[0,t]} \prn*{\mu(\ds) - Z_{s}\, \ds } +
  \int_{0}^{t}\sqrt{Z_{s}}\,\dB_{s}.
\end{align}
Unlike in \eqref{rough-resolvent} and \eqref{rough-potential}, the
stochastic integral in \eqref{frac-int} carries no path-dependence.
The informal SDE of Remark \ref{r:gen-frac-sde} becomes
\begin{align}
  \label{frac-sde}
  d^{\al} Z_t = \mu(\dt) - Z_t\, \dt + \sqrt{Z_t}\, \dB_t,
\end{align}
where $d^{\al} = d^{L}$ is a fractional differential of
Riemann-Liouville type. 
\end{example}

\begin{example}[Tempered fractional CIR processes]
  \label{e:temper}
Tempering allows us to separate roughness and memory, both of which
are controlled in the rough CIR process by the single index $\al$.
It consists of
multiplying the jump
measure of Example \ref{e:rough-cir} by $e^{-\theta t}$:
$\Pi(\dt) = \frac{\al}{\Gamma(1-\al)}\,t^{-1-\al}e^{-\theta t}\,\dt$
for $\al \in (0,1)$ and $\theta>0$, which gives
\begin{align*}
  \Psi(\ld) = 1+(\ld+\theta)^{\al}-\theta^{\al}.
\end{align*}
As $\ld \to \infty$, we have $\Psi(\ld) \sim \ld^{\al}$, just
  like 
  in the case of the  
rough CIR process. Hence, for $\al>1/2$, the H\"older range
$(0,\al-1/2)$ is unchanged. The two differ at the origin, where
$\Psi(\ld)-1 \sim \al\theta^{\al-1}\ld$ in the tempered case, 
so that the relaxation time
$\Psi'(0+) = \int_{\sp}t\,\rho(\dt) = \al\theta^{\al-1}$ is finite. The density $r$ also decays
exponentially as $t\to\infty$, as the explicit formula below shows.
At $\al = 1/2$, the process $L$ is the inverse
Gaussian subordinator, and the first-passage interpretation of
Example \ref{e:critical} below goes
through with $W$ replaced by a Brownian motion with drift
$\sqrt{2\theta}$.

Comparing Laplace transforms, one checks that
\begin{align*}
  \ell(t) = t^{\al-1}e^{-\theta t}
  E_{\al,\al}\bprn{ (\theta t)^{\al} } \eand
  r(t) = t^{\al-1}e^{-\theta t}
  E_{\al,\al}\bprn{ -(1-\theta^{\al})\,t^{\al} } .
\end{align*}
At $\theta=1$ the killing matches the
tempering exactly, $\Psi(\ld) = (1+\ld)^{\al}$, and the formula
for $r$ 
collapses to the density of the $(\al,1)$-Gamma distribution
$r(t) = t^{\al-1}e^{-t}/\Gamma(\al)$.

Tempered kernels are not new in this context. The \emph{gamma
kernel} $t^{\al-1}e^{-\beta t}/\Gamma(\al)$ appears in the affine
Volterra literature \cite{AbiLarPul19}, and the Volterra square-root
equation with that kernel in place of $\ell$ in \eqref{Z-potential}
corresponds to $\Psi(\ld) = 1+(\ld+\beta)^{\al}$. Up to the
normalizations of Remark \ref{r:cont-disc}(\ref{i:scaling}) 
and a linear change of time,
this is the family above. The present setting recovers the explicit
formulas for $\ell$ and $r$, and adds the interpretation of $L$ as a
tempered stable subordinator and the relaxation time $\Psi'(0+)$.

Let $\al>1/2$, and let $\mu \in \Mlf^{*}(\sp)$ have an
$\Ltwol(\sp)$-density. The bounds
\begin{align*}
  \frac{\al e^{-\theta}}{\Gamma(1-\al)}\,t^{-1-\al} \leq \pi(t) \leq
  \frac{\al}{\Gamma(1-\al)}\,t^{-1-\al}
\end{align*}
 on $(0,1]$ show that Corollary
\ref{c:exact-holder} applies. As in Example \ref{e:rough-cir}, the
H\"older range $(0,\al-1/2)$ is optimal, and no version of $Z$ is a
semimartingale.
\end{example}

\begin{example}[The critical index $\al=1/2$]
\label{e:critical}
The boundary case $\al = 1/2$, which Example \ref{e:rough-cir}
excludes, corresponds to $H=0$ in the parametrization $\al = H+1/2$,
toward which empirical estimates of the Hurst index have moved
\cite{BayHarPig21, FukTakWes22, ForFukGerSmi22}.

Here
$\Psi(\ld) = 1+\sqrt{\ld}$, i.e., $L$ is the $1/2$-stable
subordinator. The drift is $b=0$, and the jump
measure is $\Pi(\dt) = (2\sqrt{\pi})^{-1} t^{-3/2}\,
\dt$. The probability measure $\rho$ is the
Mittag-Leffler distribution with $\al=1/2$ and explicit density
\begin{align}\label{r-half}
  r(t) = \oo{\sqrt{\pi t}} - 2 e^{t}\Phi(-\sqrt{2 t}) \sim
  \oo{\sqrt{\pi t}} \text{ as } t\downto 0,
\end{align}
where $\Phi$ is the
distribution function of the standard normal, 
so that $r \not \in \Ltwol(\sp)$. Consequently, by Proposition
\ref{p:density-crit}, the $(\rho,\mu)$-Feller random measure admits no
square-integrable density $Z$ 
when $\mu \in \Mlf^{*}(\sp)$ (already proved in \cite[Corollary
3.13]{PacZit25}). 

As is well known, the process $\{L_{\sqrt{2}\tau}\}_{\tau \in \tm}$
has the same law as the first-passage process $\{T_{\tau}\}_{\tau \in
\tm}$ of a standard Brownian motion $W$, where $T_{\tau} =
\inf\sets{u}{W_{u} = \tau}$, and as the inverse local time of $W$ at
$0$. Vertical time can thus be read as the first-passage time of $W$,
and the horizontal time of the superprocess of Theorem \ref{t:main} as
the level of $W$.

This is the only case in which superprocesses have previously
appeared in connection with Feller random measures. It is the case
$H=0$, the upper boundary of the hyper-rough range $H \leq 0$ of
\cite{Abi21}, where the square-root Volterra equation with $H=0$ is
obtained from the local occupation time at a point catalyst of a
super-Brownian motion that branches only there \cite{DawFle94, FleLeG95,
MytSal15}. The kernel $\ell(t) = t^{-1/2}/\sqrt{\pi}$ of
\eqref{rough-potential}, with $\al=1/2$, is
completely monotone, so the weak uniqueness of \cite[Theorems 2.2 and
2.13]{Abi21} applies, and, with matching inputs, the two constructions
give the same law of $\prfis{\Gamma_t}$. Heuristically, the two
superprocesses differ by a time change. Measured in its local time at
the catalyst, the elapsed time of a
particle of the catalytic superprocess is the inverse local time, that
is, the $1/2$-stable subordinator that serves as our spatial motion.
\end{example}

\begin{example}[The critical index $\al=1/2$ adjusted]
\label{e:critical-log}
Suppose that, for modeling purposes, one wants a random measure of
roughness similar to that in Example \ref{e:critical}, but with a
square-integrable density. Then it suffices to add a slowly varying
factor to the jump measure. This is a square-root analogue of the
log-modulated fractional Brownian motion of \cite{BayHarPig21}, a
Gaussian process obtained by multiplying the Riemann-Liouville kernel
by a power of $\log(1/t)$, which remains well defined at $H=0$. For $\kappa \geq 0$ and $c_{0}>0$, keep
$b=0$ and define
\begin{align*}
  \Pi(\dt) = c_{0}\,t^{-3/2}\log^{\kappa}(1/t)\,\dt  \eon (0,1],
\end{align*}
and extend it arbitrarily to $(1,\infty)$, as long as it is finite
there. Built this way, $\Pi$ is easily seen to be a L\'evy measure.
Throughout the example, we assume that $\mu \in \Mlf^{*}(\sp)$ admits
an $\Ltwol(\sp)$-density.

By Plancherel's theorem, both the existence of an $\Ltwo$-density
of $\rho$ and the estimate \eqref{modulus-fourier} of $\omega_{T}$
depend on $\Psi$ only through integrals of $\abs{\Psi(-iu)}^{-2}$,
$u \in \bR$. It is therefore enough to know $\abs{\Psi(-iu)}$ up to
constant factors for large $\abs{u}$. By Karamata's theorem
\cite[Theorem 1.5.11]{BinGolTeu87}, the function $\oPi$ is regularly
varying of index $-1/2$ at $0$, namely $\oPi(t) \sim
2c_{0}\,t^{-1/2}\log^{\kappa}(1/t)$ as $t\downto0$. Restricting the
integral in $\Re\Psi(-iu) = 1+\int(1-\cos ut)\,\Pi(\dt)$ to $t \leq
1/\abs{u}$ gives the lower bound, and $\abs{1-e^{iut}} \leq
\min(2,\abs{u}t)$, together with the behavior of $\oPi$ at $0$,
gives the upper bound in
\begin{align*}
  c\,\abs{u}^{1/2}\log^{\kappa}\abs{u} \leq \abs{\Psi(-iu)} \leq
  C\,\abs{u}^{1/2}\log^{\kappa}\abs{u} \efor \abs{u} \geq 2,
\end{align*}
for some constants $0<c \leq C<\infty$. Since also $\abs{\Psi(-iu)}
\geq \Re\Psi(-iu) \geq 1$, the integral $\int_{\bR}\abs{\Psi(-iu)}^{-2}\,\du$
converges precisely when $\kappa>1/2$.

Let $\kappa>1/2$. Plancherel's theorem then provides an
$\Ltwo(\bR)$-density $r$ for $\rho$, and \eqref{modulus-fourier}, from
the proof of Corollary \ref{c:holder}, applies verbatim. Together with
the bounds above, it gives
\begin{align}
  \label{log-modulus}
  \omega_{T}(h)^{2} \leq \oo{2\pi}\int_{\bR}
  \min\bprn{ 4, u^{2}h^{2} }\,\abs{ \Psi(-iu) }^{-2}\,\du
  \leq C'\log^{1-2\kappa}(1/h)
  \efor h \in (0,1/2].
\end{align}

In particular, \eqref{r-sq-int} holds, and Proposition
\ref{p:density-crit} produces a square-integrable density $Z$ at the
very index $\al=1/2$ at which it does not exist in 
Example \ref{e:critical}.
The power condition \eqref{psi-power} cannot
distinguish between the two: it holds with $\beta=1/2$ in both, while
Corollary \ref{c:holder} requires $\beta>1/2$.

The bound \eqref{log-modulus} decays more slowly than any power of
$h$, so the condition $\omega_{T}(h) \leq C_{T}h^{\gm}$ is not known to
hold for any $\gm>0$, and we cannot obtain the pathwise continuity of
$Z$ this way. The estimate \eqref{Z-incr} still gives continuity of $t
\mapsto Z_{t}$ in $\bL^{p}$ with a logarithmic modulus: for every $p
\geq 2$ and $T>0$ there is a $C'' = C''(p,T)<\infty$ such that
\begin{align*}
  \norm{ Z_{t+h}-Z_{t} }_{\Lpee} \leq
  C''\log^{1/2-\kappa}(1/h)
  \eforall h \in (0,1/2] \eand t,t+h \in [0,T].
\end{align*}
\end{example}

\section{The discontinuous case}
\label{S:disc}

In this section we treat the case when $p:=\rho(\set{0})>0$. According
to Proposition \ref{p:no-atom}, this is equivalent to $L$ being a
compound Poisson process, i.e., to $b=0$ and $\Pi$ being a finite
measure. The structure of $\xi$ and of its vertical noise $N$ then
differs from the continuous case: for $\mu \neq 0$, $\xi$ is purely
atomic, and $N$ is a time-changed compensated inverse-Gaussian process
rather than a time-changed Brownian motion.

Subsection \ref{s:disc-noise} describes the noise as a time-changed
compensated inverse-Gaussian process (Theorem \ref{t:ig-rep}), locates
the atoms of $\xi$ (Corollary \ref{c:atomic}), contrasts the two cases
(Corollary \ref{c:dichotomy}), and writes $N$ as a Brownian motion
time-changed by $\Gamma$ (Corollary \ref{c:bm-tc} and Remark
\ref{r:bm-tc}). Subsection \ref{s:disc-eqs} derives a stochastic
Volterra equation for $\xi[0,\cdot]$, driven by an inverse-Gaussian
subordinator, and proves uniqueness (Theorem \ref{t:disc-eqs}).
Subsection \ref{s:disc-examples} collects examples.

\subsection{The structure of the noise}
\label{s:disc-noise}

By analogy with the classical Feller diffusion, Theorem
\ref{t:ig-rep} below shows that $\xi(\dt)$ governs the local size of
the increments of $N$. In the continuous case,
$\xi(\dt)$ plays the role of the infinitesimal variance of a
locally Gaussian noise. In the discontinuous case $p > 0$, the
noise is a time-changed, compensated inverse-Gaussian process. Here $\xi(\dt)$ is still
a ``local measure of spread'', but only in the weaker sense of Remark \ref{r:not-qv}: in general,
$\xi[0,t]$ is neither the predictable nor the optional
quadratic variation of $N$.

\smallskip

For $y \geq 0$, let
$\IG(y,p^2)$ denote the inverse-Gaussian distribution on
$\sp$ with mean $y$ and dispersion parameter $p^2$, characterized by its
Laplace transform
\begin{align}\label{ig-laplace}
 \int_{\sp} e^{-\ld z}\, \IG(y,p^2)(\diffd z) =
 \Exp{ - y \phi(\ld) },  \ \phi(\ld) = p^{-2} 
 \bprn{ \sqrt{1+2p^2 \ld} - 1 }, \  
\ld \geq 0. 
\end{align}
A L\'evy subordinator $\{S_y\}_{y \geq 0}$ 
with Laplace exponent $\phi$ is called an
inverse-Gaussian subordinator (the dispersion parameter
is fixed at $p^2$ throughout). Thus
$S_{y} \sim \IG\prn{y,p^{2}}$. In particular,
$\ee{S_{y}} = y$ and $\Var\prn{S_{y}} = p^{2}y$. For $y=0$, we have $\IG(0,p^{2}) = \delta_{0}$.

\smallskip

We recall that the vertical noise $\prfi{N_t}$ associated with
a Feller random measure $\xi$ in occupation form is defined in
\eqref{noise}, and that its fundamental distributional properties
are given in Proposition \ref{p:noise-mg}. Since $N$ turns out
to be of finite variation in the discontinuous case, the measure
$\dN$ is well defined, with the conventions of subsection \ref{s:notation}. 

Recall that the filtration $\prfi{\sV_t}$ is given in
\eqref{vfilt}. We also define 
  $\sV_{0-}$ to be the completed trivial $\sg$-algebra and
  set $\Gamma_{0-} = \Gamma^{p}_{0-} = 0$
  for the processes $\Gamma$ and $\Gamma^{p}$ defined in
  \eqref{Gp-def} below.
The following theorem describes the structure of the noise in the
discontinuous case. 
\begin{theorem}
\label{t:ig-rep}
Let $\rho$ be PGID with $p := \rho(\set{0}) > 0$, let
$\mu \in \Mlf(\sp)$, let $\xi$ be the $(\rho,\mu)$-Feller random
measure in occupation form, let $\Gamma = \xi[0,\cdot]$, and define
\begin{align}\label{Gp-def}
  \Gamma^{p}_{t} := p\,\mu[0,t] + p\textstyle\int_{(0,t]}\Gamma_{t-s}\,\Pi(\ds), \quad t\geq 0. 
\end{align}
Then
\begin{enumerate}[leftmargin=*, label=\arabic*., ref=\arabic*]
  \item \emph{Martingale structure.}
  \begin{enumerate}[label=(\alph*), ref=\arabic{enumi}.(\alph*)]
  \item \label{i:comp} $\Gamma^{p}$ is $\prfi{\sV_t}$-predictable, 
    $N$ is a $\prfi{\sV_t}$-martingale of 
    finite variation and 
\begin{align}\label{strong-disc}
  \Gamma_{t} - \Gamma^{p}_{t} = p\,N_{t} \eforall t \in \sp, \eas,
\end{align}
so that $\Gamma^p$ is the compensator of $\Gamma$.
\item \label{i:qv}The martingale $N$ is 
  purely discontinuous and  
\begin{align*}
  \ab{N}_{t} = \Gamma^{p}_{t}.
\end{align*}

\item \label{i:pred-jump}For every predictable time $\tau$, the conditional law of
$\Delta\Gamma_{\tau}$ given $\sV_{\tau-}$ is, on $\set{\tau<\infty}$,
$\IG\prn{\Delta\Gamma^{p}_{\tau},p^{2}}$. In particular, $N$ jumps at
a predictable time exactly when $\Gamma^{p}$ does, and $N$ is
quasi-left-continuous if and only if $\Gamma^{p}$ is continuous a.s.
  \end{enumerate}
\item \emph{Time-change representation.}
There exists, on a possibly extended probability space
carrying
$\prfis{\sV_{t}}$, an
in\-ver\-se-Gaussian subordinator $\set{S_{y}}_{y \geq 0}$  such that
\begin{enumerate}[label=(\alph*), ref=\arabic{enumi}.(\alph*)]
  \item \label{i:ig-tc}
    $N_{t} = \tS_{\Gamma^{p}_t}$ and $\Gamma_t =
    S_{\Gamma^p_t}$, for $t\in \sp$, a.s.,  
    where $\tS_y =\too{p}( S_{y} - y)$, and
\item \label{i:indep-sub} there exists a filtration $(\sH_{x})_{x \in
  \sp}$ which satisfies the usual conditions, such that $S$ is an
  inverse-Gaussian subordinator with respect to $(\sH_{x})$, and each
  $\Gamma^{p}_{t}$ is an $(\sH_{x})$-stopping time with $\sV_{t}
  \subseteq \sH_{\Gamma^{p}_{t}}$. In particular, for each $t \in \sp$
  the process $\{ S_{\Gamma^{p}_{t}+w} - S_{\Gamma^{p}_{t}} \}_{w \geq 0}$
  is an inverse-Gaussian subordinator, independent of $\sV_{t}$.
\end{enumerate}
\end{enumerate}
\end{theorem}

\begin{proof} In Appendix \ref{s:prf-ig-rep}.
\end{proof}

\begin{remark}[Hawkes structure]
\label{r:hawkes-disc}
By \eqref{Gp-def} and part \ref{i:comp} of Theorem \ref{t:ig-rep}, the
compensator of $\Gamma$ is $\Gamma^{p}_{t} = p\,\mu[0,t] +
(p\,\Pi*\xi)[0,t]$. This is the intensity structure of a Hawkes process
with immigration $p\,\mu$ and excitation kernel $p\,\Pi$, whose total
mass $1-p$ is the branching ratio. Its noise, however, is
inverse-Gaussian rather than Poisson, by part \ref{i:ig-tc} of Theorem
\ref{t:ig-rep}. Consistently with Remark \ref{r:cluster}, $\rho =
p\sum_{k \geq 0}(p\,\Pi)^{*k}$ is the normalized mean measure of a
cluster of this Hawkes process.
\end{remark}

\begin{remark}
\label{r:intrinsic}
As in the continuous case (see Remark \ref{r:intrinsic-cont}),
  the occupation form is only a convenient device in Theorem \ref{t:ig-rep}.
By \eqref{Gp-def} and \eqref{strong-disc}, $\Gamma^{p}$ and $N =
(\Gamma - \Gamma^{p})/p$ are functionals of $\xi$. Let
$\prfis{\sF^{\xi}_{t}}$ be the natural filtration of $\Gamma$, so that
$\sF^{\xi}_{t} \subseteq \sV_{t}$. The argument of subsection \ref{s:pred-Gp}
shows that $\Gamma^{p}$ is $\prfis{\sF^{\xi}_{t}}$-predictable, and the
tower property turns the $\prfis{\sV_{t}}$-martingales $\Gamma -
\Gamma^{p}$ and $N^{2} - \Gamma^{p}$ into
$\prfis{\sF^{\xi}_{t}}$-martingales. Hence parts \ref{i:comp} and
\ref{i:qv} hold with $\prfis{\sV_{t}}$ replaced by
$\prfis{\sF^{\xi}_{t}}$. So does part \ref{i:pred-jump}, since every
$\prfis{\sF^{\xi}_{t}}$-predictable time is
$\prfis{\sV_{t}}$-predictable and $\Delta\Gamma^{p}_{\tau}$ is
$\sF^{\xi}_{\tau-}$-measurable, and so does part \ref{i:indep-sub},
since $\sF^{\xi}_{t} \subseteq \sV_{t}$.

In this form, all parts of Theorem \ref{t:ig-rep}, including the
representation of part \ref{i:ig-tc}, are statements about the joint
law of $\xi$ and $S$ alone, because $\Gamma$, $\Gamma^{p}$, $N$ and
$\prfis{\sF^{\xi}_{t}}$ are functionals of $\xi$. By the transfer
theorem \cite[Theorem 8.17]{Kal21}, they therefore hold for every
$(\rho,\mu)$-Feller random measure, with $S$ defined on an extension
of its probability space.
\end{remark}

\begin{corollary}
\label{c:atomic}
Under the assumptions and notation of Theorem \ref{t:ig-rep}, we
have
\begin{align}\label{atoms}
   \xi = \tsum_{ y>0 :\, \Theta_{y}<\infty }
  \Delta S_{y}\,\delta_{\Theta_{y}}
\eas,\ewhere
\Theta_{y} := \inf\sets{t}{\Gamma^{p}_{t} \geq y} \efor y>0.
\end{align}
In particular, $\xi$ is purely atomic, and its atoms are dense in
every interval on which $\Gamma^{p}$ is strictly increasing. If, in
addition, $\mu$ and $\Pi$ are purely atomic, then the atoms of $\xi$
lie in the deterministic countable set
\begin{align}\label{atom-set}
   \Sigma := \set{ \text{atoms of }\mu } +
  \Bprn{ \set{0} \cup \set{ \text{finite sums of atoms of }\Pi } }. 
\end{align}
\end{corollary}

\begin{proof} In Appendix \ref{s:prf-atomic}.
\end{proof}

\begin{corollary}
\label{c:dichotomy}
Let $\rho$ be PGID, let $\mu \in \Mlf^{*}(\sp)$, and let $\xi$ be the
$(\rho,\mu)$-Feller random measure in occupation form.
\begin{enumerate}
  \item \label{i:dich-cont}If $\rho(\set{0})=0$, then $\xi$ is a.s.\
    diffuse, and $\Gamma = \bkt{N} = \ab{N}$.
  \item \label{i:three-procs}If $\rho(\set{0})>0$, then $\xi$ is
    a.s.\ purely atomic and nonzero. The increasing processes $\Gamma$,
    $\bkt{N}$ and $\ab{N}$ have the common compensator $\Gamma^{p}$, and
    they are a.s.\ pairwise distinct.
\end{enumerate}
\end{corollary}
\begin{proof} In Appendix \ref{s:prf-dichotomy}.
\end{proof}

Since the inverse-Gaussian subordinator is the first-passage
process of a Brownian motion with drift, we have the following
result:
\begin{corollary}
\label{c:bm-tc}
Under the hypotheses of Theorem \ref{t:ig-rep}, there exists, on
a possibly extended probability space, a Brownian motion $\beta$ such
that, almost surely, 
\begin{align}\label{bm-tc}
   N_{t} = \beta_{\Gamma_{t}} \eand
  \Gamma^{p}_{t} = \sup_{s \leq \Gamma_{t}} \prn{ s - p\,\beta_{s} }
  \eforall t\geq 0.
\end{align}
\end{corollary}

\begin{proof}
For a Brownian motion $\beta'$, the process
$T'_{a} := \inf\sets{s}{s - p \beta'_s > a}$
is an inverse-Gaussian subordinator with dispersion parameter $p^{2}$ (see, e.g.,
\cite[Section 1.2.5]{Kyp06}). Therefore, by 
\cite[Corollary 8.18]{Kal21}, we can find
a Brownian motion $\beta$, possibly on an extension of the
underlying probability space,  such
that $S_{a} = \inf\sets{s}{s - p\,\beta_{s} > a}$ for all $a$, a.s.
By continuity, $s - p\,\beta_{s} \leq a$ for $s \leq S_{a}$, with
equality at $s = S_{a}$. At $a = \Gamma^{p}_{t}$, where $S_{a} =
\Gamma_{t}$, this gives the second identity in \eqref{bm-tc} and
$\Gamma_{t} - p\,\beta_{\Gamma_{t}} = \Gamma^{p}_{t}$, i.e., $p\,N_{t}
= \Gamma_{t} - \Gamma^{p}_{t} = p\,\beta_{\Gamma_{t}}$ by
\eqref{strong-disc}.
\end{proof}

\begin{remark}[Two descriptions of the noise]\label{r:bm-tc}
Theorem \ref{t:ig-rep} and Corollary \ref{c:bm-tc} give two
descriptions of the noise, each of which also holds in the continuous
case. First, $N_{t} = \tS_{\ab{N}_{t}}$, where $\ab{N}$ is the
compensator of $\Gamma$ and $\tS$ is a centered L\'evy process with
$\Var\prn*{\tS_{y}} = y$. In the continuous case, $\ab{N} = \Gamma$ and
$\tS$ is a Brownian motion (Corollary \ref{c:dds}). In the
discontinuous case, $\ab{N} = \Gamma^{p}$ and $\tS$ is a compensated
inverse-Gaussian process. Its characteristic exponent converges to
$u^{2}/2$ as $p \to 0$, so $\tS$ converges in law to a Brownian motion
in the Skorokhod topology \cite[Corollary VII.3.6]{JacShi03}.

Second, $N_{t} = \beta_{\Gamma_{t}}$ for a Brownian motion $\beta$, as
in \eqref{clock-k}. In the continuous case the two descriptions
coincide. In the discontinuous case they differ: the Brownian motion
runs at the speed $\Gamma$, which jumps, while the L\'evy process runs at
the speed $\ab{N}$. By Corollary \ref{c:dichotomy}, for $\mu \in
\Mlf^{*}(\sp)$ the time change $\Gamma$ is a.s.\ neither $\ab{N}$ nor
$\bkt{N}$. This is consistent with \eqref{expmg} and Remark \ref{r:not-qv}:
formally, the processes in \eqref{expmg} are the Brownian
martingales $\beta_{s}^{2} - s$ and $\exp{\ld\beta_{s} -
\ld^{2}s/2}$ evaluated at the times $s = \Gamma_{t}$.
\end{remark}

\subsection{Stochastic equations}
\label{s:disc-eqs}
In the continuous case, $\xi$ is characterized by a stochastic
equation for $\Gamma$ driven by a continuous martingale whose
quadratic variation is $\Gamma$ itself (Theorem \ref{t:mart-char}).
This is also how stochastic Volterra equations are usually posed in
the literature. In the discontinuous case this description no longer
suffices: $N$ is purely discontinuous, and neither $\ab{N}$ nor
$\bkt{N}$ equals $\Gamma$ (Remark \ref{r:bm-tc}). Theorem
\ref{t:ig-rep} supplies the missing information: $N$ is the
compensated inverse-Gaussian process $\tS$ time-changed by $\ab{N}$. The next
theorem shows that, for a given $S$, the Volterra equation
\eqref{volterra-disc} together with this description \eqref{ig-noise}
of the noise has $(\Gamma,N)$ as its unique solution. Equivalently,
$\Gamma$ is the unique solution of the time-change equation
\eqref{disc-tc}. This is the counterpart of the characterization
\eqref{bracket}--\eqref{char-strong} of Theorem \ref{t:mart-char}, with the
inverse-Gaussian subordinator $S$ in place of a Brownian motion,
entering through a time change.
\begin{theorem}[Stochastic equation: the discontinuous case]
\label{t:disc-eqs}
Let $\xi$, $\Gamma$, $\Gamma^{p}$, $N$, $S$ and $(\sH_{x})$ be as in Theorem
\ref{t:ig-rep}. Then $\Gamma$ and $N$ solve, almost surely for
all $t \in \sp$,
\begin{align}
  \label{volterra-disc}
  \Gamma_{t} &= \prn{\rho*\mu}[0,t] +
  \textstyle\int_{[0,t]} \rho[0,t-s]\,\dN_{s}, \\
  \label{ig-noise}
  N_{t} &= \tS_{ \Gamma_{t} - p\,N_{t} },
\ewhere
\tS_{y} = \too{p}(S_{y}-y), 
\end{align}
and $\Gamma_{t} - p\,N_{t} = \Gamma^{p}_{t}$ is an $(\sH_{x})$-stopping
time for each $t$. Equivalently, $\Gamma$ solves
\begin{align}
  \label{disc-tc}
  \Gamma_{t} = S_{ p\,\mu[0,t] + p \int_{(0,t]}
  \Gamma_{t-s}\,\Pi(\ds) } \eforall t \in \sp, \eas.
\end{align}

Moreover, let $\Gamma'$ and $N'$ be processes on the same space,
with $\Gamma'$ nondecreasing and \cd{}, and $N'$ \cd{} of finite variation with
$N'_{0-} = 0$. If they solve \eqref{volterra-disc}--\eqref{ig-noise} and
each $\Gamma'_{t} - p\,N'_{t}$ is an $(\sH_{x})$-stopping
time, then $\Gamma' = \Gamma$ and $N' = N$ a.s.
\end{theorem}

\begin{proof}
In Appendix \ref{s:prf-disc-eqs}.
\end{proof}

\begin{remark}\label{r:xu}
Up to normalization, equation \eqref{volterra-disc} is the
Volterra equation of \cite[Theorem 2.13]{Xu24a} for the limit $\Xi$
and its martingale part $M$, which holds for the
scaling limits of nearly unstable Hawkes processes also when the
kernel has an atom at $0$. In \cite[Theorem 2.14]{Xu24a}, the driving
martingale is identified as $M = B_{\Xi}$, a Brownian motion $B$
time-changed by $\Xi$,
and the pair $(\Xi,M)$ is shown to be the unique continuous
solution of the resulting system. This theorem assumes that the
limit is continuous. Corollary
\ref{c:bm-tc} shows that the representation $N_{t} = \beta_{\Gamma_{t}}$ 
persists in the discontinuous case. The time change $\Gamma$,
however, is then no longer the quadratic variation of $N$ (Remark
\ref{r:bm-tc}). In
Theorem \ref{t:disc-eqs}, the role of $M_{t} = B_{\Xi_{t}}$ is
taken by \eqref{ig-noise}, in which the noise is the
compensated inverse-Gaussian process $\tS$ time-changed by 
the
predictable quadratic variation $\ab{N} = \Gamma - p\,N$ of $N$. With this
replacement, uniqueness holds, for a given $S$, also in the
discontinuous case.
\end{remark}

\subsection{Examples}
\label{s:disc-examples}

\begin{example}[The case $\rho = \delta_{0}$]
\label{e:delta}
Let $\rho = \delta_{0}$, so that $p=1$ and $L = 0$. 
The Riccati-Volterra equation \eqref{riccati} becomes the pointwise identity $v
= f - v^{2}/2$, whose nonnegative solution is $v = \sqrt{1+2f} - 1$.
Hence
\begin{align}\label{laplace-delta}
  \ee{ e^{-\ab{\xi,f}} } = \exp{ -\textstyle\int_{\sp}
  \bprn{ \sqrt{1+2f} - 1 }\,\diffd\mu } \efor f \in \Bcp(\sp).
\end{align}
Therefore, $\xi$ is independently scattered, i.e., 
$\xi(A_{1}), \dots, \xi(A_{n})$ are independent
whenever $A_{1}, \dots, A_{n}$ are disjoint, 
with 
$\xi(A) \sim 
\IG\prn{\mu(A),1}$ for every Borel set $A$ with $\mu(A)<\infty$.
In particular, there is no
self-excitation. 

All objects of Theorem \ref{t:ig-rep} are explicit. The compensator
$\Gamma^{p}_{t} = \mu[0,t]$ is deterministic, so \eqref{disc-tc} is the
formula $\Gamma_{t} = S_{\mu[0,t]}$, and the vertical noise $N_{t} =
\tS_{\mu[0,t]} = S_{\mu[0,t]} - \mu[0,t]$ has independent increments.
Its predictable quadratic variation $\ab{N}_{t} = \mu[0,t]$ is
deterministic, while $\Gamma$ is not, as in Remark \ref{r:bm-tc}. By Corollary \ref{c:atomic},
$\xi$ is purely atomic. Moreover, a.s., the closure of the set of
its atoms is $\supp\mu$: by \eqref{laplace-delta}, $\xi$ vanishes on
every interval of $\mu$-measure $0$, and charges, hence has an atom
in, every interval with rational endpoints and positive
$\mu$-measure. 
For $\mu(\dt) = m\,\dt$
with $m>0$, the process $\Gamma_{t} = S_{mt}$ is an inverse-Gaussian
L\'evy process and $N$ is its compensated version. 

\smallskip

The kernels of Example \ref{e:cir} converge weakly to $\delta_{0}$
as $b \downto 0$, so the corresponding Feller random measures converge
in law to $\xi$ by Proposition \ref{p:stability}. For $\mu(\dt) =
m\,\dt$, this is the inverse-Gaussian limit of the integrated
square-root process under fast mean reversion and high volatility,
recalled in \cite{AbiAtt25, AbiAttRos26}. In this case, the
distribution function $\Gamma$ of $\xi$ also appears, up to a
constant factor, in
\cite{AbiAttRos26} as the weak limit of integrated hyper-rough
square-root processes as $H \downto -1/2$, and Corollary
\ref{c:vague-stability} recovers that convergence. There, for
$\al_{n} := H_{n} + 1/2 \downto 0$ and constants $\lambda > 0$,
$\theta, V_{0} \geq 0$ and $\nu \neq 0$, the integrated hyper-rough
square-root process $X^{n}$ and a continuous local martingale $M^{n}$
with $\ab{M^{n}} = \nu^{2}X^{n}$ satisfy
\begin{align*}
  X^{n}_{t} = G^{n}(t) + \int_{0}^{t}K_{n}(t-s)\prn{M^{n}_{s} -
  \lambda X^{n}_{s}}\,\ds, \qquad K_{n}(t) = \al_{n}t^{\al_{n}-1},
\end{align*}
with $G^{n}(t) = V_{0}\,t + \lambda\theta\,t^{1+\al_{n}}/(1+\al_{n})$
\cite[Equations (1.4)--(1.7)]{AbiAttRos26}. The kernel $K_{n}$ is
completely monotone, so the law of $X^{n}$ is unique \cite[Theorems 2.2
and 2.13]{Abi21}. Let $\rho_{n}$ be the PGID law with $\Psi_{n}(z) = 1
+ z^{\al_{n}}/\prn{\lambda\,\Gamma(1+\al_{n})}$, let
\begin{align*}
  \mu_{n}(\dt) = \frac{\lambda}{\nu^{2}}\Bprn{ \frac{V_{0}\,
  t^{-\al_{n}}}{\Gamma(1+\al_{n})\Gamma(1-\al_{n})} + \lambda\theta }\,\dt,
\end{align*}
and let $\Gamma_{n}$ be the distribution function of the
$(\rho_{n},\mu_{n})$-Feller random measure. Since $\lambda^{-1}\hrho_{n}
= \hat K_{n}/(1+\lambda\hat K_{n})$ and $K_{n} * (\nu^{2}\lambda^{-1}\mu_{n})
= \diffd G^{n}$, a comparison of Laplace transforms turns
\eqref{volterra} for $\Gamma_{n}$, together with $\ab{N_{n}} = \Gamma_{n}$
(Corollary \ref{c:dds}), into the equation above, with $X^{n} :=
\nu^{2}\lambda^{-2}\,\Gamma_{n}$ and $M^{n} := \nu^{2}\lambda^{-1}N_{n}$, where $N_{n}$ is the noise of that Feller random measure. By
uniqueness, this identifies the law of $X^{n}$. As $n \to \infty$, we
have $\Psi_{n}(z) \to 1 + 1/\lambda$ for $z > 0$, so that $\rho_{n}
\to \frac{\lambda}{1+\lambda}\,\delta_{0}$ vaguely, and $\mu_{n} \to
\lambda g_{0}\nu^{-2}\,\dt$ vaguely, with $g_{0} := V_{0} +
\lambda\theta$. By Corollary \ref{c:vague-stability} and Remark
\ref{r:stability}, $X^{n}$ converges in distribution, in the Skorokhod
$M_{1}$ topology, to $\kappa\,\Gamma$, where $\Gamma$ corresponds to
$\mu(\dt) = m\,\dt$ with
\begin{align*}
  \kappa = \frac{\nu^{2}}{(1+\lambda)^{2}} \eand m =
  \frac{g_{0}(1+\lambda)}{\nu^{2}}.
\end{align*}
This is the convergence of the first component in \cite[Theorem
3.1]{AbiAttRos26}. Since each $X^{n}$ is continuous and $\Gamma$ is
not, the convergence cannot hold in the $J_{1}$ topology, so $M_{1}$ is
the natural topology here. The joint convergence with $M^{n}$, in the $S$
topology, is not covered by this argument.
\end{example}

\begin{example}[The Poisson case]
\label{e:lattice}
Let $b = 0$ and $\Pi = c\,\delta_{1}$ for some $c>0$, so that $L$
is a Poisson process of rate $c$ and $\Psi(\ld) = 1 +
c\prn{1-e^{-\ld}}$. By Remark \ref{r:pgid-ex}, $p = \rho(\set{0}) =
1/(1+c)$ and, with $q = 1-p$, $\rho$ is the geometric distribution
\begin{align}\label{geom}
  \rho = \textstyle\sum_{n \in \N_{0}} p\,q^{n}\,\delta_{n}.
\end{align}
With $\Pi = (a/h)\,\delta_{h}$ for some $a>0$ instead, $\rho$ is
the geometric distribution with mean $a$ on the grid $h\,\N_{0}$. As
$h \downto 0$, it converges weakly to the exponential distribution with
mean $a$, i.e., to the kernel of Example \ref{e:cir} with $b = a$, and
the Feller random measures converge accordingly, by Proposition
\ref{p:stability}.
Let $\mu = \sum_{n}\mu_{n}\delta_{n}$ be
carried by $\N_{0}$, and write $\xi_{n} := \xi(\set{n})$. By Corollary
\ref{c:atomic}, $\xi$ is carried by $\N_{0}$ as well.

\smallskip

The infinitesimal generator of the spatial motion is
$\sA g(n) = c\prn{g(n+1) - g(n)} - g(n)$, so the
superprocess of Appendix \ref{S:spier} can be
represented as a sequence of
nonnegative processes $\seqz{X_{\tau}(n)}$. This way the 
martingale problem \eqref{mart-prob} becomes the system
\begin{align}\label{walk-system}
  \dX_{\tau}(n) = \bprn{ c\,X_{\tau}(n-1) - (1+c)\,X_{\tau}(n) }\,\dtau
  + \sqrt{X_{\tau}(n)}\,\dB^{n}_{\tau}, 
\end{align}
with initial conditions $X_{0}(n) = \mu_{n}$, 
driven by independent Brownian motions $B^{n}$, and following the convention
$X_{\tau}(-1) = 0$. 

\smallskip

Let $f \in \Bcp(\sp)$
vanish on $(T,\infty)$, and put $f_{n} := f(n)$ and $v_{n} := v(n)$.
Using the fact that $\rho$ is geometric, we can write the Riccati-Volterra
equation \eqref{riccati} as 
\begin{align*}
  v_{n} = p\prn{f_{n} - v_{n}^{2}/2} + q\,v_{n+1}.
\end{align*}
Since $v$ is nonnegative, we obtain the backward recursion
\begin{align}\label{recursion}
  v_{n} = \oo{p}\Bprn{ \sqrt{ 1 + 2p\prn{ p\,f_{n} + q\,v_{n+1} } } - 1 }
  \efor n \leq T, \eand v_{n} = 0 \efor n > T,
\end{align}
which can be used to compute the values of $v$. 

\smallskip

Since $p\,\Pi
= q\,\delta_{1}$, \eqref{Gp-def} gives $\Gamma^{p}_{n} = \sum_{k=0}^{n}
\prn{p\,\mu_{k} + q\,\xi_{k-1}}$, with $\xi_{-1} := 0$. 
By \eqref{strong-disc} and \eqref{ig-noise}, evaluated at $t = n$, we obtain the recursive system
\begin{equation}\label{lattice-system}
  \begin{aligned}
    \xi_{n} &= p\,\mu_{n} + q\,\xi_{n-1} + p\,\Delta N_{n},\\
    \Gamma^{p}_{n} &= \Gamma^{p}_{n-1} + p\,\mu_{n} +
    q\,\xi_{n-1}, \\
  \Delta N_{n} &= \tS_{\Gamma^{p}_{n}} - \tS_{\Gamma^{p}_{n-1}},
\end{aligned}
\end{equation}
for $n \in \N_{0}$, where $\Gamma^{p}_{-1} := 0$ and $\tS$ is as
in \eqref{ig-noise}. 
The first equation can also be obtained by
integrating \eqref{walk-system} over $\tau \in \tm$ and setting 
$\Delta N_{n} = \int_{\tm}\sqrt{X_{\tau}(n)}\,\dB^{n}_{\tau}$.

\smallskip

The equations \eqref{lattice-system} make it clear that 
$\seqz{\xi_n}$ is a Markov chain in its own filtration. It
is an autoregressive process with ``square-root''
non-Gaussian noise. Its transition mechanism is 
\begin{align*}
  \xi_n \mid \sg(\xi_{n-1}, \dots) \sim 
  \IG( p \mu_n + q \xi_{n-1}, p^2). 
\end{align*}
This Markov chain is also known in the literature as a \emph{Ji\v{r}ina process} 
with inverse-Gaussian offspring law and time-inhomogeneous immigration
(see \cite{Jir58}).
When $\mu_{n}$ does not depend on $n$, it is also a compound
autoregressive process in the sense of \cite[Section 3.2]{DarGouJas06},
with an inverse-Gaussian compounding law. It is the inverse-Gaussian
counterpart of the autoregressive gamma process discussed there.
\end{example}

\begin{example}[The general lattice case]
\label{e:lattice-gen}
Let $b = 0$ and $\Pi = c\,\pi$, where $c>0$ and $\pi$ is a
probability measure on $\set{1,2,\dots}$ with $\pi_{j} :=
\pi(\set{j})$. With $p$, $q$, $\mu$ and $\xi_{n}$ as in Example
\ref{e:lattice}, everything goes through as in the Poisson case. The
only difference is that the dependence on the past is no longer
Markovian. By Remark \ref{r:pgid-ex}, $\rho = p\sum_{k\geq0}q^{k}\,\pi^{*k}$, so that $\rho =
p\,\delta_{0} + q\,\pi*\rho$, and the Riccati-Volterra equation \eqref{riccati}
becomes the backward recursion
\begin{align*}
  v_{n} = \oo{p}\Bprn{ \sqrt{ 1 + 2p\bprn{ p\,f_{n} +
  q\textstyle\sum_{j\geq1}\pi_{j}\,v_{n+j} } } - 1 }.
\end{align*}
The recursive system \eqref{lattice-system} becomes
\begin{equation}\label{lattice-system-gen}
  \begin{aligned}
    \xi_{n} &= p\,\mu_{n} + q\textstyle\sum_{j\geq1}\pi_{j}\,\xi_{n-j}
    + p\,\Delta N_{n},\\
    \Gamma^{p}_{n} &= \Gamma^{p}_{n-1} + p\,\mu_{n} +
    q\textstyle\sum_{j\geq1}\pi_{j}\,\xi_{n-j}, \\
    \Delta N_{n} &= \tS_{\Gamma^{p}_{n}} - \tS_{\Gamma^{p}_{n-1}},
  \end{aligned}
\end{equation}
with $\xi_{m} := 0$ for $m<0$, and the transition mechanism is
\begin{align*}
  \xi_{n} \mid \sg(\xi_{n-1}, \xi_{n-2}, \dots) \sim \IG\Bprn{
  p\,\mu_{n} + q\textstyle\sum_{j\geq1}\pi_{j}\,\xi_{n-j},\, p^{2} }.
\end{align*}
\end{example}

\begin{example}[Exponential excitation]
\label{e:exp-disc}
Let $b = 0$ and $\Pi(\du) = c\,\theta\,e^{-\theta u}\,\du$ with
$c,\theta>0$, so that $p = \rho(\set{0}) = 1/(1+c)$, and let $q = 1-p$. Then $\Psi(\ld) = 1
+ c\ld/(\theta+\ld)$, and, by Remark \ref{r:pgid-ex},
\begin{align*}
  \rho(\dt) = p\,\delta_{0}(\dt) + q\,p\,\theta\,e^{-p\theta t}\,\dt.
\end{align*}
This is the discontinuous counterpart of Example \ref{e:cir}. If
$c\to\infty$ and $\theta/(1+c)$ converges to a positive limit, then
the measure $\rho$ converges weakly to an exponential distribution,
i.e., to the kernel $\rho$ of Example \ref{e:cir}, and the Feller
random measures converge accordingly, by Proposition \ref{p:stability}. The kernel $\rho = \delta_{0}$ of Example \ref{e:delta}
corresponds formally to $c = 0$.

\smallskip

Since $\Pi$ is diffuse, $(\Pi*\xi)(\dt) = \zeta_{t}\,\dt$ with
$\zeta_{t} := c\,\theta\int_{[0,t)}e^{-\theta(t-s)}\,\xi(\ds)$.
Therefore
\begin{align}\label{exp-intensity}
  \Gamma_{t} = S_{\Gamma^{p}_{t}}, \qquad \Gamma^{p}_{t} = p\,\mu[0,t]
  + p\textstyle\int_{0}^{t}\zeta_{s}\,\ds, \qquad \diffd\zeta_{t} =
  -\theta\,\zeta_{t}\,\dt + c\,\theta\,\diffd\Gamma_{t}.
\end{align}
If $\mu$ is diffuse, then $\Gamma^{p}$ is continuous. By part
\ref{i:pred-jump} of Theorem \ref{t:ig-rep}, $N$ is then
quasi-left-continuous, and all jumps of $\Gamma$ are totally
inaccessible. This is the opposite of the lattice cases, where every
jump occurs at a predictable time.

For $\mu(\dt) = m\,\dt$ with $m>0$, the jumps of $\Gamma$ have the
compensator $p\,(m+\zeta_{t})\,\dt\,\nu(\diffd z)$, where $\nu$ is the
L\'evy measure of $S$. Using part \ref{i:indep-sub} of Theorem
\ref{t:ig-rep} and the uniqueness in Theorem \ref{t:disc-eqs}, one can
show that $\zeta$ is then a Markov process with generator
\begin{align*}
  \mathcal{Q} g(x) = -\theta\,x\,g'(x) + p\,(m+x)\textstyle\int_{\sp}
  \bprn{ g(x+c\,\theta z) - g(x) }\,\nu(\diffd z).
\end{align*}
Thus $\zeta$ is a continuous-state branching process with
immigration, without a diffusion part and with inverse-Gaussian jumps.
Since $\int z\,\nu(\diffd z) = 1$ and $p\,c = q$, its mean reverts to
$c\,m$ at rate $p\,\theta$. 
\end{example}

\appendix

\section{Dawson-Watanabe superprocesses}
\label{S:spier}

This appendix collects what we use about the Dawson-Watanabe
superprocess $\sfamtau{X_{\tau}}$ with \define{spatial motion} $K$ (the
killed subordinator of subsection \ref{s:subord}) and \define{branching
mechanism} $z^{2}/2$. We work first with finite measures and extend
to locally finite ones in Remark \ref{r:lf-ext}. Our reference
throughout is \cite{Li23}; see \cite{Eth00} for an overview in
the Brownian case. 

\subsection{Construction}
\label{s:construct}
For $F \in \Bp(\sp)$, the \emph{log-Laplace equation}
\begin{align}
\label{log-laplace}
V_{\tau}[F](t) = \eeup{t}{ F(K_{\tau}) } - \eeup{t}{
\int_{0}^{\tau}\tot  V_{\tau - \sg}[F]^{2}\prn{ K_{\sg} }\,\dsg},
\quad \tau \in \tm,\ t \in \sp,
\end{align}
has a unique nonnegative, locally bounded, measurable solution
\cite[Proposition 2.20]{Li23}. It is a weak form of
$\frac{\diffd}{\dtau} V_{\tau}[F] = \sA\, V_{\tau}[F] - 
V_{\tau}[F]^{2}/2$, $V_{0}[F] = F$. The \define{cumulant semigroup}
of its solutions
satisfies $V_{\tau_{1}+\tau_{2}}[F] = V_{\tau_{1}}\bkt{V_{\tau_{2}}[F]}$
\cite[Theorem 2.21]{Li23}. For $\mu \in \Mf(\sp)$ and $\tau \in
\tm$, the map $F \mapsto e^{-\ab{\mu,V_{\tau}[F]}}$, $F \in \Bcp(\sp)$,
is the Laplace functional of a probability measure on $\Mf(\sp)$. These
measures are the transition kernels of a Markov semigroup, whose
Chapman-Kolmogorov property is the semigroup property of $V$. We
denote by $\set{\bP^{\mu}}_{\mu \in \Mf(\sp)}$ the associated Markov
family on the space $D_{\Mf(\sp)}$ of \cd{} paths.

\subsection{First properties}
\label{s:first-prop}
Since $K$ is Feller, the family $\set{\bP^{\mu}}$ is a Hunt process
\cite[Corollary 5.12]{Li23}. In particular, it is quasi-left-continuous and strongly Markov for the natural and the universally
completed filtrations. Under $\bP^{\mu}$, the canonical process $X$ is
the Dawson-Watanabe superprocess with spatial motion $K$ and branching
mechanism $z^{2}/2$, and
\begin{align}\label{dw-laplace}
\eemu{ e^{- \ab{X_{\tau},F}} } = e^{- \ab{\mu,V_{\tau}[F]}}.
\end{align}
\begin{proposition}
\label{p:dw-props}
Suppose that $\mu \in \Mf(\sp)$. Then
\begin{enumerate}
  \item \label{i:first-moment} $ \eemu{ \ab{X_{\tau},\vp} } = \ab{ \mu, P_{\tau}\vp
    }$ for $\vp \in \Bcp(\sp)$ and $\tau \in \tm$.
  \item \label{i:branching} The
    \define{branching property}: under $\bP^{\mu_{1}+\mu_{2}}$, $X$ is
    distributed as the sum of two independent copies, under
    $\bP^{\mu_{1}}$ and under $\bP^{\mu_{2}}$.
  \item \label{i:mass-feller} The \define{total mass process} $m_{\tau} =
    \ab{X_{\tau},1}$ is a Feller diffusion,
    \begin{align}\label{mass-sde}
      \dm_{\tau} = - m_{\tau}\,\dtau + \sqrt{m_{\tau}}\,\dW_{\tau},\
      m_0 = \ab{\mu,1},
    \end{align}
    whatever the law of $L$. It is absorbed at $0$ in finite time, a.s.,
    so $X$ goes extinct in finite time, a.s.
  \item \label{i:support}
    $\set{ \supp X_{\tau}  \subseteq [t,\infty) } \subseteq
    \set{ \supp X_{\sg}  \subseteq [t,\infty) }$
    for all $ \sg \geq \tau$ and $t \in \sp$, modulo $\bP^{\mu}$-null
    sets.
\end{enumerate}
\end{proposition}
\begin{proof}
Part (\ref{i:first-moment}) is \cite[Proposition 2.24]{Li23},
and (\ref{i:branching}) follows from \eqref{dw-laplace}. Part
(\ref{i:mass-feller}) is \cite[Example 3.2]{Li23}. The state $0$ is absorbing, since the drift in
\eqref{mass-sde} has no constant term, and it is reached in finite
time, a.s. \cite[Proposition 1.2.15]{Alf15}.

For (\ref{i:support}), Lemma \ref{l:one-sided} shows that
$P_{\tau}\ind{[0,t)}$ vanishes on $[t,\infty)$. By
(\ref{i:first-moment}), $\eeup{\nu}{ X_{\tau}[0,t) } = 0$ for every $\nu
\in \Mf(\sp)$ carried by $[t,\infty)$. Hence $X_{\tau}[0,t) = 0$ a.s.\
for each $\tau$, and, by right-continuity of $X$ and the
portmanteau theorem, since $[0,t)$ is open, for all $\tau$
simultaneously. The claim follows from the Markov property at $\tau$.
\end{proof}

\begin{remark}
\label{r:lf-ext}
For general superprocesses, the passage to locally finite measures
needs appropriate tail restrictions. Here it is simple, thanks to the one-sidedness
of $K$. Let $\mu$ be finite and $t<t'$. By the branching property and
Proposition \ref{p:dw-props}(\ref{i:support}), mass started in
$[t,t')$ never enters $[0,t)$, so the restrictions to $[0,t)$ of the
superprocesses started from $\mu|_{[0,t')}$ and from $\mu|_{[0,t)}$
have the same law. Hence, for $\mu \in \Mlf(\sp)$, the laws of the
restrictions to $[0,t)$ of the superprocesses started from
$\mu|_{[0,t)}$, $t \in \bN$, are consistent. Their projective limit
(Kolmogorov's extension theorem on the Polish path spaces) is a law
$\bP^{\mu}$ of an $\Mlf(\sp)$-valued process with vaguely \cd{} paths.
We take $\sfammu{\bP^{\mu}}$ to be the Markov family of the canonical
process $X$ on $D_{\Mlf(\sp)}$. This is the setup used throughout the
paper. All the properties above, except those of the total mass and
extinction, carry over.
\end{remark}

\subsection{The martingale measure}
\label{s:mart}
Let $\sfamtau{\sF_{\tau}}$ be the universally completed filtration
of $X$,
\begin{align}\label{dw-filt}
\sF_{\tau} = \bigcap_{\mu \in \Mlf(\sp)}\sF_{\tau}^{\mu}, \ewhere
\sF_{\tau}^{\mu} = \sigma\prn{ X_{\tau'},\tau' \leq \tau } \vee
\sZ^{\mu},
\end{align}
and $\sZ^{\mu}$ is the family of $\bP^{\mu}$-null sets. Martingales
are understood with respect to $\sfamtau{\sF_{\tau}}$ and all
$\bP^{\mu}$ simultaneously. For $\vp \in \sD(\sA)$ with bounded support
(no restriction is needed when $\mu$ is finite), the process
\begin{align}\label{mart-prob}
M_{\tau}(\vp) := \ab{X_{\tau},\vp} - \ab{X_{0},\vp}
- \int_{0}^{\tau}\ab{X_{\sg},\sA\vp}\,\dsg
\end{align}
is a continuous martingale \cite[Theorem 7.16]{Li23} with
\begin{align}\label{mart-bracket}
\bkt{ M(\vp),M(\vp') }_{\tau} =
\int_{0}^{\tau}\ab{X_{\sg},\vp\vp'}\,\dsg,\quad \tau \geq 0 .
\end{align}

\subsection{The Walsh integral}
\label{s:walsh}
By \eqref{mart-bracket}, $\eemu{M_{\tau}(\vp)^{2}} =
\eemu{\int_{0}^{\tau}\ab{X_{\sg},\vp^{2}}\,\dsg}$, so $\vp \mapsto
M_{\tau}(\vp)$ is an isometry into $\Ltwo(\bP^{\mu})$ from $\Ltwo$ of the
mean measure $(\mu P_{\sg})(\dt)\,\dsg$ on $\sp \times [0,\tau]$. It
therefore extends from $\sD(\sA) \supseteq \Conec(\sp)$ to all bounded
Borel $\vp$ with bounded support (to all bounded Borel $\vp$ if $\mu$ is
finite), and, for bounded Borel $A$, $M_{\tau}(A) := M_{\tau}(\ind{A})$ is an orthogonal, hence worthy, martingale measure
with covariance measure $X_{\tau}(\dt)\,\dtau$ (see \cite{Wal86} and
\cite[Chapter 7]{Li23}). For predictable $\psi = \psi(t,\sg,\omega)$
with $\eemu{ \int_{0}^{\tau} \ab{X_{\sg},\psi(\cdot,\sg)^{2}}\,\dsg } <
\infty$ for all $\tau$ (the class $\Ltwo(M)$), the \define{Walsh
integral}
\begin{align*}
M_{\tau}(\psi)=\int_{0}^{\tau}\int_{\sp}\psi(t,\sg)\, M(\dt,\dsg)
\end{align*}
is a continuous $\Ltwo$-martingale with quadratic variation
$\int_{0}^{\tau}\ab{X_{\sg},\psi(\cdot,\sg)^{2}}\,\dsg$. By stopping, it
extends to a local martingale for $\psi$ with $\int_0^{\tau}
\ab{X_{\sg}, \psi(\cdot, \sg)^{2}}\,\dsg < \infty$ a.s.\ for all
$\tau$.

\subsection{The SPDE}
\label{s:SPDE}
If $X_{\tau}(\dt) = Y(\tau,t)\,\dt$, then \eqref{mart-prob} is the
weak formulation of the SPDE
\begin{align}\label{spde}
\tfrac{\diffd}{\dtau} Y = \sA^* Y + \sqrt{Y} \dot{W},
\end{align}
with $W$ a space-time white noise and $\sA^*$ the adjoint
\eqref{adjoint}. Its mild form, valid without a density, is the
\define{convolution representation} \cite[Theorem 7.32]{Li23}
\begin{align}\label{conv-rep}
\ab{X_{\tau},\vp} = \ab{\mu,P_{\tau}\vp}
+ \int_{0}^{\tau}\int_{\sp}\prn{ P_{\tau - \sigma}\vp }(t)\,
M\prn{\dt,\dsg}, \quad \vp \in \Bc(\sp).
\end{align}

\section{Proofs for Section \ref{S:defs}}
\subsection{Proof of Proposition \ref{p:unique}}
\label{s:prf-unique}
Let $T$ be such that $\supp f \subseteq
[0,T]$,
let $v \geq 0$ be a bounded solution of \eqref{riccati}, and let
$U$ be the operator defined in \eqref{conv}.
Since $U(v^{2})/2 \geq 0$,
\eqref{riccati} gives $v \leq Uf \leq \norm{f}_{\infty}$. For $t > T$
we have $f(t+s) = 0$ for every $s \geq 0$, so that 
$v = -U(v^{2})/2 \leq 0$ and hence $v = 0$ on $(T,\infty)$.

Recalling the convention that all functions and measures are extended
trivially to $(-\infty,0)$, we observe that, since $f$ and $v$ both vanish above $T$, the reflections $\tv(t)
:= v(T-t)$ and $\tilde f(t) := f(T-t)$ vanish for $t < 0$.
Thus, \eqref{riccati} turns into
\begin{align}\label{reflect}
  \tv = \prn{ \tilde f - \tot \tv^{2} } * \rho \efor
  t \in [0,T],
\end{align}
all convolutions here and below being understood with functions so extended.
Up to a sign, \eqref{reflect} is the convolutional Riccati equation of \cite{PacZit25}. Its kernel $\rho$ may, however, charge the origin, so we transform the equation before applying the results of \cite{PacZit25}.
Let
$p := \rho(\set{0})$ and let
$\rho_{0}$ be the restriction of $\rho$ to $(0,\infty)$. The
equation  \eqref{reflect} now
reads
\begin{align}\label{riccati-fix}
  F(\tv)
  = \tilde f * \rho - \tot \tv^2 * \rho_0,
  \ewhere
  F(x) = x+ \tot p x^2.
\end{align}

Let now $v$ and $w$ be two bounded nonnegative solutions of
\eqref{riccati}, and let $\tv$ and $\tw$ be their reflections, so that
both satisfy \eqref{riccati-fix}. Since
$\abs{F(x) - F(y)} \geq \abs{y-x}$ for all $x,y \in \sp$,
subtracting the equations \eqref{riccati-fix} for $\tv$ and $\tw$
gives $\abs{\tv-\tw} \leq \prn{ \abs{\tv^{2}-\tw^{2}} * \rho_{0}
}/2$. By the first paragraph $0 \leq \tv,\tw \leq
\norm{f}_{\infty}$, so that $\abs{\tv^{2}-\tw^{2}} =
\abs{\tv-\tw}\,(\tv+\tw) \leq 2\norm{f}_{\infty}\abs{\tv-\tw}$
and therefore
\begin{align*}
  \abs{\tv - \tw} \leq M
  \prn{ \abs{\tv - \tw} * \rho_{0} } \ewhere M = \norm{f}_{\infty} .
\end{align*}
We now repeat, for self-containment, the argument of \cite[Proposition
A.1]{PacZit25}. The running supremum $m(t) :=
\sup_{[0,t]}\abs{\tv - \tw}$ is bounded and satisfies $m \leq M\prn{m *
\rho_{0}}$ on $[0,T]$, because $m * \rho_{0}$ is nondecreasing.
Convolution with $\rho_{0}$ preserves inequalities, so iterating gives,
for $k \in \mathbb{N}$ and $t \in [0,T]$,
\begin{align*}
  m(t) \leq M^{k}\prn{m * \rho_{0}^{*k}}(t) \leq
  M^{k}\norm{m}_{\infty}\,\rho_{0}^{*k}[0,T] \leq
  \norm{m}_{\infty}\,e^{\ld T}\prn{M\hrho_{0}(\ld)}^{k}
  \efor \ld>0,
\end{align*}
where $\hrho_{0}(\ld) = \int_{(0,\infty)} e^{-\ld s}\,\rho(\ds)$. By
dominated convergence, $\hrho_{0}(\ld) \to \rho_{0}(\set{0}) = 0$ as
$\ld \to \infty$. Choosing $\ld$ with $M\hrho_{0}(\ld)<1$ and letting
$k \to \infty$ yields $m = 0$, i.e.\ $\tv = \tw$ and so
$v = w$ on $[0,T]$. Since both $v$ and $w$ vanish on $(T,\infty)$, it follows that $v = w$.

\section{Proofs for Section \ref{S:repr}}
\subsection{Proof of Theorem \ref{t:main}}
\label{s:prf-main}
Let $\xi$ be given by \eqref{occup}. By Definition \ref{d:frm} it
is enough to produce, for each $f \in \Bcp(\sp)$, a bounded
nonnegative solution of \eqref{riccati} satisfying \eqref{laplace}. We fix such an $f$ for the rest of the proof and set
$T_{f} := \sup\supp f < \infty$.
The first-moment formula of Proposition
\ref{p:dw-props}(\ref{i:first-moment}), available for
$\mu \in \Mlf(\sp)$ by Remark \ref{r:lf-ext}, 
gives
\begin{align*}
  \eemu{\xi[0,t]}
  = \int_0^{\infty}\ab{ \mu, P_{\tau}\ind{[0,t]} }\,\dtau
  = \ab{ \mu, U\ind{[0,t]} }
  \leq \mu[0,t] < \infty ,
\end{align*}
because $U\ind{[0,t]}(s) = \rho[0,t-s] \leq 1$ for $s \leq t$ and
$U\ind{[0,t]}(s) = 0$ for $s>t$.
In particular $\xi$ is a.s.\ locally finite, and
$\ab{\xi,f} = \int_{\tm}\ab{X_{\sg},f}\,\dsg$.

Assume for the moment that $\mu \in \Mf(\sp)$. The general case is recovered at the end of the proof. A Feynman-Kac-type extension of
the log-Laplace equation \eqref{log-laplace} (see
\cite[Theorem 5.15]{Li23}) gives
\begin{align}\label{fk}
  \eeup{\tau, \mu}{ e^{- \int_{\tau}^{T}\ab{X_{\sg},f}\,\dsg} }
  = e^{- \ab{\mu,v^{T}(\tau, \cdot )}}
  \efor 0 \leq \tau \leq T < \infty,
\end{align}
where $v^{T}$ is the unique nonnegative bounded solution on
$[0,T]\times\sp$ of
\begin{align}\label{mild}
  v^{T}(\tau,t) + \eeup{\tau,t}{ \int_{\tau}^{T}\tot \prn{ v^{T}\prn{
  \sg,K_{\sg} } }^{2}\,\dsg }
  = \eeup{\tau, t}{ \int_{\tau}^{T}f\prn{ K_{\sg} }\,\dsg } .
\end{align}
Taking $\tau = 0$ and $\mu = \delta_{t}$ in \eqref{fk} and letting
$T \to \infty$, we see that $v^{T}(0,t)$ increases to a limit
$v^{\infty}(t)$. By time-homogeneity $v^{T}(\tau,t) = v^{T-\tau}(0,t)$,
so in fact $v^{T}(\tau,t) \nearrow v^{\infty}(t)$ for every $\tau$.
Dropping the second term on the left of \eqref{mild} gives, in
addition, the uniform bound
\begin{align*}
  0 \leq v^{T}(\tau,t)
  \leq \eeup{\tau,t}{ \int_{\tau}^{T}f(K_{\sg})\,\dsg }
  \leq \norm{f}_{\infty}\int_{\tau}^{T}e^{-(\sg-\tau)}\,\dsg
  \leq \norm{f}_{\infty}. 
\end{align*}

We may now set $\tau = 0$ in \eqref{mild} and let $T \to \infty$.
Monotone convergence applies to both sides, and \eqref{pot-op}
allows us to express the result as 
\begin{align}
  \label{v-inf}
  v^{\infty} = U\prn{ f - \tot\prn{v^{\infty}}^{2} } ,
\end{align}
which is \eqref{riccati}. By Proposition \ref{p:unique}, \eqref{v-inf}
has a unique bounded nonnegative solution, so $v^{\infty} = v[f]$.

Finally, let $\mu \in \Mlf(\sp)$ be arbitrary and fix $a >
T_{f}$. We check that neither side of \eqref{fk}, at $\tau = 0$,
changes when $\mu$ is replaced by $\mu|_{[0,a)}$. On the left,
$\int_{0}^{T}\ab{X_{\sg},f}\,\dsg$ is a functional of the
restriction of $X$ to $[0,a)$, because $f$ is supported by
$[0,T_{f}]$, and Remark \ref{r:lf-ext} gives that restriction the
same law under $\bP^{\mu}$ as under $\bP^{\mu|_{[0,a)}}$. On the
right, Lemma \ref{l:one-sided} makes $\bE^{0,t}\bkt*{
\int_{0}^{T}f(K_{\sg})\,\dsg }$ vanish for $t > T_{f}$. Since the
two terms on the left of \eqref{mild} are nonnegative,
$v^{T}(0,\cdot)$ vanishes on $(T_{f},\infty)$, and so
$\ab*{\mu,v^{T}(0,\cdot)} = \ab*{\mu|_{[0,a)},v^{T}(0,\cdot)}$.
Since $\mu|_{[0,a)}$ is finite, \eqref{fk} holds for it and
therefore also for $\mu$. Letting $T \to \infty$ gives $\eemu{
\exp*{-\ab{\xi,f}} } = \exp*{-\ab*{\mu,v[f]}}$, which is
\eqref{laplace}.

\subsection{Proof of Proposition \ref{p:basic}}
\label{s:prf-basic}
(\ref{i:moments}) was proven in
\cite[Proposition 3.7]{PacZit25}, and
the proof there does not depend on the assumption that
$\rho(\set{0})=0$. Statement (\ref{i:sum}) follows from \eqref{laplace}, since $\xi_{1}$ and $\xi_{2}$ have the same $v[f]$ by Proposition \ref{p:unique}. It is also recorded in \cite[Proposition 3.9]{PacZit25}. Statement (\ref{i:diffuse}) is the content of \cite[Proposition 3.14]{PacZit25}.

Both (\ref{i:mass-ig}) and (\ref{i:locality}) below use the existence of auxiliary Feller random
measures, with parameters $(\rho,\delta_{t})$ and $(\rho,\mu|_{[0,T]})$
respectively. Their existence follows from Corollary \ref{c:exists}.

\smallskip

(\ref{i:mass-ig}) For $\ld \geq 0$ let $v_{T} := v[\ld\,\ind{[0,T]}]$.
Applying \eqref{laplace} to a $(\rho,\delta_{t})$-Feller random
measure shows that the map $f \mapsto v[f]$ is
nondecreasing and so is $T \mapsto v_{T}$.
The relation
\eqref{shift} gives
$v_{T}(t) = v_{T-t}(0)$, so the increasing limit of $v_{T}$ does not
depend on $t$. By Proposition \ref{p:unique} it is a number
$\kappa := \lim_{T \to \infty}v_{T}(0) \in [0,\ld]$. Passing to
the limit in \eqref{riccati} gives $\kappa = \ld - \kappa^{2}/2$,
and, by positivity, $\kappa = \sqrt{1+2\ld}-1$. Monotone
convergence on both sides of \eqref{laplace} then implies
\eqref{mass} for $\ld \geq 0$, which is enough to identify the
law of $\xi(\sp)$ as inverse-Gaussian with mean and variance
$\mu(\sp)$. That law has exponential moments of every order below
$1/2$, and $\ee{\exp{\xi(\sp)/2}} = e^{\mu(\sp)}$, so
both sides of \eqref{mass} are finite and analytic on $\ld >
-1/2$ and continuous at $\ld = -1/2$. The identity
therefore extends to the whole range stated.

\smallskip

(\ref{i:locality}) Let $f \in \Bcp(\sp)$ be supported by $[0,T]$. By the second
statement of Proposition \ref{p:unique} the solution $v[f]$ vanishes on
$(T,\infty)$, so that $\ab{\mu,v[f]} = \ab{\mu|_{[0,T]},v[f]}$ in
\eqref{laplace}. Moreover, for $t \leq T$ the integrand in \eqref{riccati}
vanishes for $s > T-t$, so that $v[f]$ depends on $\rho$ there only
through $\rho|_{[0,T]}$. Since the $f \in \Bcp(\sp)$ supported by
$[0,T]$ form a measure-determining class for the restriction of $\xi$
to $[0,T]$ \cite[Corollary 2.3]{Kal17}, the first statement
follows. 

For the second statement, we may assume that $0 \leq \ld \leq
1/2$, since the left-hand side of \eqref{expmom} is at most $1$
for $\ld \leq 0$. Let $\xi$ be a $(\rho,\mu)$-Feller random
measure, and let $\xi'$ be a $(\rho,\mu|_{[0,T]})$-Feller random
measure. Using the first statement, and (\ref{i:mass-ig}) for
$\xi'$ with $\ld$ replaced by $-\ld$, we get
\begin{align*}
  \ee{ \exp{ \ld \xi[0,T]} }
  &= \ee{ \exp{ \ld \xi'[0,T]} }
  \leq \ee{ \exp{ \ld \xi'(\sp)} }
  \\ &= \exp{ \mu[0,T] (1- \sqrt{1-2\ld})} \leq \exp{ \mu[0,T] }
  \efor 0 \leq \ld \leq \tot.
\end{align*}

\subsection{Proof of Proposition \ref{p:stability}}
\label{s:prf-stability}
Fix a nonnegative $f \in \Cc(\sp)$, let $T := \sup\supp f$, and
write $v_{n} := v_{n}[f]$ and $v := v[f]$. By Proposition
\ref{p:unique}, $0 \leq v_{n}, v \leq \norm{f}_{\infty}$, and all of
them vanish on $(T,\infty)$.

We first show that $\seq{v_{n}}$ is equicontinuous. Let $\zeta_{n}$
be the $(\rho_{n},\delta_{0})$-Feller random measure, which exists by
Corollary \ref{c:exists}. By \eqref{laplace} and \eqref{shift},
$e^{-v_{n}(t)} = \ee{e^{-\ab{\zeta_{n},\theta_{t}f}}}$ for $t \in \sp$.
Since $\ee{\zeta_{n}(\sp)} = \rho_{n}(\sp) = 1$ by Proposition
\ref{p:basic}(\ref{i:moments}), we have, for $s,t \in \sp$,
\begin{align*}
  \babs{ e^{-v_{n}(t)} - e^{-v_{n}(s)} } \leq
  \ee{ \abs{ \ab{\zeta_{n},\theta_{t}f - \theta_{s}f} } } \leq
  \omega_{f}\prn{\abs{t-s}},
\end{align*}
where $\omega_{f}$ is the modulus of continuity of $f$. Since
$e^{-v_{n}} \geq e^{-\norm{f}_{\infty}}$, it follows that
$\abs{v_{n}(t) - v_{n}(s)} \leq e^{\norm{f}_{\infty}}\,
\omega_{f}\prn{\abs{t-s}}$. The same bound holds for $v$.

By the Arzel\`a-Ascoli theorem, every subsequence of $\seq{v_{n}}$
has a further subsequence which converges uniformly on $[0,T+1]$, and
hence on $\sp$, to a continuous function $w$ with $0 \leq w \leq
\norm{f}_{\infty}$. Along it, $h_{n} := f - v_{n}^{2}/2$ converges
uniformly to $h := f - w^{2}/2$, which is bounded and continuous.
With $U_{n}$ and $U$ as in \eqref{conv}, for $\rho_{n}$ and $\rho$, we
have
\begin{align*}
  \abs{ U_{n}h_{n}(t) - Uh(t) } \leq \norm{h_{n} - h}_{\infty} +
  \babs{ \textstyle\int_{\sp} h(t+s)\,\prn{\rho_{n}-\rho}(\ds) },
\end{align*}
and both terms tend to $0$, the second one by weak convergence.
Hence $w = U\prn{f - w^{2}/2}$, and $w = v$ by Proposition
\ref{p:unique}. It follows that $v_{n} \to v$ uniformly on $\sp$.

Finally, vague convergence gives $\sup_{n}\mu_{n}[0,T] < \infty$,
and $v$ is continuous with support in $[0,T]$. Therefore
\begin{align*}
  \abs{ \ab{\mu_{n},v_{n}} - \ab{\mu,v} } \leq \mu_{n}[0,T]\,
  \norm{v_{n}-v}_{\infty} + \abs{ \ab{\mu_{n}-\mu, v} } \to 0 .
\end{align*}
By \eqref{laplace}, the Laplace transforms of the
$(\rho_{n},\mu_{n})$-Feller random measures at $f$ converge to that of
the $(\rho,\mu)$-Feller random measure, for every nonnegative $f \in
\Cc(\sp)$. This implies convergence in distribution with respect to
the vague topology \cite[Chapter 4]{Kal17}.

\subsection{Proof of Corollary \ref{c:vague-stability}}
\label{s:prf-vague-stability}
Since $\rho_{n}(\sp) = 1$, vague convergence implies that 
$\int
g\,\diffd\rho_{n} \to \int g\,\diffd\rho_{\infty}$ for every 
$g \in \Cz(\sp)$. In particular, 
$\hrho_{n}(\ld) \to
\hrho_{\infty}(\ld)$ for $\ld > 0$.

Suppose first that $\rho_{\infty} \neq 0$. Then $\hrho_{\infty} > 0$
on $(0,\infty)$, and $\Psi_{n} = 1/\hrho_{n}$ converges pointwise on
$(0,\infty)$ to $\Phi := 1/\hrho_{\infty}$. As a pointwise limit of
Bernstein functions, $\Phi$ is a Bernstein function
\cite[Corollary 3.8(ii)]{SchSonVon12}. Since $\Psi_{n} \geq \Psi_{n}(0) =
1$, the limit $a := \Phi(0+)$ satisfies $a \geq 1$, and $\Phi - a$ is a
Bernstein function vanishing at the origin. Put $c := 1/a$ and $\rho_{0}
:= a\,\rho_{\infty}$. Then $\rho_{0}(\sp) = a\,\hrho_{\infty}(0+) = 1$,
and $1/\hrho_{0} - 1 = (\Phi - a)/a$ is a Bernstein function, so that
$\rho_{0}$ is PGID by Lemma \ref{l:pgid}.

The proof of Proposition \ref{p:stability} now applies with two
changes. First, the function $s \mapsto h(t+s)$ is in $\Cc(\sp)$, 
so that the second term converges to $0$ by
vague convergence alone. Second, the limit $w$ solves $w = c\,U_{0}\prn{f
- w^{2}/2}$, where $U_{0}$ is the operator \eqref{conv} for
$\rho_{0}$. By 
Remark \ref{r:cont-disc}(\ref{i:scaling})
with $\sigma = 1$, the function
$c\,w$ solves \eqref{riccati} for $\rho_{0}$ with $f$ replaced by
$c^{2}f$. Hence $c\,w = v_{0}[c^{2}f]$ by Proposition \ref{p:unique},
where $v_{0}$ denotes the solution for $\rho_{0}$. Therefore the
Laplace transforms at $f$ converge to
\begin{align*}
  e^{-\ab{\mu,w}} = e^{-\ab{\mu/c,\,v_{0}[c^{2}f]}}
  = \ee{ e^{-\ab{c^{2}\xi_{0},f}} },
\end{align*}
and the claim follows as before. If $\rho_{\infty} = 0$, the same
argument gives $w = 0$, and the limit is $0$.

\subsection{Proof of Proposition \ref{p:weak}}
\label{s:prf-weak}
For \eqref{weak-gen}, we let $\tau \to \infty$ in the
defining equation \eqref{mart-prob} of $M$. Since $\Cc(\sp)
\subseteq \Ltn$, the martingale $\tau \mapsto M_{\tau}(\vp)$
converges almost surely and in $\Ltwo$ toward $N(\vp)$, as noted
after \eqref{noise}.

The integral term converges to
$\ab{\xi,\sA\vp}$ by dominated convergence, since $\sA\vp \in
\Cc(\sp)$: it vanishes above $\sup\supp\vp$ by Lemma
\ref{l:one-sided}, and is continuous because $\sA$ maps $\sD(\sA)$ into
$\Cz(\sp)$.

To show that $\ab{X_{\tau},\vp} \to 0$, we fix $T \in \sp$ such
that the support of $\vp$ is a subset of $[0,T)$. Then
$\ab{X_{\tau},\vp}$ is a functional of the restriction of $X$ to
$[0,T)$. By Remark \ref{r:lf-ext}, this restriction agrees in law
with the corresponding restriction of the superprocess started
from the finite measure $\mu|_{[0,T)}$. The latter goes extinct
in finite time by Proposition
\ref{p:dw-props}(\ref{i:mass-feller}), so that $\ab{X_{\tau},\vp}
= 0$ for all large $\tau$, a.s. Substituting these three limits
into \eqref{mart-prob} gives \eqref{weak-gen}. To prove
\eqref{weak-res}, it suffices to substitute $\vp = U \psi \in
\Cc(\sp)$ in \eqref{weak-gen}. Then $\vp \in \sD(\sA)$ and $\sA
\vp = -\psi$.

\subsection{Proof of Proposition \ref{p:noise-mg}}
\label{s:prf-noise-mg}
The proof is divided into three steps.
Throughout, we write
$N_{t}(\vp) := N\prn{\vp\,\ind{[0,t]}}$ whenever
$\vp\,\ind{[0,t]} \in \Ltn$, so that, e.g.,
$N_{t} = N_{t}(1)$.

\smallskip

\emph{Step 1.} This step identifies the raw vertical
$\sigma$-algebras with those generated by the martingale measure on the
strips $[0,t] \times \tm$. We set 
\begin{align}\label{restfilt}
  \sF^{[0,t]}_{\tau} := \sigma\prn{ X_{\sg}(A) : \sg \leq \tau,\,
  A \in \sB[0,t] }, \quad \tau \in \tm.
\end{align}
With $\Ltwo(M)$ as defined in
subsection \ref{s:walsh}, let $\Ltwo_{[0,t]}(M)$ consist of those $\phi \in
\Ltwo(M)$ which are supported by $[0,t] \times \tm$, predictable with
respect to $\sfamtau{\sF^{[0,t]}_{\tau}}$ (and, therefore, also
with respect to $\sfamtau{\sF_{\tau}}$)
and satisfy
$\seemu{ \int_{\tm}\ab{ X_{\sg},\phi(\cdot,\sg)^{2} }\,\dsg } <
\infty$ (e.g., bounded $\phi$).
With
\begin{align*}
  \sV^{0,M}_{t} := \sigma\sets{ \tint_{[0,t]} \tint_{\tm}
  \phi(s,\sigma)\, M(\ds, \dsg)}{ \phi \in \Ltwo_{[0,t]}(M)} \efor t
  \in \sp, 
\end{align*}
we claim that
\begin{align}\label{filt-eq}
  \sV^{0}_{t} \vee \sZ^{\mu} = \sV^{0,M}_{t} \vee \sZ^{\mu},
  \eforall t \in \sp.
\end{align}
For $A \in \sB[0,t]$,
Lemma \ref{l:one-sided} implies that
$P_{\tau} \ind{A}$ vanishes off $[0,t]$
for
all $\tau \in \tm$.
The convolution representation \eqref{conv-rep} for $\vp = \ind{A}$
then yields
\begin{align*}
X_{\tau}(A) = \ab{\mu,P_{\tau} \ind{A}}
+ \int_{[0,t]} \int_{0}^{\tau}
\prn{ P_{\tau - \sigma}\ind{A} }(s)\,
M\prn{\ds,\dsg} \in \sV^{0,M}_t \vee \sZ^{\mu}.
\end{align*}
Conversely, let $\Phi_{t}$ be the set of all $\vp \in \sD(\sA)$ with
$\supp \vp \subseteq [0,t]$ when $\rho(\set{0})=0$.
When $\rho(\set{0})>0$, $L$ is a compound Poisson process by
Proposition \ref{p:no-atom}, so that $b=0$ and $\Pi$ is finite. The
formula for $\sA$ in subsection \ref{s:subord} then defines a bounded
operator on $\B(\sp)$, with
$\tfrac{\diffd}{\dtau}P_{\tau}\vp = P_{\tau}\sA\vp$ in the supremum
norm, and we let $\Phi_{t}$ be the set of all $\vp \in \B(\sp)$
vanishing on $(t,\infty)$. Applying \eqref{conv-rep} to $\sA\vp$,
integrating over $[0,\tau]$ and using the stochastic Fubini theorem,
we see that \eqref{mart-prob} also holds for all $\vp \in \Phi_{t}$
in this case, with $M_{\tau}(\vp)$ as extended in subsection
\ref{s:walsh}. For
$\vp \in \Phi_{t}$ and
$0 \leq \tau_{1} < \tau_{2}$,
\eqref{mart-prob} implies that
\begin{align}
  \label{strip-id}
  \int_{[0,t]} \int_{(\tau_{1},\tau_{2}]} \hspace{-1.5em}
\vp(s)\,
M(\ds, \dsg) =
\ab{X_{\tau_{2}}-X_{\tau_{1}}, \vp} -
\int_{\tau_{1}}^{\tau_{2}} \ab{X_{\sg}, \sA \vp}\, \dsg \in
\sV^0_t \vee \sZ^{\mu},
\end{align}
where we used Lemma \ref{l:one-sided} again (or, if
$\rho(\set{0})>0$, the formula for $\sA$) to conclude that $\sA
\vp$ vanishes off $[0,t]$.
Finally, linear combinations of random fields of the form
$Z\,\ind{(\tau_{1}, \tau_{2}]}(\sg)\,\vp(s)$, with
$Z$ bounded and $\sF^{[0,t]}_{\tau_{1}}$-measurable and
$\vp \in \Phi_{t}$, are dense in $\Ltwo_{[0,t]}(M)$. If
$\rho(\set{0})=0$, this uses that
$\ee{\xi}(\set{t}) = (\rho*\mu)(\set{t}) = 0$, since $\rho$ has no
atoms by Proposition \ref{p:no-atom}. The integral of such a field is
$Z\prn{ M_{\tau_{2}}(\vp) - M_{\tau_{1}}(\vp) } \in
\sV^{0}_{t} \vee \sZ^{\mu}$
by \eqref{strip-id}. By the isometry \eqref{mart-bracket}, the integrals of
such approximations of a general $\phi \in \Ltwo_{[0,t]}(M)$ converge in
$\Ltwo(\bP^{\mu})$, so the limit is
$\sV^{0}_{t}\vee\sZ^{\mu}$-measurable. This proves \eqref{filt-eq}.
In particular, each $N_{t}$ is
$\sV^{0}_{t}\vee\sZ^{\mu}$-measurable, hence has a
$\sV_{t}$-measurable version,
and it is square integrable by the Walsh isometry for the integrand
$\ind{[0,t]}$ and Proposition
\ref{p:basic}(\ref{i:moments}).

\smallskip

\emph{Step 2.} This step proves the conditional identity
\begin{align} \label{cond-one}
\eemu{ \sE(\ld\chi) \giv \sV^{0}_{s} } = 1 \eas
\end{align}
for the stochastic exponential $\sE(\ld\chi)$ of the Walsh integral
of $\ld\chi$, in the notation introduced below. The passage from
$\sV^{0}_{s}$ to $\sV_{s}$ is carried out in Step 3.
For a bounded $I \subseteq \sp$, let $\Linf_{I}(M)$ denote the set
of all bounded $\phi \in \Ltwo(M)$, $\bP^{\mu}$-a.s.\ supported by
$I \times \tm$.
For such $\phi$ we abbreviate $M(\phi) := M_{\infty}(\phi)$ and write
\begin{align*}
\sE_{\tau}(\phi) := \exp{ M_{\tau}(\phi) - \tot\bkt{ M(\phi)
}_{\tau} }, \quad \sE(\phi) := \sE_{\infty}(\phi),
\end{align*}
for the stochastic exponential of the continuous martingale
$\tau \mapsto M_{\tau}(\phi)$, whose bracket we denote by
$\bkt{ M(\phi) }$.

For $\phi \in \Linf_{[0,t]}(M)$ we have
$\bkt{ M(\phi) }_{\tau} \leq \norm{\phi}_{\infty}^{2}\,\Gamma_t$. If, in addition,
$\phi$ does not depend on the branching variable, then
$\bkt{ M(\phi) }_{\infty} = \ab{ \xi, \phi^{2} }$.

Fix $0 \leq s < t$, let $\chi \in \Linf_{(s,t]}(M)$ and $\ld
  \in \bR$
satisfy $\abs{\ld}\norm{\chi}_{\infty} \leq \tfrac{1}{4}$. Being a
nonnegative local martingale, $\sfamtau{\sE_{\tau}(4\ld\chi)}$ is a supermartingale, so
that $\eemu{ \sE_{\tau}(4\ld\chi) } \leq 1$ for all $\tau$. Since
\begin{align*}
\sE_{\tau}(\ld\chi)^{2} = \sE_{\tau}(4\ld\chi)^{1/2}\,
e^{3\ld^{2}\bkt{ M(\chi) }_{\tau}},
\end{align*}
the Cauchy-Schwarz inequality and \eqref{expmom} give
\begin{align}\label{L2bound}
\sup_{\tau \in \tm}\ 
\eemu{ \sE_{\tau}(\ld\chi)^{2} } \leq
\eemu{ e^{6\ld^{2}\norm{\chi}_{\infty}^{2}\,\xi[0,t]} }^{1/2} <
\infty,
\end{align}
since
$6\ld^{2}\norm{\chi}_{\infty}^{2} \leq \tfrac{3}{8}<
1/2$. The same bound holds with $\tau$ replaced by any stopping
time $S$, because $\eemu{ \sE_{S}(4\ld\chi) } \leq 1$ by Fatou's lemma
and $\bkt{ M(\chi) }_{S} \leq \norm{\chi}_{\infty}^{2}\,\xi[0,t]$.
Consequently $\sfamtau{\sE_{\tau}(\ld\chi)}$ is of class (D), hence a uniformly integrable martingale
with $\eemu{ \sE(\ld\chi) } = 1$, and
$\sup_{\tau}\sE_{\tau}(\ld\chi) \in \Ltwo(\bP^{\mu})$ by Doob's
inequality.

Given $\psi \in \Linf_{[0,s]}(M)$, the supports of $\psi$ and
$\chi$ are disjoint, so $M(\psi)$ and $M(\chi)$ are orthogonal by 
\eqref{mart-bracket}. Hence, by integration by parts,  
$\tau \mapsto M_{\tau}(\psi)\,\sE_{\tau}(\ld\chi)$ is a local martingale,
which vanishes at $\tau = 0$. It is dominated by
$\sup_{\tau}\abs{ M_{\tau}(\psi) }
\sup_{\tau}\sE_{\tau}(\ld\chi)$, which is integrable by the
Cauchy-Schwarz inequality, Doob's inequality and \eqref{L2bound}.
It is
therefore a true martingale, so that
$\eemu{ M(\psi)\,\sE(\ld\chi) } = 0$ and then 
\begin{align}\label{orth}
\eemu{ M(\psi)\prn{ \sE(\ld\chi) - 1 } } = 0.
\end{align}

Next, we extend this orthogonality from Walsh integrals with
respect to $M$ to all mean-zero random variables in
$\Ltwo(\sV^{0}_s)$. By Remark \ref{r:lf-ext} (with $[0,s]$ in
place of $[0,t)$), the restriction $X|_{[0,s]}$ is itself a
superprocess, with spatial motion $K$ killed upon exiting
$[0,s]$. Since the motion is one-sided, this killed motion has
the same generator as $K$ on functions supported by $[0,s]$. The
martingale measure of $X|_{[0,s]}$ is therefore the restriction
of $M$ to $[0,s] \times \tm$, and its integrands are predictable
with respect to the filtration \eqref{restfilt} (with $s$ in
place of $t$). We use the predictable representation theorem for
martingale measures (see \cite[Theorem 2.2 and subsection
3.1]{Ove95}). It states that every square-integrable
$\sV^{0,M}_{s}$-measurable (equivalently, by Step 1,
$\sV^{0}_{s}$-measurable) mean-zero random variable is a Walsh
integral of an integrand in $\Ltwo_{[0,s]}(M)$. Bounded
integrands are dense in $\Ltwo_{[0,s]}(M)$, and $\sE(\ld\chi) \in
\Ltwo(\bP^{\mu})$ by \eqref{L2bound}, so \eqref{orth} extends
from $\Linf_{[0,s]}(M)$ to all of $\Ltwo_{[0,s]}(M)$. Thus
$\sE(\ld\chi) - 1$ is orthogonal to every mean-zero element of
$\Ltwo(\sV^{0}_{s})$, so that $\seemu{ \sE(\ld\chi) \giv
\sV^{0}_{s} }$ is almost surely constant. Its mean is $1$, so
\eqref{cond-one} holds.

\smallskip

\emph{Step 3.} In this step we complete the proof of the
proposition.
We start by specializing
\eqref{cond-one} 
to $\chi := \vp \ind{(s,t]}$ with
$\vp \in \B(\sp)$; then $\vp\,\ind{[0,t]} \in \Ltn$ automatically,
since $\ee{\xi}[0,t] < \infty$. 
Since $\chi$ does not depend on the branching variable, we have,
for $\abs{\ld}\norm{\vp}_{\infty} \leq 1/4$,
\begin{align}
  \label{raw-exp}
  \eemu{ \exp{ \ld\prn{ N_{t}(\vp) - N_{s}(\vp) }
  - \tot\ld^{2}\tint_{(s,t]}\vp^{2}\,\dxi }
   \giv \sV^{0}_s} = 1 \eas.
\end{align}
The noise $N$ is right-continuous in $\bL^{2}$: by the Walsh
isometry we have
$\seemu{ \prn{N_{s+h}-N_{s}}^{2} } = \ee{\xi}(s,s+h] \to 0$
as $h \downto 0$. To replace $\sV^{0}_{s}$ by $\sV_{s}$ in
\eqref{raw-exp}, let $r \in (s,t)$. Since
$\sV_{s} \subseteq \sV^{0}_{r} \vee \sZ^{\mu}$, \eqref{raw-exp} with
$r$ in place of $s$ remains true when we condition on $\sV_{s}$. As
$r \downto s$, the random variables inside the conditional
expectation converge in probability, because
$N_{r}(\vp) \to N_{s}(\vp)$ in $\bL^{2}$ by the Walsh isometry and
$\xi(r,t] \to \xi(s,t]$. They are uniformly integrable by
\eqref{L2bound}, so \eqref{raw-exp} holds with $\sV_{s}$ in place of
$\sV^{0}_{s}$.
The bound \eqref{L2bound}, valid for $\ld$
in a neighborhood of $0$, lets us pass derivatives at $\ld=0$
inside the conditional expectation to obtain
\begin{align*}
\eemu{ N_{t}(\vp) - N_{s}(\vp) \giv \sV_{s}} = 0 \eand
\eemu{ \prn{ N_{t}(\vp) - N_{s}(\vp) }^{2}
- \tint_{(s,t]}\vp^{2}\,\dxi \giv \sV_{s}} = 0,
\end{align*}
for all $0 \leq s \leq t$. Taking $\vp \equiv 1$, we conclude
that $\prfis{N_{t}}$ and $\prfis{N_{t}^{2}-\Gamma_t}$ are
$(\bP^{\mu},\prfis{\sV_{t}})$-martingales, while \eqref{raw-exp}
itself, with $\sV_{s}$ in place of $\sV^{0}_{s}$, is the
martingale property of the exponential in \eqref{expmg}. The
range $\abs{\ld} \leq 1/4$ comes from the constraint
$\abs{\ld}\norm{\chi}_{\infty} \leq 1/4$ at $\chi = \ind{(s,t]}$.
The filtration \eqref{vfilt} is right-continuous and $N$ is
right-continuous in $\bL^{2}$, so $N$ admits 
an a.s.\ \cd{} modification adapted to the $\bP^{\mu}$-completion
of $\prfis{\sV_{t}}$, and \eqref{stoch-int} follows by
$\bL^{2}$-approximation.

%
%

\section{Proofs for Section \ref{S:cont}}
\subsection{Proof of Proposition \ref{p:noise-cont}}
\label{s:prf-noise-cont}
Since $\rho(\set{0}) = 0$,
the measure $\xi$ has no atoms by
Proposition \ref{p:basic}(\ref{i:diffuse}), so $t \mapsto \xi[0,t]$
is continuous and
nondecreasing, and finite for every $t$. In particular
$\ee{N_{0}^{2}} = \ee{\xi}(\set{0}) = 0$ by the Walsh isometry, so
that $N_{0} = N_{0-} = 0$.
Moreover, by \eqref{expmg},
the process $\exp{ \ld N_{t} - \ld^{2}\xi[0,t]/2 }$
is a positive $\prfis{\sV_{t}}$-martingale for $\ld$ in a
neighborhood of $0$.

Let $B$ be a Brownian motion independent of everything else,
possibly defined on an extension of the probability space, and put
\begin{align*}
  \tN_{t} := N_{t} + B_{t} \eand \tA_{t} := \xi[0,t] + t .
\end{align*}
The process
$\sE_{t} := \exp*{ \ld \tN_{t} - \ld^{2}\tA_{t}/2 }$ is again a positive
martingale, now for the filtration $\{ \sV^B_t \}_{t \in \sp}$
defined as $\{\sV_{t} \vee \sF^{B}_t\}_{t\in\sp}$, where $\{ \sF^B_t \}_{t\in
\sp}$ is the completed right-continuous augmentation of the
natural filtration of $B$.
The process $\tA$ is continuous, strictly increasing and unbounded,
so its inverse $\tT$ consists of finite stopping times and
$\tA_{\tT_{u}} = u$, $\tT_{\tA_{t}} = t$.

Stopped at $\tT_{v}$, the exponential
martingale $\sE$
 is bounded in $\Ltwo$ for small enough $\ld$. Indeed,
$\sE_t^2 = \sE'_t \exp*{\ld^{2}\tA_{t}}$, where
$\sE'_t = \exp*{ 2\ld \tN_{t} - 2\ld^{2}\tA_{t} }$ is
a positive supermartingale and
$\tA_{t \wedge \tT_{v}} \leq v$. Therefore,
$\bE^{\mu}[ \sE_{t \wedge \tT_{v}}^{2} ] \leq
\exp*{\ld^{2}v}$ and  we can use
the optional-sampling theorem at $\tT_{u} \leq \tT_{v}$ to conclude
that 
\begin{align*}
  \eemu{ e^{ \ld\prn{ W_{v} - W_{u} } - \tot\ld^{2}(v-u) }
  \giv \sV^B_{\tT_{u}} } = 1,
\end{align*}
where $W_{u} := \tN_{\tT_{u}}$.
The process $N$ is \cd{} and $B$ is continuous
so $\tN$ is \cd{}. Furthermore,
as $\tT$ is continuous and increasing, 
$W$ is right-continuous and the display therefore 
makes $W$ a Brownian motion, started at
$W_{0} = \tN_{0} = N_{0} = 0$. Since $\tT_{\tA_{t}} = t$, we have
$\tN_{t} = W_{\tA_{t}}$, and so $N_{t} = W_{\tA_{t}} - B_{t}$ is
continuous and $\bkt*{\tN}_{t} = \tA_{t}$. Finally, $B$ is independent
of $N$, so $\bkt{N,B} = 0$ and
$\bkt{ N }_{t} = \bkt*{\tN}_{t} - \bkt{B}_{t} = \tA_{t} - t =
\xi[0,t]$.

\subsection{Proof of Theorem \ref{t:mart-char}}
\label{s:prf-mart-char}

\emph{$(\ref{i:frm}) \Rightarrow (\ref{i:strong-form})$.} 
Let $\xi$ be a $(\rho,\mu)$-Feller random
measure on some probability space $(\Omega, \sG, \bP)$, and let 
$\prfis{\sH_{t}}$ be the usual augmentation of the filtration generated
by $\prfis{\Gamma_t}$, where $\Gamma_t = \xi[0,t]$. 
Let the process $\eta$ be defined by
\begin{align}
  \label{eta-def}
  \eta_t  = \lim_{n} n \prn{ (k*\xi)[t-1/n,t]
  - \textstyle\int_{t-1/n}^{t}\mu[0,s]\,\ds } ,
\end{align}
where $k*\xi$ and $\mu$ are extended trivially to $(-\infty,0)$, 
and $\eta := 0$ on the event where the limit in \eqref{eta-def}
fails to exist uniformly on compact sets. By Proposition
\ref{p:strong}, the measure $k*\xi$ is almost surely absolutely
continuous, so each approximant in \eqref{eta-def} is continuous
in $t$. Hence $\eta$ is a measurable functional of $\xi$ with
values in $C(\sp)$. Since
the restriction of $k*\xi$ to
$[0,t]$ is a function of $\xi|_{[0,t]}$, we have $\eta_t
\in \sH_t$ for all $t\in \sp$.

It remains to show that the limit in \eqref{eta-def} exists
uniformly on compact sets, that \eqref{char-strong} and
\eqref{bracket} hold, and that $\eta$ is an
$\prfis{\sH_{t}}$-martingale with $\eta_{0}=0$. By standard
approximation and separability arguments, these claims depend on
the law of $\xi$ alone. We may therefore assume for the rest of
the paragraph that $\xi$ is in occupation form. Then
\eqref{strong-k} reads $\prn{k*\xi}(\dt) =
\prn{N_{t}+\mu[0,t]}\,\dt$, so the $n$-th approximant in
\eqref{eta-def} equals $n\int_{(t-1/n)\vee 0}^{t}N_{s}\,\ds$.
Since $N$ is continuous with $\bkt{N}_{t} = \xi[0,t]$ by
Proposition \ref{p:noise-cont}, the convergence to $N_{t}$ is
uniform on compact sets and $\eta = N$ almost surely. Thus
\eqref{char-strong} and \eqref{bracket} follow. Proposition
\ref{p:noise-mg} makes $N$ a $\prfis{\sV_{t}}$-martingale,
with $N_{0}=0$ because $\ee{N_{0}^{2}} = (\rho*\mu)(\set{0}) = 0$ here.
As $N$ is $\prfis{\sH_{t}}$-adapted and $\sH_{t} \subseteq \sV_{t} \vee \sZ^{\mu}$ for
every $t$, it is an $\prfis{\sH_{t}}$-martingale as well.

\smallskip

\emph{$(\ref{i:strong-form}) \Rightarrow (\ref{i:weak-form})$.} Suppose that (\ref{i:strong-form}) holds, with a
filtration $\prfis{\sH_{t}}$ and a local martingale $\eta$ as described
there. We show that (\ref{i:weak-form}) holds with the same $\prfis{\sH_{t}}$ and
$\eta$. Since (\ref{i:weak-form}) repeats all the requirements of
(\ref{i:strong-form}) except for
replacing \eqref{char-strong} by \eqref{char-weak}, we only need
to prove \eqref{char-weak}. 
For $\vp \in \Conec(\sp)$, \eqref{adjoint} and Lemma
\ref{l:adjoint-k} give
$\ab{\xi,\sA\vp} = -\ab{\partial\prn{k*\xi},\vp} = \ab{k*\xi,\vp'}$,
which by \eqref{char-strong} equals
$\int_{\sp}\vp'(t)\prn{\eta_{t}+\mu[0,t]}\,\dt$. Standard
integration by parts, in both
deterministic and stochastic variants, yields
\eqref{char-weak}.
A general $\vp \in \sD(\sA) \cap \Cc(\sp)$ is approximated by forward
mollification: fix $j \in \Conec(\bR)$ with $\supp j \subseteq (-1,0)$
and $\int j = 1$, put $j_{\eps}(w) := \oo{\eps}\,j(\oo{\eps} w)$ and
$\vp_{\eps} := \vp * j_{\eps}$. Then $\vp_{\eps} \in \Conec(\sp)$, and,
since $\sA$ is a closed operator which commutes with the shift,
we have 
$\sA \vp_{\eps} = (\sA \vp)* j_{\eps}$. Moreover, by Lemma \ref{l:one-sided} both
$\vp_{\eps}$ and $\sA\vp_{\eps}$ are supported by $[0,\sup\supp\vp]$.
Since $\vp$ and $\sA\vp$ are uniformly continuous, $\vp_{\eps} \to \vp$
and $\sA\vp_{\eps} \to \sA\vp$ uniformly, so
$\ab{\mu,\vp_{\eps}} \to \ab{\mu,\vp}$,
$\ab{\xi,\sA\vp_{\eps}} \to \ab{\xi,\sA\vp}$, and
$\bkt*{ \int\prn{ \vp_{\eps}-\vp }\,\diffd\eta }_{\infty} =
\int_{\sp}\prn{ \vp_{\eps}-\vp }^{2}\,\dxi \to 0$ a.s.
Hence
$\int \vp_{\eps}\,\diffd\eta \to \int \vp\,\diffd\eta$ in
probability, and the three
convergences give \eqref{char-weak}.

\smallskip

\emph{$(\ref{i:weak-form}) \Rightarrow (\ref{i:frm})$.} We first
check that $\rho(\set{0})=0$. Assume, to the contrary, that
$\rho(\set{0})>0$. Then, by Proposition \ref{p:no-atom}, $L$ is a
compound Poisson process, so that $b=0$ and $\Pi$ is finite. The
kernel $k$ therefore has a bounded and nonincreasing density
$1+\oPi$, and convolution with such a kernel turns a locally finite
measure into an absolutely continuous one whose density is of finite
variation (split $1+\oPi(u) = (1+c)-\Pi(0,u]$, where
$c := \Pi(0,\infty)$; the two resulting terms are both nondecreasing
in $t$). The integration by parts of the previous paragraph, read
backwards, turns \eqref{char-weak} into \eqref{char-strong}, so, by
the monotonicity of $\mu[0,\cdot]$, the process $\eta$ is of finite
variation. Since it is a continuous local martingale with
$\eta_{0}=0$, it vanishes identically. Hence $\xi=0$ by
\eqref{bracket}, and then $\mu=0$ by \eqref{char-strong},
contradicting $\mu \ne 0$.

It remains to show that $\xi$ is a $(\rho,\mu)$-Feller random
measure. The idea is to first apply \eqref{char-weak} to 
the solution of \eqref{riccati} with the sign of the quadratic term
reversed, which will ensure the finiteness of an exponential
moment of $\ab{\xi,f}$. This will, in turn, allow us to use
Novikov's condition and deduce \eqref{laplace} from \eqref{char-weak}
applied to $v[f]$.

We start by preparing the test
functions to be used later.
Let $f$ be a nonnegative function in $\Cc(\sp)$, supported by
$[0,T]$, with $a := \norm{f}_{\infty} < 1/2$, and let
\begin{align*}
  \B_{T,2a} := \sets{ \vp \in \B(\sp) }{ \vp = 0 \eon (T,\infty)
  \eand \norm{\vp}_{\infty} \leq 2a }.
\end{align*}
It is a complete metric space under $\norm{\cdot}_{\infty}$ with 
$\B_{T,2a} \cap \Cc(\sp)$ as a closed subset.
Using \eqref{conv} and Lemma
\ref{l:one-sided}, together with the fact that $a<1/2$,
one checks that, for each $\theta \in \set{-1,1}$, the map $\vp
\mapsto U\prn{ f + \tfrac{\theta}{2}\vp^{2} }$ maps
both $\B_{T,2a}$ and $\B_{T,2a} \cap \Cc(\sp)$ into
itself, with a $\norm{\cdot}_{\infty}$-Lipschitz constant at most
$2a<1$. It therefore admits a unique fixed point $u_{\theta}$ in
$\B_{T,2a}$, which lies in $\B_{T,2a} \cap \Cc(\sp)$.
Being the image of a
continuous compactly supported function under the resolvent operator
$U$, it is an element of $\sD(\sA) \cap \Cc(\sp)$ and satisfies
\begin{align}\label{fixed-pts}
  -\sA u_{\theta} = f + \tfrac{\theta}{2}\, u_{\theta}^{2} .
\end{align}
Moreover, by Corollary \ref{c:exists}, Definition \ref{d:frm} and Proposition
\ref{p:unique}, the solution $v[f]$ of \eqref{riccati} is bounded and
nonnegative, vanishes on $(T,\infty)$, and satisfies
$0 \leq v[f] = U\prn{ f - v[f]^{2}/2 } \leq U f \leq a$. Hence
$v[f] \in \B_{T,2a}$, and uniqueness of the fixed point for
$\theta = -1$ gives $u_{-1} = v[f]$ and, in particular,
$u_{-1} \geq 0$.

\smallskip

Relation \eqref{char-weak}, applied to $u_{\theta}$, gives
$\ab{\mu,u_{\theta}}
- \ab{\xi, f + \tfrac{\theta}{2} u_{\theta}^{2}}
+ \int u_{\theta}\,\diffd\eta = 0$, so that, by
\eqref{bracket},
\begin{align}
  \label{exp-id}
  \exp{ \theta \ab{\xi,f} } = \exp{ \theta\ab{\mu,u_{\theta}} }\,
  \exp{ \theta\int u_{\theta}\,\diffd\eta
  - \tot \bkt{ \int u_{\theta}\,\diffd\eta }_{\infty} } .
\end{align}
The last factor is the terminal value of the stochastic
exponential of $\theta \int u_{\theta}\,\diffd\eta$, which is a
nonnegative continuous local martingale, hence a supermartingale.

Take first $\theta = 1$ and, for $T \geq 1$, an $f$
  with $\ind{[0,T-1]} \leq 4f \leq \ind{[0,T]}$.
By \eqref{exp-id},
$\Exp{\xi[0,T-1]/4} \leq \Exp{\ab{\xi,f}}$ is then integrable, with
expectation at most $\Exp{\ab{\mu,u_{1}}}$. Since $T \geq 1$ is arbitrary, we get
$\ee{ \Exp{\xi[0,t]/4} } < \infty$ for all $t \in \sp$.

Next, pick $\theta = -1$ and a general nonnegative
$f \in \Cc(\sp)$ supported by $[0,T]$ with
$a = \norm{f}_{\infty} < 1/2$. Since $0 \leq u_{-1} \leq a$, we have
$\bkt*{ \int u_{-1}\,\diffd\eta }_{\infty}/2 \leq
a^{2}\,\xi[0,T]/2 \leq \tfrac{1}{8} \xi[0,T]$, whose exponential is
integrable by the
bound just established. Novikov's criterion therefore upgrades the
supermartingale in \eqref{exp-id} to a uniformly integrable martingale, and taking
expectations  gives
\begin{align}\label{laplace-id}
  \ee*{ e^{ -\ab{\xi,f} } } = e^{ -\ab{\mu,u_{-1}} } .
\end{align}
By Corollary \ref{c:exists} there exists a $(\rho,\mu)$-Feller
random measure $\xi'$, and \eqref{laplace}, together with the
identification $u_{-1} = v[f]$ above, turns the right-hand side of
\eqref{laplace-id} into $\ee*{ e^{-\ab{\xi',f}} }$. Applied to
$f = c\,\vp$ for a nonnegative $\vp \in \Cc(\sp)$ and
$c \in \prn{0,1/(2\norm{\vp}_{\infty})}$, this shows that the
Laplace transforms of $\ab{\xi,\vp}$ and $\ab{\xi',\vp}$ coincide in a
right neighborhood of $0$. 

A Laplace transform that is finite on an interval determines the
law of a nonnegative random variable, so $\ab{\xi,\vp}$ and
$\ab{\xi',\vp}$ have the same law for every nonnegative $\vp \in
\Cc(\sp)$. Hence $\xi$ and $\xi'$ have the same law and $\xi$ is
a $(\rho,\mu)$-Feller random measure.

\subsection{Proof of Theorem \ref{t:Z-eqs}}
\label{s:prf-Z-eqs}
The case $\mu = 0$ is trivial, since then $Z = 0$.
By absolute continuity, $\rho(\set{0}) = 0$, so Theorem
\ref{t:mart-char} guarantees the existence of a continuous local
martingale $\eta$ with $\bkt{\eta}_{t} = \xi[0,t] =
\int_{0}^{t}Z_{u}\,\du$, satisfying \eqref{char-strong}. We
choose $Z$ progressively measurable, e.g.,
$Z_{t} = \limsup_{n} n\,\xi(t-1/n,t]$. Since
its bracket is absolutely continuous,
\cite[Theorem 3.4.2]{KarShr91} provides a Brownian motion $B$, on
a possibly extended probability space, such that
$\eta_{t} = \int_{0}^{t}\sqrt{Z_{u}}\,\dB_{u}$. Therefore, if we
split $k$ at its atom in \eqref{sonin},
equation \eqref{char-strong} yields \eqref{Z-sde} for almost all
$t\in\sp$, a.s.

For \eqref{Z-resolvent}, we convolve \eqref{char-strong} with $\rho$ and
use the Sonin relation \eqref{sonin-rel}, $(k * \rho)(\dt) = \dt$, to
obtain
$\xi[0,t] = (\rho*\mu)[0,t] + \int_{0}^{t}\rho[0,t-s]\,\diffd\eta_{s}$,
the analogue of \eqref{volterra} with $\eta$ in place of $N$. Since
$\xi[0,t] = \int_{0}^{t}Z_{u}\,\du$ and $\rho[0,v] = \int_{0}^{v}r$,
differentiating in $t$, which \eqref{r-sq-int} and the stochastic
Fubini theorem \cite[Theorem 2.6]{Wal86} permit, yields
\eqref{Z-resolvent}, with the same Brownian motion $B$.

Lastly, we turn to \eqref{Z-potential}.
Every summand in \eqref{U-renewal} is absolutely continuous and,
therefore, so is the measure $U^{L}$.
Convolving \eqref{U-renewal} with $\rho$
gives $\rho * U^{L} = \sum_{n\geq2}\rho^{*n} = U^{L} - \rho$, that
is,
\begin{align}\label{r-ell}
  \int_{0}^{t} r(t-s)\,\ell(s)\,\ds = \ell(t) - r(t)
\end{align}
for almost every $t>0$. In particular $\ell = r + r*\ell$ lies in
$\Ltwo[0,T]$ for every $T$, by \eqref{r-sq-int}, the local
integrability of $\ell$, and Young's
inequality.

Next, for almost every $t \geq 0$, we integrate \eqref{Z-resolvent} against $\ell(t-\cdot)$
over $[0,t]$, so that the left-hand side becomes
$\int_{0}^{t}\ell(t-u)Z_{u}\,\du$. On the right, we use Fubini's
theorem for the first term and its stochastic version
\cite[Theorem 2.6]{Wal86} for the second one. By \eqref{r-ell},
we have
\begin{align*}
  \int_{s}^{t}\ell(t-u)\,r(u-s)\,\du
  = \int_{0}^{t-s}\ell(t-s-v)\,r(v)\,\diffd v
  = \ell(t-s) - r(t-s),
\end{align*}
so that 
\begin{align*}
  \int_{0}^{t}\ell(t-u)\,Z_{u}\,\du
  &= \int_{[0,t]}\prn{ \ell(t-s) - r(t-s) }\,\mu(\ds) \\
  &\quad + \int_{0}^{t}\prn{ \ell(t-s) - r(t-s) }
  \sqrt{Z_{s}}\,\dB_{s} .
\end{align*}
By \eqref{Z-resolvent}, the two terms with $r$ in them add up to
$-Z_{t}$, and moving them to the left-hand side gives \eqref{Z-potential}.

\subsection{Proof of Proposition \ref{p:density-crit}}
\label{s:prf-density-crit}
The proof relies on the following lemma:
\begin{lemma}
\label{l:conv-l1}
Given $\nu, \mu \in \Mlf^{*}(\sp)$, suppose that
$\nu$
admits a density $r$. With $g = r*\mu$ denoting the density of
$\nu* \mu$, we have
\begin{align}
  \label{conv-equiv}
r^2 \in \Lonel(\sp) \eiff\, r^2*g \in \Lonel(\sp). 
\end{align}
In that case, we also have $g \in \Ltwol(\sp)$.
\end{lemma}
\begin{proof}
By Young's convolution inequality and
$\norm{g}_{\Lone[0,T]} = (\nu*\mu)[0,T] \leq \nu[0,T]\,\mu[0,T]$, we
have
\begin{align}
  \label{conv-bound}
  \norm{ r^2*g}_{\Lone[0,T]}
  \leq \norm{r^2}_{\Lone[0,T]} \norm{g}_{\Lone[0,T]}
  \leq \nu[0,T]\,\mu[0,T]\, \norm{r^2}_{\Lone[0,T]} .
\end{align}
On the other hand,
$[0,T/2]^{2} \subseteq D$, where $D :=
\sets{(u,s) \in [0,T]^2}{ u+s \leq T}$,  so
\begin{align*}
  \norm{r^2*g}_{\Lone[0,T]}
  & =\int_{0}^{T} \int_{0}^{t}r^2(t-s)g(s)\,\ds\,\dt =
  \iint_D r^2(u) g(s)\, \du\,\ds\\
  & \geq
  \int_0^{T/2} r^2(u)\, \du
  \int_0^{T/2} g(s)\, \ds = \norm{r^2}_{\Lone[0,T/2]}
  \norm{g}_{\Lone[0,T/2]}.
\end{align*}
Since
$\lim_{T\to\infty} \norm{g}_{\Lone[0,T/2]} =
\int_0^{\infty} g(s)\, \ds = \nu(\sp)\, \mu(\sp) \in (0,\infty]$, there are a
constant $A>0$ and a $T_0 \geq 0$ with
$\norm{g}_{\Lone[0,T/2]} \geq A$ for $T \geq T_{0}$. Together with
\eqref{conv-bound} this gives
\begin{align*}
  A \norm{r^2}_{\Lone[0,T/2]} \leq \norm{r^2 * g}_{\Lone[0,T]} \leq
  \nu[0,T]\, \mu[0,T]\, \norm{r^2}_{\Lone[0,T]}  \eforall T
  \geq T_0,
\end{align*}
and \eqref{conv-equiv} follows. Finally, $g \in \Ltwol(\sp)$
follows from Minkowski's integral inequality
$\norm{g}_{\Ltwo[0,T]}
= \norm{r*\mu}_{\Ltwo[0,T]} \leq \norm{r}_{\Ltwo[0,T]}\, \mu[0,T]$.
\end{proof}

\emph{Proof of Proposition \ref{p:density-crit}.}
The existence of a density with
$\ee*{\int_{0}^{T}Z^{2}_u\, \du}<\infty$ depends on
$\xi$ only through its law, so we assume throughout that $\xi$ is
in occupation form. 

\smallskip

\emph{Sufficiency.}
By Lemma \ref{l:conv-l1}, \eqref{r-sq-int} implies that
$r^{2}*g \in \Lonel(\sp)$. In particular, $(r^{2}*g)(t)<\infty$
for almost every $t$. For such $t$, $s \mapsto r(t-s)\ind{[0,t]}(s)$ belongs
to $\Ltn$, and the inner stochastic integral in \eqref{dn-moment} below is
defined. The isometry of \eqref{noise}, the fact that
$\ee{\xi} = \rho * \mu$ from Proposition \ref{p:basic}(\ref{i:moments}), and
Tonelli's theorem now give
\begin{align}\label{dn-moment}
  \ee{\int_0^T\prn{\int_0^t r(t-s)\,\dN_{s}}^2\,\dt}
  =\int_0^T\int_0^t r(t-s)^2g(s)\,\ds\,\dt<\infty.
\end{align}
Consequently, for each $T \geq 0$, the restriction to $[0,T]$ of the
process $Z$ defined by \eqref{Z-rep} (which we take in its jointly
measurable version) is in
$\Ltwo(\Omega\times[0,T])$. The bound
$\int_{0}^{T}g^{2}(u)\, \du<\infty$ of
Lemma \ref{l:conv-l1} takes care of the
deterministic part. For
$\vp \in \Cc(\sp)$
with $\supp \vp \subseteq [0,T]$, we have
\begin{align*}
  \int_0^T \vp(t) Z_t\, \dt
&=\int_0^T \vp(t) \int_{[0,t]} r(t-s)\, \mu(\ds) \, \dt +
    \int_{\sp} U\vp(s)\,\dN_{s},
\end{align*}
where we used the stochastic Fubini theorem (\cite[Theorem 2.6]{Wal86})
for the second integral. Its hypothesis, namely
$\int_{0}^{T}\vp(t)^{2}\,(r^{2}*g)(t)\,\dt < \infty$, holds because
$\vp$ is bounded and $r^{2}*g \in \Lonel(\sp)$. Thus,
\eqref{weak-res} of Proposition \ref{p:weak} implies that 
\begin{align*}
  \ab{Z, \vp} = \ab{\rho*\mu, \vp} + \tint_{\sp} U\vp(s)\,\dN_{s} = \ab{\xi,
  \vp}, \eas \eforeach \vp \in \Cc(\sp).
\end{align*}
Both sides are continuous under uniform convergence of $\vp$
within a fixed compact support, because $\xi[0,T]<\infty$ and
$\int_{0}^{T}\abs{Z_t}\, \dt<\infty$, a.s. Passing to a countable family,
dense for the uniform norm in the set of $\vp \in \Cc(\sp)$ supported
by $[0,T]$, for each $T \in \N$, we find a single exceptional
set for all $\vp \in \Cc(\sp)$. Hence
$\xi(\dt) = Z_{t}\, \dt$, a.s.

\smallskip

\emph{Necessity.}
Suppose that $\xi(\dt)=Z_{t}\,\dt$ with
$\ee*{ \int_0^T Z_{t}^2\,\dt}<\infty$ for every $T$.
With $g(t) = (r*\mu)(t) = \ee{Z_t}$, a.e., 
Jensen's inequality implies that 
$\int_0^T g(t)^2\, \dt \leq \int_0^T \ee{Z_t^2}\, \dt < \infty$ so
that 
\begin{align*}
  \tZ := Z - g \in \Ltwo([0,T] \times \Omega) \eforall T>0 .
\end{align*}
Since $\xi$ is absolutely continuous and
$\ab{\mu,U\psi} = \ab{\rho*\mu,\psi} = \ab{g,\psi}$,
\eqref{weak-res} reads
\begin{align}\label{Z-weak}
  \ab*{\tZ, \psi} = \int_{\sp} U\psi(s)\,\dN_{s} \eas \eforeach \psi \in
  \Cc(\sp).
\end{align}
Both sides of \eqref{Z-weak} pass to the limit under bounded pointwise
convergence of $\psi$ within a fixed compact support --- on the right
because $U\psi_{n} \to U\psi$ boundedly and $\supp U\psi_{n}$ stays in
a fixed compact set by Lemma \ref{l:one-sided} --- so, by the
monotone class theorem, \eqref{Z-weak} holds for
every bounded Borel $\psi$ of bounded support. 

The isometry of \eqref{noise} and the fact that $\rho*\mu$ has
density $g$ turn
\eqref{Z-weak} into
\begin{align}
  \label{Z-iso}
  \tint_{\sp}\prn{ U\psi(s) }^{2}\, g(s)\,\ds
  = \ee*{ \ab*{ \tZ, \psi }^{2} } \eforall \psi \in \Bc(\sp).
\end{align}

For $T>0$, pick
$0 < \eps \leq 1$ and $a \in (0,T)$, and define
$\psi_{\eps} = \eps^{-1}\ind{(a,a+\eps)}$, so that
$U\psi_{\eps}(s) = \bar r_{\eps}(a,s) :=
\eps^{-1}\int_{a}^{a+\eps}r(w-s)\,\dw$ for $s<a$.
Jensen's inequality yields
$\ab*{\tZ,\psi_{\eps}}^{2} \leq
\eps^{-1}\int_{a}^{a+\eps}\tZ_{u}^{2}\, \du$. By \eqref{Z-iso}, after discarding the contribution of $s \geq a$ to
its left-hand side, we have
\begin{align*}
  \int_{0}^{a}\bar r_{\eps}(a,s)^{2}\,g(s)\,\ds
  \leq \oo{\eps}\, \ee{ \int_{a}^{a+\eps}
   \tZ_{u}^{2}\,\du } .
\end{align*}
Next, we integrate both sides over $a \in [0,T]$. The right-hand
side is bounded by $\int_0^{T+1} \ee*{\tZ_u^2}\,\du < \infty$,
uniformly over $\eps \leq 1$. 

For each $s$, almost every $a$ is a Lebesgue point of $r(\cdot-s)$, so
by Fubini's theorem the integrand converges to
$r(a-s)^{2}g(s)$ for almost every pair $(a,s)$. With
$D_{T} := \set{(a,s) : 0<s<a<T}$, Fatou's lemma implies
\begin{align*}
  \iint_{D_{T}} r^{2}(a-s)\,g(s)\,\ds\,\da
  &= \iint_{D_{T}} \liminf_{\eps \to 0}
     \bar r_{\eps}(a,s)^{2}\,g(s)\,\ds\,\da \\
  &\leq \liminf_{\eps \to 0} \iint_{D_{T}}
     \bar r_{\eps}(a,s)^{2}\,g(s)\,\ds\,\da
   \leq \int_{0}^{T+1}\ee*{\tZ_{t}^{2}}\,\dt < \infty .
\end{align*}
Hence $r^{2}*g \in \Lonel(\sp)$, and Lemma
\ref{l:conv-l1} turns it into \eqref{r-sq-int}. 

\smallskip 

It remains to establish \eqref{Z-rep}. Under \eqref{r-sq-int}, the
right-hand side of \eqref{Z-rep} is a density of $\xi$ by the
sufficiency part above. Since any two densities of $\xi$ agree
$\dt \otimes \dbP$-a.e., every density satisfies \eqref{Z-rep}.

\subsection{Proof of Proposition \ref{p:Z-modulus}}
\label{s:prf-Z-modulus}
Since $r,m \in \Ltwol(\sp)$,
Young's convolution inequality (see, e.g., \cite[Theorem 10.49]{Leo17}) implies that the density $g = r*m$ of
$\rho*\mu$ is bounded on each $[0,T]$.
Proposition
\ref{p:density-crit} provides a density $Z$ of $\xi$ with
$\ee*{ \int_{0}^{T}Z_{t}^{2}\,\dt } < \infty$ for each $T$.
In particular,
$Z \in \Ltwol(\sp)$, a.s., and Theorem \ref{t:Z-eqs}
shows that it satisfies \eqref{Z-resolvent}. Since
$\prn{r^{2}*g}(t) \leq \norm{g}_{\infty,T}\,\norm{r}^{2}_{\Ltwo[0,T]}$
is finite at every $t \leq T$, we may take $Z$ to be a jointly
measurable version of the right-hand side of \eqref{Z-resolvent}. This version is again
a density of $\xi$ and satisfies \eqref{Z-resolvent} at each fixed $t$,
a.s. We use it throughout. 

For $p \geq 2$,
the Burkholder-Davis-Gundy inequality applied to the martingale
$u \mapsto \int_{0}^{u}r(t-s)\sqrt{Z_{s}}\,\dB_{s}$, $u \in [0,t]$,
followed by Minkowski's integral inequality (valid since $p/2 \geq
1$), implies that, for every $t \leq T$, we have
\begin{align*}
  \norm{Z_t}_{\bL^p}^2 - 2g(t)^{2}
& \leq
  2 \norm{ Z_{t}-g(t) }_{\bL^{p}}^{2}
  \leq 2C_{p}^{2}
  \norm{ \tint_{0}^{t}r(t-s)^{2}Z_{s}\,\ds }_{\bL^{p/2}}
  \\  & \leq 2C_{p}^{2}\tint_{0}^{t}r(t-s)^{2}\, 
  \norm{ Z_{s} }_{\bL^{p/2}}\,\ds
      \leq
      2C_{p}^{2} \norm{r}^{2}_{\Ltwo[0,T]} \sup_{s\leq t}
  \norm{Z_s}_{\bL^{p/2}}.
\end{align*}
It follows that
$\sup_{t\leq T} \norm{Z_t}_{\bL^{p/2}}<\infty$ implies
$\sup_{t\leq T} \norm{Z_t}_{\bL^{p}}<\infty$. For the base
case $p=1$, since $\ee{\xi} = \rho*\mu$, we have $\ee{Z_s} = g(s)$
for a.e.\ $s$, so that, by Fubini's theorem, $\ee*{\int_{0}^{t}r(t-s)^{2}Z_{s}\,\ds} \leq
\norm{g}_{\infty,T}\norm{r}^{2}_{\Ltwo[0,T]} < \infty$.
The stochastic integral in \eqref{Z-resolvent} is then a true
martingale, and $\ee{Z_t} = g(t)$ for every $t \leq T$, with $g\in
\Linfl$.
Iterating along $p = 2^{n}$ and
interpolating, we conclude that
$\sup_{t\leq T} \norm{Z_t}_{\bL^{p}}<\infty$ for all $p \geq 1$.

To establish the second inequality in \eqref{Z-incr}, we let $0 \leq s < t \leq T$ and put $h := t-s$. Subtracting two
instances of \eqref{Z-resolvent}, we get
  $Z_{t}-Z_{s} = I_1 + I_2 + J_1  + J_2$, where
  \begin{equation}
    \label{decomp}
\begin{aligned}
  I_{1} &:= \int_{0}^{s}\bprn{ r(t-u)-r(s-u) } m(u)\,\du,
        &I_{2} &:= \int_{s}^{t}r(t-u) m(u)\,\du, \\
  J_{1} &:= \int_{0}^{s}\bprn{ r(t-u)-r(s-u) }
\sqrt{Z_{u}}\,\dB_{u},
        &J_{2} &:= \int_{s}^{t}r(t-u)
        \sqrt{Z_{u}}\,\dB_{u}.
\end{aligned}
\end{equation}
By the Cauchy-Schwarz inequality, the terms $I_1$ and $I_2$ are both bounded
by $\norm{m}_{\Ltwo[0,T]} \omega_T(h)$.
For $J_1$ and $J_2$,
the Burkholder-Davis-Gundy inequality, followed by Jensen's
inequality against the measures
$\bprn{ r(t-u)-r(s-u) }^{2}\,\du$ on $[0,s]$ and $r(t-u)^{2}\,\du$
on $[s,t]$, whose total masses are at most $\omega_{T}(h)^{2}$,
implies that
\begin{align*}
  \ee{ \abs{J_{j}}^{p} } \leq C_{T, p}\,\omega_{T}(h)^{p}\,
  \textstyle \sup_{u \leq T}\ee*{ Z_{u}^{p/2} }, \, j = 1,2,
\end{align*}
the supremum being finite by the first inequality of
\eqref{Z-incr}. Summing the four bounds proves the second inequality of
\eqref{Z-incr}.

When $\omega_{T}(h) \leq C_{T}h^{\gm}$ for $h \in (0,1]$, the
resulting bound on $\norm{Z_{t}-Z_{s}}_{\bL^{p}}$ persists for $h \in [1,T]$ as well, since the first inequality of
\eqref{Z-incr} gives $\norm{Z_{t}-Z_{s}}_{\bL^{p}} \leq 2C \leq
2C h^{\gm}$ there. The
Kolmogorov-Chentsov criterion produces a version of $Z$ on $[0,T]$
whose paths are H\"older continuous of every order
$\theta < \gm - 1/p$, for each $p$, and, therefore, for each
order $\theta < \gm$. This version agrees with $Z$ for a.e.\ $t$,
a.s., so it is still a density of $\xi$.

To show \eqref{Z-incr-lb}, we apply the decomposition
\eqref{decomp} above to the pair
$(t,t+h)$. The terms $I_{1}$ and $I_{2}$ are then deterministic, while
$J_{1}$ and $J_{2}$ are centered and, being stochastic integrals over
disjoint intervals, orthogonal, so that
\begin{align*}
  \norm{ Z_{t+h}-Z_{t} }_{\Ltwo}^{2} \geq
  \Var\prn{ Z_{t+h}-Z_{t} } = \ee{J_{1}^{2}} + \ee{J_{2}^{2}}
  \geq \ee{ J_{2}^{2} } .
\end{align*}
Moreover, 
\begin{align*}
  \ee{J_{2}^{2}} = \int_{t}^{t+h}
  r(t+h-u)^{2}g(u)\,\du
  = \int_{0}^{h} r(v)^{2}g(t+h-v)\,\diffd v \geq
  \norm{r}_{\Ltwo[0,h]}^2 \inf_{[t,t+h]}g.
\end{align*}

\subsection{Proof of Corollary \ref{c:holder}}
\label{s:prf-holder}
By \eqref{rho-psi}, the Fourier transform of $\rho$ is
$u \mapsto \hrho(-iu) = 1/\Psi(-iu)$. Since
$\abs{\hrho(-iu)} \leq = 1$ and $2\beta>1$, \eqref{psi-power}
gives
$\int_{\bR}\abs{ \Psi(-iu) }^{-2}\,\du \leq 2\ld_{0} +
c^{-2}\int_{\abs{u}>\ld_{0}}\abs{u}^{-2\beta}\,\du < \infty$. By
Plancherel's theorem, $\rho$ therefore admits a density
$r \in \Ltwo(\bR)$, which we take to vanish on $(-\infty,0)$, so that
\begin{align*}
  \hrho(-i u) = \oo{\Psi(-i u)} = \int_{\bR} e^{i u t}
  r(t)\,\dt \efor u \in \bR.
\end{align*}

Since $r$ vanishes on $(-h,0)$, 
we have 
$\int_{-h}^{0}\prn{ r(u+h)-r(u) }^{2}\,\du = \norm{r}^{2}_{\Ltwo[0,h]}$.
As $(-h,0)$ and $[0,T]$ are disjoint, both terms of $\omega_{T}$ are
dominated at once: for $h>0$,
\begin{align}
  \label{modulus-fourier}
  \omega_{T}(h)^{2} \leq
  \int_{\bR}\prn{ r(u+h)-r(u) }^{2}\,\du =
  \frac{1}{2\pi}\int_{\bR}\abs*{ e^{ - i u h}-1 }^{2}
  \abs{ \hrho( - i u) }^{2}\,\du,
\end{align}
where we used Plancherel's theorem for the equality. The bound in \eqref{modulus-fourier} does not depend on
$T$. 

It suffices to treat $h < 1/\ld_{0}$. Indeed, \eqref{modulus-fourier} bounds
$\omega_{T}(h)^{2}$ by $4\norm{r}^{2}_{\Ltwo(\bR)}$ for every $h$,
while $h^{2\beta-1} \geq \ld_{0}^{1-2\beta}$ for $h \geq 1/\ld_{0}$, so
that those $h$ are absorbed into the constant.

In order to use the standard estimate
$\abs*{ e^{ - i u h}-1 } \leq \min(2, \abs{u} h)$, we split
  the integral on the right-hand side of \eqref{modulus-fourier}
  at $\abs{u} = 1/h$ and at $\abs{u} = \ld_{0} < 1/h$.
  Using \eqref{psi-power}, we get
\begin{align*}
  \int_{\abs{ u}>1/h}\abs*{ e^{-iuh}-1 }^{2}
  \abs{ \hrho(-iu) }^{2}\,\du
  & \leq C_c \int_{\abs{ u}>1/h}\abs{ u}^{-2\beta}\,\du
  = C_{c, \beta} h^{2\beta-1},  \\
  \int_{\ld_{0}<\abs{ u}\leq 1/h}\abs*{ e^{-iuh}-1 }^{2}
  \abs{ \hrho(-iu) }^{2}\,\du
  & \leq
  C_c \int_{\ld_{0}<\abs{ u}\leq 1/h}h^2 \abs{ u}^{2-2\beta}\,\du
  \\
  & \leq C_c \frac{2h^{2}}{3-2\beta}\,h^{2\beta-3}
  = C_{c, \beta} h^{2\beta-1}, \eand  \\
  \int_{\abs{u}\leq \ld_0}\abs*{ e^{-iuh}-1 }^{2}
  \abs{ \hrho(-iu) }^{2}\,\du &\leq 2 h^{2}\ld_{0}^{3}  =
  C_{\ld_0} h^2.
\end{align*}
Since $3-2\beta>0$ and $h<1/\ld_{0}$, we have
$h^{2} = h^{3-2\beta}h^{2\beta-1} \leq \ld_{0}^{2\beta-3}h^{2\beta-1}$,
so the third bound is absorbed into the first two. We conclude that
$\omega_{T}(h)^{2} \leq C\,h^{2\beta-1}$.

\subsection{Proof of Proposition \ref{p:max-holder}}
\label{s:prf-max-holder}
We define the essential H\"older seminorm for $\gm>0$ and
an interval $J$ by
\begin{align*}
  \abs{Z}_{\gm,J} = \esssup_{(u,v) \in J^{2},\ u \neq v}
  \Phi_{\gm}(u,v), \ewhere
  \Phi_{\gm}(u,v) := \frac{\abs{ Z_{u}-Z_{v} }}{\abs{u-v}^{\gm}},
\end{align*}
with the essential supremum being taken with respect to
two-dimensional Lebesgue measure. 
It is readily seen to be a random variable whose a.s.-equivalence
class does not depend on the version of $Z$.

The function $\mu[0,\cdot]$ is monotone and almost all paths of
$Z$ are locally integrable, so, by Lebesgue's differentiation
theorem and Fubini's theorem,
there exists a set $D \subseteq \sp$ of
full Lebesgue measure such that each $t\in D$ is a point of
differentiability of $\mu[0,\cdot]$, a Lebesgue point of $Z$ a.s.,
and a point at which \eqref{gen-frac-sde} holds a.s. We may also
assume that $0 \notin D$.

Fix $t \in D$, and suppose that there exist an $A \in \sF$ with
$\pp{A}>0$
and $c,\eps>0$ such that $Z_t \geq c$ and
$\abs{Z}_{\al-1/2,(t-\eps,t+\eps)} < \infty$, on $A$.
Since $t>0$, we may choose a smaller $\eps$ if necessary,
and assume that $\eps \leq \min(t,1)$.
The standard extension argument for
uniformly continuous functions implies that
$Z$ admits a $\Leb \otimes \bP$-version
such that $Z(\omega)$ is $(\al-1/2)$-H\"older continuous on
$(t-\eps, t+\eps)$, for each $\omega \in A$ (we do not
distinguish
a.s.-equal sets notationally).
Once we know
such a version exists, an explicit one is given by
\begin{align}
  \label{leb-version}
  \tZ_r = \liminf_{h \downto 0} \oo{h} \int_r^{r+h} Z_u\, \du
  \efor r \in \sp.
\end{align}
Since $t$ is a Lebesgue point of $Z$, a.s., we have $\tZ_t = Z_t$,
so that $\tZ_t\geq c$ on $A$.

Next, we show that $(\oPi*\tZ)(t+h)-(\oPi*\tZ)(t)=O(h^{1/2})$ as
$h\downto0$.
Fix $\omega \in A$ and $0<h\leq\eps/2$, and split the increment
of $\oPi*\tZ$ at $t$:
\begin{multline*}
  (\oPi*\tZ)(t+h)-(\oPi*\tZ)(t)
   = \int_{0}^{t}\bprn{ \oPi(t+h-s)-\oPi(t-s) }\tZ_{s}\,\ds
  + \int_{t}^{t+h} \!\!\!\!\!\oPi(t+h-s)\tZ_{s}\,\ds .
\end{multline*}
For $u \geq \eps/2$ we have $0 \leq \oPi(u)-\oPi(u+h) \leq h\sup_{[\eps/2,\infty)}\pi$, which is finite by
\eqref{pi-bound}. Hence, the part of the first integral over $[0,t-\eps/2]$ is
$O(h)\int_{0}^{t}\tZ_{s}\,\ds$. On the remaining range
$[t-\eps/2,t+h]$ we write $\tZ_{s} = \tZ_{t}+(\tZ_{s}-\tZ_{t})$. The
coefficient of $\tZ_{t}$ is
\begin{align*}
  \int_{0}^{\eps/2}\bprn{ \oPi(u+h)-\oPi(u) }\,\du
  + \int_{0}^{h}\oPi(u)\,\du
  = \int_{\eps/2}^{\eps/2+h}\oPi(u)\,\du
  = O(h).
\end{align*}
In the remainder we may use
$\abs*{\tZ_{s}-\tZ_{t}} \leq
\abs{Z}_{\gm,(t-\eps,t+\eps)}\,\abs{s-t}^{\gm}$, where
$\gm := \al-1/2$. By \eqref{pi-bound}, $\oPi(u) \leq C'u^{-\al}$ for
$u \in (0,1]$, so that $\int_{0}^{h}\oPi \leq C''h^{1-\al}$. Hence, over $[t,t+h]$ the remainder contributes
$O\bprn{h^{\gm}h^{1-\al}}$, and over $[t-\eps/2,t]$, after splitting
at $u=h$ and using $\oPi(u)-\oPi(u+h) \leq C h\,u^{-\al-1}$ for
$u>h$, it contributes $O\bprn{h^{1-\al+\gm}}$ as well. Since
$1-\al+\gm=1/2$, we conclude that
$(\oPi*\tZ)(t+h)-(\oPi*\tZ)(t) = O\bprn{h^{1/2}}$ for
$0<h\leq\eps/2$, with a random constant, on $A$. 
This convolution stays in its indistinguishability class
if $\tZ$ is replaced by $Z$, and becomes the left-hand side of
\eqref{gen-frac-sde}.

It follows that  the same is true of the right-hand side of
\eqref{gen-frac-sde}, on $A$. Indeed, \eqref{gen-frac-sde}
holds at $t$, by the choice of $D$, and at $t+h$ for a.e.\ $h$, by
Theorem \ref{t:Z-eqs}, a.s.
Its right-hand side is right-continuous in $h$, so the bound extends
to every $h \in (0,\eps/2]$. In particular, there
exist a constant $C>0$ and $B \subseteq A$ with $\pp{B}>0$ such that
$\abs*{ \mu(t,t+h] - \int_t^{t+h} Z_s\,\ds + \int_t^{t+h}
\sqrt{Z_s}\, \dB_s} \leq C h^{1/2}$
for $h \leq \eps/2$, on $B$. We chose $t$ as a point of
differentiability of $\mu[0,\cdot]$, so that $\mu(t,t+h] = O(h)$.
Moreover,
$\tZ$ is bounded on $[t,t+\eps/2]$ on $A$, so
$\int_{t}^{t+h}Z_{s}\,\ds = O(h)$ as well, but with a random
constant. Increasing $C$ by a deterministic amount absorbs the first
term, and choosing $\delta$ small enough absorbs the second, so that
\begin{align}
  \label{sqrt-bound}
  \tZ_{t+h} \geq c/2 \eand
  \abs*{\tint_t^{t+h} \sqrt{Z_s}\, \dB_s} \leq C
  h^{1/2} \efor h\in[0, \delta], \eon B,
\end{align}
where $\delta$ is a strictly positive random variable.
By the theorem of Dambis, Dubins and  Schwarz (see \cite[Theorem
4.6]{KarShr91}), $h \mapsto \int_{t}^{t+h} \sqrt{Z_s}\, \dB_s$
is, on a possibly enlarged probability space, a time-changed Brownian motion $\beta_{Q_h}$, where $Q_h =
\int_{t}^{t+h} Z_s\, \ds$. Since
$\limsup_{q\downto 0} \abs{\beta_q}/\sqrt{q} = + \infty$ by the law
of the iterated logarithm \cite[Theorem 9.23]{KarShr91}
and
$\tZ_s \geq c/2$ for $s \in [t,t+\delta]$, we
conclude that
\begin{align*}
  \limsup_{h \downto 0} \oo{\sqrt{h}}
  \abs{\int_{t}^{t+h} \sqrt{Z_s}\, \dB_s} = +\infty
  \eon B,
\end{align*}
which contradicts \eqref{sqrt-bound}.

It remains to pass to a single exceptional set for all $t\in D$.
We fix an interval $J$ with rational endpoints.
If $(t-\eps,t+\eps) \subseteq J$, then
$\abs{Z}_{\al-1/2,(t-\eps,t+\eps)} \leq \abs{Z}_{\al-1/2,J}$, so the
preceding argument gives
$\pp*{ \set{Z_{t}>0} \cap \{\abs{Z}_{\al-1/2,J}<\infty\} } = 0$ for
almost every $t \in J$. By Tonelli's theorem,
\begin{align*}
  \ee*{ \ind{\set{\abs{Z}_{\al-1/2,J}<\infty}}\,
\Leb\sets{ t \in J}{Z_{t}>0 } } = \tint_{J}
  \pp*{ \set{Z_{t}>0} \cap \{\abs{Z}_{\al-1/2,J}<\infty\} }\,\dt
  = 0,
\end{align*}
so that, a.s., $\abs{Z}_{\al-1/2,J}<\infty$ implies $Z = 0$
Lebesgue-a.e.\ on $J$. Since there are countably many intervals with
rational endpoints, this holds a.s.\ simultaneously for all of them.
To deduce the statement, let $J$ be any interval on which $Z$ does
not vanish a.e., and suppose that $f$ is locally H\"older continuous of
order $\al-1/2$ on $J$ and equal to $Z$ Lebesgue-a.e.\ on $J$. Since
$J$ is the union of countably many intervals with rational endpoints
and closures contained in $J$, $Z$ does not vanish a.e.\ on one of
them, say $J'$. Since $f$ is H\"older continuous on the closure of
$J'$, we have $\abs{Z}_{\al-1/2,J'}<\infty$. By the above, a.s., no
interval $J'$ with rational endpoints has both properties.

\subsection{Proof of Proposition \ref{p:not-semimart}}
\label{s:prf-not-semimart}
Suppose that some version of $Z$, denoted again by $Z$, is a
semimartingale. Its paths are then \cd{}, hence locally bounded.

By \eqref{pi-bound} and the monotonicity of $\oPi$, for each $T>0$ there
is a constant $C_{T}$ such that $\oPi(x) \leq C_{T}\,x^{-\al}$ for
$x \in (0,T]$. In particular, $\oPi \in \Lonel(\sp)$, and the function
$\Lambda(t) := \int_{0}^{t}\oPi(s)\,\ds$ is continuous,
nondecreasing and concave, with $\Lambda(0)=0$ and $\Lambda(h) \leq
C' h^{1-\al}$ for $h \in (0,1]$ and some constant $C'$. Since $Z$ is locally bounded, the
convolution $\oPi*Z$ is continuous, while the right-hand side of
\eqref{gen-frac-sde}, which we denote by $Y$, is \cd{}. By Theorem
\ref{t:Z-eqs}, $\oPi*Z = Y$ for a.e.\ $t$, a.s., and right-continuity
gives $\oPi*Z = Y$ for all $t$, a.s.

The process $Y$ is a semimartingale whose continuous local
martingale part $\int_{0}^{t}\sqrt{Z_{s}}\,\dB_{s}$ has quadratic
variation $\int_{0}^{t}Z_{s}\,\ds$. Hence, along partitions of
$[0,T]$ with mesh tending to $0$, the sums of squared increments of $Y$
converge in probability to a limit that is at least
$\int_{0}^{T}Z_{s}\,\ds$. Since $\mu \neq 0$, we have
$\ee*{\int_{0}^{T}Z_{s}\,\ds} = (\rho*\mu)[0,T] > 0$ for large enough
$T$. We fix such a $T$. Then $\int_{0}^{T}Z_{s}\,\ds>0$ with
positive probability.

We write $Z = Z_{0}+M+A$, where $M$ is a local martingale and $A$ is
of finite variation, with $M_{0}=A_{0}=0$. We also assume, without
loss of generality, that $M$ is locally bounded (see
\cite[Theorem III.25]{Pro04}).
We have
\begin{align*}
  \oPi*Z = Z_{0}\,\Lambda + \oPi*M + \oPi*A.
\end{align*}
The first term is continuous and monotone. The last one is continuous
by the argument above, and of finite variation, since $\oPi*A^{\pm}$
are nondecreasing, where $A = A^{+}-A^{-}$ with $A^{\pm}$
nondecreasing. Since 
the sums
of squared increments of a continuous function of finite variation
tend to $0$
along any partitions with mesh tending to $0$, 
it suffices to treat $X := \oPi*M$.

For fixed $t$, integration by parts and the fact that 
$\Lambda(0) = M_{0} = 0$ yield
\begin{align*}
  X(t) = \int_{0}^{t}\Lambda(t-u)\,\diffd M_{u}.
\end{align*}
Since $M$ is locally bounded, there exists a sequence 
$\tau_{n} \upto \infty$ of stopping times such that
$\ee{\bkt{M^{\tau_{n}}}_{T}}<\infty$. On $\set{\tau_{n} \geq T}$, the
paths of $\oPi*M^{\tau_{n}}$ and $X$ agree on $[0,T]$, and
$\pp{\tau_{n} \geq T} \to 1$. It is therefore enough to treat
$M^{\tau_{n}}$ in place of $M$, and we assume that
$\ee{\bkt{M}_{T}}<\infty$.

Let $t_{i} = ih$, $h = T/m$, for some $m\in\bN$, and let 
  $\phi_{i}(u) := \Lambda(t_{i}+h-u)-\Lambda(t_{i}-u)$. 
Since $\Lambda$ vanishes on $(-\infty,0)$, we have
  $X(t_{i}+h)-X(t_{i}) = \int_{0}^{t_{i}+h}\phi_{i}(u)\,\diffd
  M_{u}$, 
and the It\^o isometry yields
\begin{align*}
  \ee{ \tsum_{i}\bprn{ X(t_{i}+h)-X(t_{i}) }^{2} }
  = \ee{ \tint_{0}^{T}\Bprn{ \tsum_{i}\phi_{i}(u)^{2} }\,\diffd\bkt{M}_{u} }.
\end{align*}
Fix $u \in [0,T]$. Since $\phi_{i}(u) = 0$ when $t_{i}+h \leq u$,
only the indices $i$ with $t_{i}+h>u$ contribute to the sum above.
We number them
$k = 0,1,2,\dots$ in increasing order, so that the $k$-th one satisfies
$t_{i}+h-u = (k+\theta)h$, where $\theta \in (0,1]$ depends only on
$u$. The first two terms, $k=0$ and $k=1$, are bounded by the
subadditivity of the concave function $\Lambda$:
$\phi_{i}(u) \leq \Lambda(h) \leq C' h^{1-\al}$ for $h \leq 1$. For
the remaining terms, $k \geq 2$, concavity gives
$\Lambda(x+h)-\Lambda(x) \leq h\,\oPi(x)$ with $x = (k-1+\theta)h
\geq (k-1)h$, and the bound on $\oPi$ yields
\begin{align*}
  \phi_{i}(u) \leq h\,\oPi\bprn{ (k-1)h } \leq C_{T}\,h^{1-\al}(k-1)^{-\al}.
\end{align*}
Hence
\begin{align*}
  \tsum_{i}\phi_{i}(u)^{2} \leq \Bprn{ 2(C')^{2} + C_{T}^{2}\tsum_{j \geq
  1}j^{-2\al} }\,h^{2-2\al},
\end{align*}
where the series converges because $\al>1/2$. The expected sum of
squared increments of $X$ is therefore at most a constant multiple of
$h^{2-2\al}\,\ee{\bkt{M}_{T}}$, which tends to $0$ because $\al<1$.
It follows that the sums of squared
increments of $\oPi*Z$ along the partitions $t_{i}=ih$ tend to $0$ in
probability as $m\to\infty$. This contradicts the lower bound of
$\int_{0}^{T}Z_{s}\,\ds$ obtained above.

\section{Proofs for Section \ref{S:disc}}

\subsection{Proof of Theorem \ref{t:ig-rep}}
\label{s:prf-ig-rep}
The proof is quite long, but the underlying idea is
relatively simple:
the starting point is 
the fundamental exponential martingale of \eqref{expmg}, which, in
this case, takes the form $\Exp{-\kappa\Gamma_{t} +
\phi(\kappa)\Gamma^{p}_{t}}$. Informally, it implies that 
\begin{align}
  \label{inf-IG}
   \eemu{ e^{-\kappa\,\diffd\Gamma_{t}} \giv \sV_{t-} } =
  e^{-\phi(\kappa)\,\diffd\Gamma^{p}_{t}} 
\end{align}
where $\diffd \Gamma$ and $\diffd \Gamma^p$ are `infinitesimal
increments'.
Given the immediate past, each increment of $\Gamma$ 
has a conditional inverse-Gaussian distribution
with mean equal to the corresponding increment of its compensator
$\Gamma^{p}$. In other words, $\Gamma$ moves like an inverse-Gaussian
subordinator $S$ run at the speed $\diffd\Gamma^{p}$, i.e., 
$\Gamma_{t} = S_{\Gamma^{p}_{t}}$. 
This identity
forces the values of $S$ on
the levels attained by $\Gamma^{p}$. 
The remaining levels, which form the gaps corresponding to
jumps of $\Gamma^{p}$ together with an unbounded gap beyond
its terminal value, are then filled by independent
inverse-Gaussian bridges. They fit exactly because, thanks to
\eqref{inf-IG} again, 
the jump of $\Gamma$ at a jump time of $\Gamma^{p}$ is conditionally
$\IG\prn{\Delta\Gamma^{p},p^{2}}$-distributed. 

\smallskip

We divide the proof into several steps and use the 
following notation:
\begin{align}
  \label{QR-not}
   Q_{t} := \Exp{ -\kappa\,\Gamma_{t} + \phi(\kappa)\,\Gamma^{p}_{t} },
  \quad
  \Theta_{x} := \inf\sets{t}{\Gamma^{p}_{t} \geq x},
  \quad
  R_{x} := \inf\sets{t}{\Gamma^{p}_{t} > x},
\end{align}
for $t, x \in \sp$ and $\kappa \geq 0$. Thus $\Theta$ and $R$ are the
left- and right-continuous inverses of $\Gamma^{p}$, and $\Theta_{x}
\leq R_{x}$. 
Together with the conventions $\Gamma_{0-} = \Gamma^{p}_{0-} = 0$ and
$\sV_{0-}$ trivial of subsection \ref{s:disc-noise}, we set $Q_{0-} := 1$, so that the
time $0$ is treated like a jump time for a jump of size $\Gamma_0$. 

\medskip

\subsubsection{Predictability of $\Gamma^p$}\label{s:pred-Gp} 
We first show that $\Gamma^{p}$ is predictable.
Splitting $\Pi$ in \eqref{Gp-def} into its diffuse part and its atoms gives
\begin{align}
  \label{Gp-split}
   \Gamma^{p}_t = p\,\mu[0,t] +
  p\int_0^t \Gamma_{t-s} \, \Pi^c(\ds) +
  p\sum_{s>0}
  \Pi(\set{s})\, \Gamma_{t-s}, 
\end{align}
where $\Pi^c$ is the diffuse part of $\Pi$ and the sum runs over
its atoms. The first summand is deterministic, the second continuous
and adapted, and $t \mapsto \Gamma_{t-s}$ is $\prfi{\sV_t}$-predictable for each
$s>0$. Hence $\Gamma^{p}$ is predictable. Moreover,
$\Gamma^{p}_{t} \leq p\,\mu[0,t] + p\,\Pi(0,\infty)\,\Gamma_{t} < \infty$
for all $t$, since $\Pi$ is finite and $\xi$ is locally finite.

\subsubsection{Parts \ref{i:comp} and \ref{i:qv} of 
Theorem \ref{t:ig-rep}}\label{s:part1}
Next we identify $\Gamma^{p}$ as the compensator of $\Gamma$ and compute $\ab{N}$.
 Proposition \ref{p:no-atom} and Remark \ref{r:pgid-ex} state that   $b=0$ and $\Pi$ is finite with
$1+\Pi(0,\infty) = 1/p$, so that $\sA\vp = -p^{-1}\vp +
\int_{(0,\infty)}\vp(\cdot+s)\,\Pi(\ds)$ and $\ab{\xi,
\int\vp(\cdot+s)\,\Pi(\ds)} = \ab{\Pi*\xi,\vp}$. Thus \eqref{weak-gen}
reads
\begin{align*}
   \ab{\xi,\vp} = p\,\ab{\mu + \Pi*\xi,\vp} + p\,N(\vp) \eas
  \eforeach \vp \in \Conec(\sp). 
\end{align*}
Approximating $\ind{[0,t]}$ by elements
of $\Conec(\sp)$ and using the Walsh isometry for the term
$N(\vp)$, 
we can turn this into \eqref{strong-disc},
i.e., 
$\Gamma_{t} = \Gamma^{p}_{t} + p\,N_{t}$ for all $t$, a.s.
By Proposition \ref{p:noise-mg}, $N = (\Gamma - \Gamma^{p})/p$ is then a
square-integrable martingale of finite variation.
Since $\Gamma^{p}$ is predictable (subsection \ref{s:pred-Gp}),
nondecreasing and integrable, and $\Gamma - \Gamma^{p}$ is a martingale,
$\Gamma^{p}$ is the compensator of $\Gamma$. This proves part \ref{i:comp}.

As a martingale of finite variation, $N$ is purely
discontinuous,  with the jump $\Delta N_{0} = N_{0}$ at $0$, by the
conventions of subsection \ref{s:notation}. 
The process $N^{2} - \Gamma$ is a martingale by Proposition
\ref{p:noise-mg}, and so is $\Gamma - \Gamma^{p}$. Hence so is
$N^{2} - \Gamma^{p}$, and $\ab{N} = \Gamma^{p}$ by predictability.
This proves part \ref{i:qv}.

\subsubsection{The martingale $Q$}\label{s:mart-Q}
Next, we show that the exponential process $Q$ is a martingale. 
It is straightforward to check that 
$\phi(\kappa) + \tfrac{p^{2}}{2}\phi(\kappa)^{2} = \kappa$,
and that $0 \leq p\phi(\kappa) \leq 1/4$ for $0 \leq \kappa \leq
\kappa_0$, where $\kappa_{0} := \oo{4p} + \tfrac{1}{32}$. We
assume from now on that $\kappa \in (0,\kappa_0]$. 
Together with the fact that
$p\,N = \Gamma - \Gamma^{p}$ and the general exponential-martingale
relation \eqref{expmg}, applied with $\ld = -p\,\phi(\kappa)$,
this implies that
\begin{align*}
Q_t = \Exp{-\kappa\Gamma_{t} +
\phi(\kappa)\Gamma^{p}_{t}} \text{ is a } \prfi{\sV_t}\text{-martingale.}
\end{align*}
We claim that, moreover, $\eemu{Q_0}=1$. To see that, 
note that $\Gamma^{p}_{0} = p\,\mu(\set{0})$, since
$(\Pi*\xi)(\set{0}) = 0$, and $\Gamma_{0} = \xi(\set{0})$, and apply
\eqref{laplace} with $f = \kappa\ind{\set{0}}$.
Using  the same function 
$f = \kappa\ind{\set{0}}$ in \eqref{riccati}, and noting that
$v(s)=0$ for $s>0$ in that case,  yields
$v(0) = p\prn{\kappa - v(0)^{2}/2}$,
whose nonnegative root is $v(0) = p\,\phi(\kappa)$. Therefore
$\eemu{e^{-\kappa\Gamma_{0}}} = e^{-p\,\phi(\kappa)\mu(\set{0})} =
e^{-\phi(\kappa)\Gamma^{p}_{0}}$, i.e., $\eemu{Q_{0}} = 1$.

Our next task is to show that $Q^{\Theta_y}$, where $\Theta_y$ is
defined in \eqref{QR-not}, is uniformly integrable for each
$y$. 
For $t<\Theta_{y}$ we have $\Gamma^{p}_{t} < y$, so
$Q^{\Theta_{y}}_{t} \leq e^{y\phi(\kappa)} + Q_{\Theta_{y}}
\inds{\Theta_{y}<\infty}$, where the last term is integrable by
Fatou's lemma, being the limit of $Q_{t\wedge\Theta_{y}}$, whose
mean is $1$. Thus, for each $y$,
$Q^{\Theta_{y}}$ is uniformly integrable, with
terminal value $Q_{\Theta_{y}}$. Here $Q_{\infty} :=
Q_{\infty-} := \lim_{t}Q_{t}$, which exists since $Q$ is a positive
martingale.

\subsubsection{Conditional distributions of predictable jumps}
  Since $Q$ is a
  martingale, for every predictable time
$\tau$ we have
\begin{align}\label{jump-law}
   \eemu{ e^{-\kappa\Delta\Gamma_{\tau}} \giv \sV_{\tau-} } =
  e^{-\phi(\kappa)\Delta\Gamma^{p}_{\tau}} \text{ on } \set{\tau<\infty},
  \eas, 
\end{align}
so that the conditional law of $\Delta\Gamma_{\tau}$ given
$\sV_{\tau-}$ is $\IG\prn{\Delta\Gamma^{p}_{\tau},p^{2}}$ there.
Indeed,
since $Q$ is a martingale and $\tau$ is predictable, we have
$\eemu{Q_{\tau} \giv \sV_{\tau-}} = Q_{\tau-}$ on $\set{\tau<\infty}$,
in the sense of generalized conditional expectations
\cite[Lemma I.2.27 and Corollary I.2.31]{JacShi03}. For $\tau = 0$,
where our convention on $\sV_{0-}$ differs from that of
\cite{JacShi03}, this is the identity $\eemu{Q_{0}} = 1$ established
above. The two cases can be treated separately because, $\sV_{0-}$
being trivial, the event $\set{\tau = 0}$ has probability $0$ or $1$
for every predictable time $\tau$. Dividing both sides by  $Q_{\tau-}>0$ and rearranging gives
\eqref{jump-law}.

The conditional law $\IG\prn{\Delta\Gamma^{p}_{\tau},p^{2}}$ obtained from
\eqref{jump-law} is the first assertion of part \ref{i:pred-jump} of Theorem \ref{t:ig-rep}.
Since $p\,\Delta N_{\tau} = \Delta\Gamma_{\tau} -
\Delta\Gamma^{p}_{\tau}$, $\Delta\Gamma^{p}_{\tau}$ is
$\sV_{\tau-}$-measurable, and $\IG(c,p^{2})$ does not charge $c$ for
$c>0$, while $\IG(0,p^{2}) = \delta_{0}$, we have $\Delta N_{\tau}
\neq 0$ a.s.\ on $\set{\Delta\Gamma^{p}_{\tau}>0}$ and $\Delta
N_{\tau} = 0$ a.s.\ on $\set{\Delta\Gamma^{p}_{\tau}=0}$. 
Since the jump times of $\Gamma^{p}$ are exhausted by predictable
times (see subsection \ref{s:constr-S} below), it follows that $N$ is
quasi-left-continuous if and only if $\Gamma^{p}$ is continuous, a.s.
This proves part \ref{i:pred-jump}.

\subsubsection{$\Gamma$ moves only when $\Gamma^{p}$ does}
\label{s:flat}
The following fact is folklore, but we could not find an exact
reference, so we give a short proof. We claim that
\begin{align}\label{flat-const}
   \Gamma^{p}_{s} = \Gamma^{p}_{t} \text{ implies } \Gamma_s = \Gamma_t
  \eforall s<t, \eas 
\end{align}
To see that, set 
$F := \bigcup_{q \in \mathbb{Q}_{+}}
\sets{(\omega,u)}{u>q,\ \Gamma^{p}_{u}(\omega) = \Gamma^{p}_{q}(\omega)}$
and observe that $F$ 
is a predictable set and  that 
$\int \ind{F}\,\diffd\Gamma^{p} = 0$. Since $\Gamma^{p}$ is the
compensator of $\Gamma$, we have 
$\eemu{\int\ind{F}\,\diffd\Gamma} =
\eemu{\int\ind{F}\,\diffd\Gamma^{p}} = 0$, so $\int\ind{F}\,\diffd\Gamma
= 0$ a.s. If $s<t$ and $\Gamma^{p}_{s} = \Gamma^{p}_{t}$, then 
each $u \in (s,t]$ lies in $F$ (via some rational $q \in (s,u)$),
and so \eqref{flat-const} holds. 
A similar argument yields
$\eemu{\Gamma_{0}} =
\Gamma^{p}_{0}$, so that $\Gamma_{0} = 0$ when $\Gamma^{p}_{0} = 0$.

\subsubsection{Construction of the process $S$}
\label{s:constr-S}
We build $S$ from $\Gamma$ on the levels attained by
$\Gamma^{p}$. The gaps left by the jumps of $\Gamma^{p}$ are
filled with inverse-Gaussian bridges, and an independent
inverse-Gaussian subordinator is appended above
$\Gamma^{p}_{\infty-}$.
A standard result on the jumps of predictable processes (see, e.g., \cite[Theorems 3.31 and
3.33]{HeWanYan92}) states that 
the jump times of a \cd{} predictable process 
are exhausted by a sequence of 
predictable times with disjoint graphs.
Let $\seq{\tau_n}$ denote such a sequence for 
$\Gamma^{p}$, and assume it includes $t=0$ when
$\Gamma^{p}_{0}>0$. 

Let $J$ be an inverse-Gaussian
subordinator, and, for $l \in (0,\infty]$, let $j \mapsto \Ld_{l,j}$
be a regular conditional law of $J$ given $J_{l} = j$, with
$\Ld_{\infty,\infty}$ denoting the law of $J$. Using the fact
that the inverse-Gaussian distribution $\IG(y,p^{2})$ admits a
density in $x$ which is jointly measurable in $x$ and in the
parameter $y$, we can choose $\Ld$ to be a probability kernel from
$(0,\infty]^{2}$ to $D_{\sp}$. For the remaining $(l,j)$, where no
conditional law is needed, we extend it arbitrarily, say by
$\Ld_{l,j} := \Ld_{\infty,\infty}$.
By the randomization lemma \cite[Lemma 4.22]{Kal21}, there
exists a measurable map $\Xi :
[0,\infty]^{2} \times [0,1] \to D_{\sp}$ such that $\Xi(l,j;\iota) \sim
\Ld_{l,j}$ whenever $\iota$ is uniform on $[0,1]$.
Enlarging the
probability space if necessary, we assume there exists an 
i.i.d.~sequence $\set{\iota_{n}}_{n \in \N\cup\set{\infty}}$ of random
variables uniform on $[0,1]$, independent of
$\sV_{\infty}$, and set
\begin{align*}
   G^{n} := \Xi\prn{ \Delta\Gamma^{p}_{\tau_{n}},
  \Delta\Gamma_{\tau_{n}}; \iota_{n} } \text{ on } \set{\tau_{n}<\infty}
  \eand
  G^{\infty} := \Xi\prn{\infty,\infty;\iota_{\infty}}. 
\end{align*}
Write $\sI := \sg\prn{\iota_{n} : n \in \N}$ and $\sI_{-n} :=
\sg\prn{\iota_{m} : m \in \N,\ m \neq n}$. By \eqref{jump-law} and the
choice of $\Ld$, conditionally on $\sV_{\tau_{n}-} \vee \sI_{-n}$, of which $\iota_{n}$ is
independent, the pair $\bprn{G^{n}|_{[0,l]},\Delta\Gamma_{\tau_{n}}}$,
with $l := \Delta\Gamma^{p}_{\tau_{n}}$, has the law of $\bprn{J|_{[0,l]},J_{l}}$,
while $G^{\infty}$ is an inverse-Gaussian subordinator independent of
$\sV_{\infty} \vee \sI$.

Let us call the level intervals $[\Gamma^{p}_{\tau_{n}-},
\Gamma^{p}_{\tau_{n}})$ the \emph{gaps} and $[\Gamma^{p}_{\infty-},
\infty)$ the \emph{tail}, where $\Gamma^{p}_{\infty-} := \lim_{t}
\Gamma^{p}_{t}$, and observe that these intervals are pairwise disjoint. Since
$\Gamma^{p}$ is the compensator of $\Gamma$, the limit
$\Gamma_{\infty-} := \lim_{t}\Gamma_{t}$ is finite a.s.\ on
$\set{\Gamma^{p}_{\infty-}<\infty}$ (see, e.g., \cite[Section I.3]{JacShi03}). For $x \in \sp$ we
define
\begin{align}\label{S-def}
   S_{x} :=
  \begin{cases}
    \Gamma_{\tau_{n}-} + G^{n}_{x-\Gamma^{p}_{\tau_{n}-}}, &
    x \in [\Gamma^{p}_{\tau_{n}-},\Gamma^{p}_{\tau_{n}}), \\
    \Gamma_{\infty-} + G^{\infty}_{x-\Gamma^{p}_{\infty-}}, &
    x \geq \Gamma^{p}_{\infty-}, \\
    \Gamma_{R_{x}}, & \text{otherwise.}
  \end{cases} 
\end{align}
where $R_x$ is defined in \eqref{QR-not}.

\subsubsection{Pathwise properties of $S$}
We check that $S$ is a nondecreasing \cd{} path and that $S_{\Gamma^{p}}$ recovers $\Gamma$.
We first prove that $S$ is nondecreasing and \cd{} with $S_0=0$, and that 
\begin{align}\label{tc-id}
   S_{\Gamma^{p}_{t}} = \Gamma_{t} \eforall t \in \sp. 
\end{align}
Let $x := \Gamma^{p}_{t}$. If $x$ lies in the $n$-th gap, then
$x = \Gamma^{p}_{\tau_{n}-}$ and $t<\tau_{n}$, since the levels in
$(\Gamma^{p}_{\tau_{n}-},\Gamma^{p}_{\tau_{n}})$ are not
attained.
It follows that 
$\Gamma^{p}$ is constant on $[t,\tau_{n})$ and that $S_{x} =
\Gamma_{\tau_{n}-} = \Gamma_{t}$ by \eqref{flat-const}. If $x$ lies in
the tail, then $x = \Gamma^{p}_{\infty-}$, $\Gamma^{p}$ is constant on
$[t,\infty)$, and, since $G^{\infty}_{0} = 0$, $S_{x} =
\Gamma_{\infty-} = \Gamma_{t}$ by \eqref{flat-const}. 
If $x$ lies
neither in a gap nor in the tail, we have
$\Gamma^{p}_{R_{x}} = x$, because otherwise $\Gamma^{p}$ would jump
across $x$ at $R_{x}$. Therefore, 
$t \leq R_{x}$, $\Gamma^{p}$ is constant on $[t,R_{x}]$, and
$S_{x} = \Gamma_{R_{x}} = \Gamma_{t}$ by \eqref{flat-const}.

Next, we establish that, almost surely, 
$S$ is nondecreasing within each piece of \eqref{S-def}.  $R$ is nondecreasing,
each bridge climbs from $\Gamma_{\tau_{n}-}$ to
$\Gamma_{\tau_{n}}$, and the tail increases from
$\Gamma_{\infty-}$. Let $x$ be a level in neither a gap nor the
tail. If
it is
below the $n$-th gap, then $R_{x}<\tau_{n}$ and, therefore,
$\Gamma_{R_{x}} \leq \Gamma_{\tau_{n}-}$. If it is 
above the $n$-th gap, we have $R_{x} \geq
\tau_{n}$ and so $\Gamma_{R_{x}} \geq \Gamma_{\tau_{n}}$.
Finally, $\Gamma_{R_{x}} \leq \Gamma_{\infty-}$, since $R_{x}<\infty$.
The bridge values in the $n$-th gap lie between
$\Gamma_{\tau_{n}-}$ and $\Gamma_{\tau_{n}}$, and those in the tail are
at least $\Gamma_{\infty-}$, so it follows that $S$ is nondecreasing.

The process $S$ is right-continuous inside gaps
and in the tail by construction. 
At any other
level $x$, we pick $t_{k} \downto R_{x}$ with $\Gamma^{p}_{t_{k}}
> x$.
Then $\Gamma^{p}_{t_{k}} \downto x$ and, by monotonicity and
\eqref{tc-id}, we have $S_{x+} = \lim_{k}\Gamma_{t_{k}} = \Gamma_{R_{x}} =
S_{x}$. Left limits exist since $S$ is nondecreasing. Finally, $S_{0} = 0$ because either 
$0$ lies in the gap opened at
$t=0$, or $S_{0} = \Gamma_{0} = 0$ by \eqref{tc-id} and 
subsection \ref{s:flat}.

\medskip

\subsubsection{The distribution of $S$}\label{s:dist-S}
It remains to show that $S$ is an inverse-Gaussian subordinator
and,  to this end, we prove that $e^{-\kappa S_{x}+\phi(\kappa)x}$ is a martingale in $x$.
For $x \in \sp$ define $
  \sG_{x} := \sV_{\Theta_{x}-} \vee \sg\prn{S_{w} : w \leq x}$
  and
\begin{align*}
   \hQ_{x} := e^{-\kappa S_{x} + \phi(\kappa)x}, 
   \eand
  W_{x} := Q_{\Theta_{x}} \inds{\Theta_{x}<\infty} +
  \hQ_{x}\inds{\Theta_{x} = \infty}. 
\end{align*}
Since $\Theta$ is nondecreasing, $\prfx{\sG_{x}}$ is a filtration.
Since $\Gamma^p$ is predictable, and hence so is $\Theta_{x}$,
the random variables $\Theta_{x}$, $\Gamma^{p}_{\Theta_{x}}$ and $\Gamma^{p}_{\Theta_{x}-}$ are
$\sV_{\Theta_{x}-}$-measurable, and so are all the events used below
to distinguish cases. 
\begin{lemma}
\label{l:W-id}
For $0 \leq x < y$ we have 
\begin{align}\label{W-id}
   \text{(i)}\ \ \eemu{ W_{x} \giv \sG_{x} } = \hQ_{x}
  \eand
  \text{(ii)}\ \ \eemu{ W_{y} \giv \sG_{x} } = \eemu{ W_{x} \giv
  \sG_{x} }. 
\end{align}
\end{lemma}
\begin{proof} $(i)$\ 
  If $\Theta_{x} = \infty$, then
$W_{x} = \hQ_{x}$. If $\Theta_{x}<\infty$ and
$\Gamma^{p}_{\Theta_{x}} = x$, then $S_{x} = \Gamma_{\Theta_{x}}$ by
\eqref{tc-id}, so again $W_{x} = \hQ_{x}$. In the remaining case
$\Gamma^{p}_{\Theta_{x}-} \leq x < \Gamma^{p}_{\Theta_{x}}$, so
$\Theta_{x} = \tau_{n}$ for some $n$ and $x$ lies in the $n$-th gap.
The event $\set{\Theta_{x} = \tau_{n}} =
\set{\Gamma^{p}_{\tau_{n}-} \leq x \leq \Gamma^{p}_{\tau_{n}}}
\setminus \set{\Theta_{x} < \tau_{n}}$ belongs to $\sV_{\tau_{n}-}$.
With $u := x - \Gamma^{p}_{\tau_{n}-}$, $l :=
\Delta\Gamma^{p}_{\tau_{n}}$ and $B_{v} := e^{-\kappa G^{n}_{v} +
\phi(\kappa) v}$, \eqref{S-def} gives
\begin{align*}
   \hQ_{x} = Q_{\tau_{n}-}\,B_{u} \eand W_{x} = Q_{\tau_{n}} =
  Q_{\tau_{n}-}\,B_{l}. 
\end{align*}
Conditionally on $\sV_{\tau_{n}-} \vee \sI_{-n}$, the process $B$
is a bounded martingale on $[0,l]$ with respect to the filtration $\sF^{n}_{v} :=
\sV_{\tau_{n}-} \vee \sI_{-n} \vee \sg\prn{G^{n}_{w} : w \leq
v}$. Moreover, the random time 
$u$ is $\sF^{n}_{0}$-measurable, so $\eemu{B_{l} \giv \sF^{n}_{u}} =
B_{u}$. 

It remains to note that $A \cap \set{\Theta_{x} = \tau_{n}}
\in \sF^{n}_{u}$ for every $A \in \sG_{x}$. Indeed, on that event,
$\sV_{\Theta_{x}-}$ and $\sV_{\tau_{n}-}$ have the same trace, and the  
values $S_{w}$, $w \leq x$, are values of $\Gamma$ at times before
$\tau_{n}$, bridges $G^{m}$ of earlier gaps (functions of
$\sV_{\tau_{m}}$-measurable jumps, $\tau_{m}<\tau_{n}$, and of
$\iota_{m}$, $m \neq n$), or values of $G^{n}$ on $[0,u]$.

\smallskip

$(ii)$\ On $\set{\Theta_{x} = \infty}$
both levels lie in the tail, and $W_{y}/W_{x} =
e^{-\kappa(G^{\infty}_{v}-G^{\infty}_{u}) + \phi(\kappa)(v-u)}$ with
$u := x - \Gamma^{p}_{\infty-}$ and $v := y - \Gamma^{p}_{\infty-}$.
Since $G^{\infty}$ is an inverse-Gaussian subordinator independent of
$\sV_{\infty} \vee \sI$, this ratio has conditional mean $1$ given
$\sV_{\infty} \vee \sI \vee \sg\prn{G^{\infty}_{w} : w \leq u}$, whose
trace on $\set{\Theta_{x} = \infty}$ contains that of $\sG_{x}$.

On $\set{\Theta_{x}<\infty}$, we first note that $\eemu{W_{y}
\giv \sV_{\infty} \vee \sI} = Q_{\Theta_{y}}$. Indeed, on
$\set{\Theta_{y} = \infty}$ we have $W_{y} = Q_{\infty-}\,
e^{-\kappa G^{\infty}_{v} + \phi(\kappa)v}$, whose second factor has
conditional mean $1$ given $\sV_{\infty} \vee \sI$, because
$G^{\infty}$ is independent of $\sV_{\infty} \vee \sI$ and $v$ is
$\sV_{\infty}$-measurable. The $\sg$-algebra
$\sI$, being independent of $\sV_{\infty}$, can be adjoined to
$\prfi{\sV_{t}}$ without affecting the martingale property of $Q$ or the uniform
integrability of $Q^{\Theta_{y}}$ of subsection \ref{s:mart-Q}. Optional sampling
of the uniformly integrable martingale $Q^{\Theta_{y}}$ at
$\Theta_{x} \leq \Theta_{y}$ gives $\eemu{Q_{\Theta_{y}} \giv
\sV_{\Theta_{x}} \vee \sI} = Q_{\Theta_{x}} = W_{x}$. 

The trace
of $\sG_{x}$ on $\set{\Theta_{x}<\infty}$ is contained in that of
$\sV_{\Theta_{x}} \vee \sI$ (the values $S_{w}$, $w \leq x$, involve, by \eqref{tc-id},
only $\Gamma$ on $[0,\Theta_{x}]$ and bridges $G^{m}$ with $\tau_{m}
\leq \Theta_{x}$), so $(ii)$ follows. 
\end{proof}

By Lemma \ref{l:W-id} and the tower property,
\begin{align*}
  \seemu{ \hQ_{y} \mid \sG_{x} } =
  \eemu{ \eemu{W_{y} \giv \sG_{y}} \giv \sG_{x} } =
  \eemu{ W_{y} \giv \sG_{x} } = \eemu{ W_{x} \giv \sG_{x} } =
  \hQ_{x}, 
\end{align*}
i.e., $\eemu{e^{-\kappa(S_{y}-S_{x})} \giv \sG_{x}} =
e^{-(y-x)\phi(\kappa)}$ for all $\kappa \in (0,\kappa_{0}]$. Hence
$S_{y}-S_{x}$ is independent of $\sG_{x}$ with law
$\IG\prn{y-x,p^{2}}$, and $S$, being nondecreasing, \cd{} and null
at $0$, is an inverse-Gaussian subordinator with respect to
$\prfx{\sG_{x}}$.

\subsubsection{Conclusion}
Finally, we combine the previous steps to prove parts \ref{i:ig-tc} and \ref{i:indep-sub} of Theorem \ref{t:ig-rep}.
By \eqref{tc-id} and \eqref{strong-disc}, we have 
$p\,N_{t} = \Gamma_{t}
- \Gamma^{p}_{t} = S_{\Gamma^{p}_{t}} - \Gamma^{p}_{t} =
p\,\tS_{\Gamma^{p}_{t}}$. Together with \eqref{tc-id} itself,
this establishes part \ref{i:ig-tc} of Theorem \ref{t:ig-rep}. 

It remains to prove part \ref{i:indep-sub}, which says that no
information accrues while the compensator $\Gamma^{p}$ is constant. By subsection \ref{s:dist-S}, for
$x<x'<y$ the increment $S_{y}-S_{x'}$ is independent of $\sG_{x'}
\supseteq \sG_{x+}$ with law $\IG\prn{y-x',p^{2}}$. Letting $x' \downto
x$ and using the right-continuity of $S$, we see that $S_{y}-S_{x}$ is
independent of $\sG_{x+}$ with law $\IG\prn{y-x,p^{2}}$. Thus $S$ is an
inverse-Gaussian subordinator with respect to the right-continuous
filtration $\prfx{\sG_{x+}}$, and it has the strong Markov property at
the stopping times of that filtration.

Fix $s \in \sp$. Since $\Gamma^{p}$ is finite-valued and nondecreasing,
we have $\set{\Gamma^{p}_{s}<x} = \set{s<\Theta_{x}}$ for every $x$, so that,
for $B \in \sV_{s}$, the set $B \cap \set{\Gamma^{p}_{s}<x}$ is one of
the generators of $\sV_{\Theta_{x}-}$ and therefore lies in $\sG_{x}$.
Intersecting over $x' \downto x$ gives $B \cap \set{\Gamma^{p}_{s} \leq
x} \in \sG_{x+}$. With $B = \Omega$, this shows that $\Gamma^{p}_{s}$
is a $\prfx{\sG_{x+}}$-stopping time. For general $B$, it shows that
$\sV_{s}$ is contained in the $\sg$-algebra $\sG_{\Gamma^{p}_{s}+}$ of
events prior to that stopping time. Let $(\sH_{x})$ be the usual
augmentation of $\prfx{\sG_{x+}}$. Then $S$ is still an
inverse-Gaussian subordinator with respect to $(\sH_{x})$, each
$\Gamma^{p}_{s}$ is still an $(\sH_{x})$-stopping time, and $\sV_{s}
\subseteq \sG_{\Gamma^{p}_{s}+} \subseteq \sH_{\Gamma^{p}_{s}}$. This
proves part \ref{i:indep-sub}. Its last statement follows from the
strong Markov property of $S$ at $\Gamma^{p}_{s}$, which makes
$\set{ S_{\Gamma^{p}_{s}+w}-S_{\Gamma^{p}_{s}} }_{w \geq 0}$ an
inverse-Gaussian subordinator independent of $\sH_{\Gamma^{p}_{s}}
\supseteq \sV_{s}$. This completes the proof of Theorem
\ref{t:ig-rep}.

\subsection{Proof of Corollary \ref{c:atomic}}
\label{s:prf-atomic}
The inverse-Gaussian subordinator $S$
has zero drift, so
that the associated Lebesgue--Stieltjes 
measure 
$\diffd S = \sum_{y>0}\Delta S_{y}\,\delta_{y}$ is purely atomic. 
Since $\Gamma^{p}$ is nondecreasing and right-continuous, we have 
\begin{align*}
  \sets{y>0}{s < \Theta_{y} \leq t} =
  (\Gamma^{p}_{s},\Gamma^{p}_{t}]
 \efor s<t, \eand
  \sets{y>0}{\Theta_{y} = 0} = (0,\Gamma^{p}_{0}].
\end{align*}
By Theorem
\ref{t:ig-rep}, part \ref{i:ig-tc}, we have
   $\xi(s,t] = S_{\Gamma^{p}_{t}} - S_{\Gamma^{p}_{s}}
  = \diffd S\prn{ \sets{y}{ \Theta_{y} \in (s,t] } }$, and
similarly, $\xi(\set{0}) = \diffd S\prn{\sets{y}{\Theta_{y} =
0}}$. Hence \eqref{atoms} follows. 
If $\Gamma^{p}$ is strictly increasing on an interval, then, since $S$ is a.s.\ strictly increasing (its L\'evy
measure is infinite), we have 
$\xi(s,t] = S_{\Gamma^{p}_{t}} - S_{\Gamma^{p}_{s}} > 0$ for all
$s<t$ in that interval, and 
$\xi$ has an atom in every
subinterval.

Since it admits the representation 
$\rho = p\sum_{k \geq 0}(p\,\Pi)^{*k}$, 
$\rho$ is purely atomic when $\Pi$ is.
Moreover, its atoms are finite sums of atoms of $\Pi$, so
$\rho*\mu$ is carried by $\Sigma$. Since $\eemu{ \xi(\Sigma^{c}) } =
\prn{\rho*\mu}(\Sigma^{c}) = 0$ by Proposition
\ref{p:basic}(\ref{i:moments}), $\xi$ is carried by $\Sigma$ almost
surely.

\subsection{Proof of Corollary \ref{c:dichotomy}}
\label{s:prf-dichotomy}
Part (\ref{i:dich-cont}) is Proposition \ref{p:basic}(\ref{i:diffuse})
together with Corollary \ref{c:dds}. In part (\ref{i:three-procs}), $\xi$
is purely atomic by Corollary \ref{c:atomic}. Fix $t$ with $\mu[0,t] >
0$, so that $\Gamma^{p}_{t} \geq p\,\mu[0,t] > 0$. Since $S$ is a.s.\ strictly
increasing, $\Gamma_{t} = S_{\Gamma^{p}_{t}} > 0$, and so $\xi \neq 0$.
The three processes have the common compensator $\Gamma^{p}$, since
$\Gamma - \Gamma^{p} = p\,N$ and $\bkt{N} - \ab{N}$ are martingales and
$\ab{N} = \Gamma^{p}$.

Let $\seq{\tau_{n}}$ be predictable times exhausting the jumps of
$\Gamma^{p}$, as in subsection \ref{s:constr-S}. On $\set{\tau_{n} \leq
t}$, with $\tau := \tau_{n}$, given 
$\sV_{\tau-}$, $x := \Delta\Gamma_{\tau}$ has the diffuse law
$\IG\prn{c,p^{2}}$, $c := \Delta\Gamma^{p}_{\tau} > 0$, by part
\ref{i:pred-jump} of Theorem \ref{t:ig-rep}. The jumps $x$,
$p^{-2}(x-c)^{2}$ and $c$ of $\Gamma$, $\bkt{N}$ and $\ab{N}$ at $\tau$
are therefore a.s.\ pairwise distinct. On the complement of
$\bigcup_{n}\set{\tau_{n} \leq t}$, the process $\Gamma^{p}$ is
continuous on $[0,t]$. Since $S$ has jumps in $(0,\Gamma^{p}_{t}]$, the
process $\Gamma_{s} = S_{\Gamma^{p}_{s}}$ has a jump $x > 0$ at some $s
\leq t$, at which $\Gamma^{p}$ does not jump. The corresponding jumps are
$x$, $p^{-2}x^{2}$ and $0$, and $x \neq p^{2}$ a.s., since the L\'evy
measure of $S$ is diffuse.

\subsection{Proof of Theorem \ref{t:disc-eqs}}
\label{s:prf-disc-eqs}
We first reduce the statement to \eqref{disc-tc}. Let $\Gamma'$ and
$N'$ be as in the statement, put $\xi' := \diffd\Gamma'$, and define
${\Gamma'}^{p}_{t} := p\,(\mu + \Pi*\xi')[0,t]$ as in \eqref{Gp-def}.
Since $\rho = p\sum_{n\geq0}(p\,\Pi)^{*n}$, we have $\rho*(\delta_{0} -
p\,\Pi) = p\,\delta_{0}$. Therefore \eqref{volterra-disc}, which reads
$\xi' = \rho*(\mu + \diffd N')$, is equivalent to $\xi' - p\,\Pi*\xi' =
p\,(\mu + \diffd N')$, i.e., $-\sA^{\ast}\xi' = p^{-1}\xi' - \Pi*\xi' =
\mu + \diffd N'$, since $\sA\vp = -p^{-1}\vp + \int\vp(\cdot+s)\,\Pi(\ds)$
(subsection \ref{s:part1}). In terms of distribution functions, this is
$\Gamma' - p\,N' = {\Gamma'}^{p}$. Given this identity,
\eqref{ig-noise} becomes $\Gamma'_{t} = {\Gamma'}^{p}_{t} +
p\,\tS_{{\Gamma'}^{p}_{t}} = S_{{\Gamma'}^{p}_{t}}$, which is
\eqref{disc-tc}. Moreover, $N' = (\Gamma' - {\Gamma'}^{p})/p$ is
determined by $\Gamma'$. For $\Gamma$ and $N$ themselves,
\eqref{volterra-disc} is \eqref{volterra}. It therefore suffices to prove the statements
for \eqref{disc-tc}.

By \eqref{Gp-def}, \eqref{disc-tc}
is the identity $\Gamma_{t} = S_{\Gamma^{p}_{t}}$ of Theorem
\ref{t:ig-rep}, part \ref{i:ig-tc}, and the stopping-time property is
part \ref{i:indep-sub}.

Optional sampling of the martingale $S_{x}-x$ at $\sigma \wedge k$,
followed by monotone convergence, gives $\ee{S_{\sigma}} = \ee{\sigma}$
for every finite $(\sH_{x})$-stopping time $\sigma$. Recall also that
$p\,\Pi(0,\infty) = 1-p$, by subsection \ref{s:part1}. Let $\Gamma'$ be as in
the statement, fix $t$, and put $\sigma := {\Gamma'}^{p}_{t}$ and $a
:= p\,\mu[0,t]$. Since $\Gamma'$ is nondecreasing, we have $\sigma \leq
a + (1-p)\,\Gamma'_{t} = a + (1-p)\,S_{\sigma}$. If $p = 1$, then
$\sigma \leq a$. Suppose now that $p<1$, and let $\psi(\theta) :=
\log\ee{e^{\theta S_{1}}} = p^{-2}\bprn{1 - \sqrt{1-2p^{2}\theta}}$.
Since $\psi(0) = 0$ and $\psi'(0) = 1 < 1/(1-p)$, there is $\theta \in
(0,1/(2p^{2}))$ with $c := \theta/(1-p) - \psi(\theta) > 0$. We apply
Fatou's lemma to the martingale $e^{\theta S_{x} - \psi(\theta)x}$,
stopped at $\sigma \wedge k$. Together with the bound on $\sigma$,
this gives $\ee{e^{c\sigma}} \leq e^{\theta a/(1-p)}$. In either case
$\ee{\sigma} < \infty$. Hence $\ee{\Gamma'_{t}} =
\ee{{\Gamma'}^{p}_{t}} \leq p\,\mu[0,t] + (1-p)\,\ee{\Gamma'_{t}}$,
i.e., $\ee{\Gamma'_{t}} \leq \mu[0,t]$. The same bound holds for
$\Gamma$.

Now let $\sigma := \Gamma^{p}_{t} \vee {\Gamma'}^{p}_{t}$ and
$\sigma' := \Gamma^{p}_{t} \wedge {\Gamma'}^{p}_{t}$. These are
stopping times, and $\abs{\Gamma_{t} - \Gamma'_{t}} = S_{\sigma} -
S_{\sigma'}$. Since the terms $p\,\mu[0,t]$ in $\Gamma^{p}_{t}$ and
${\Gamma'}^{p}_{t}$ cancel, we obtain
\begin{align*}
  e(t) := \ee{\abs{\Gamma_{t} - \Gamma'_{t}}} = \ee{\sigma - \sigma'}
  = \ee{\abs{\Gamma^{p}_{t} - {\Gamma'}^{p}_{t}}} \leq
  p\textstyle\int_{(0,t]} e(t-s)\,\Pi(\ds).
\end{align*}
By the integrability bound, $e \leq 2\mu[0,T]$ on $[0,T]$. Hence
$e^{*}(T) := \sup_{t \leq T} e(t)$ is finite, and the display gives
$e^{*}(T) \leq (1-p)\,e^{*}(T)$. Thus $e =  0$. By
right-continuity, $\Gamma' = \Gamma$ a.s., and then $N' = N$
a.s., by the first paragraph.

\fin